\documentclass[3p]{elsarticle}

\usepackage{lineno,hyperref}
\usepackage{amsmath,amsthm}
\usepackage{amssymb}
\usepackage{color}
\usepackage{empheq}

\modulolinenumbers[5]

\journal{Journal of Differential Equations}

\def \n {\nu}

\def \ve{\varepsilon}
\def \be{\begin{equation}}
\def \ee{\end{equation}}

\newtheorem{theorem}{Theorem}
\newtheorem{remark}{Remark}
\newtheorem{corollary}{Corollary}
\newtheorem{definition}{Def}
\newtheorem{proposition}{Proposition}
\newtheorem{lemma}{Lemma}
\renewenvironment{proof}{{\it Proof:}}{}

\begin{document}

\begin{frontmatter}

\title{Sharp interface  limit  in a phase field model of cell motility}
\tnotetext[mytitlenote]{The work of LB and VR was partially supported by NSF grants DMS-1106666 and DMS-1405769.
	The work of  MP was partially supported by the NSF grant DMS-1106666.
	The authors are grateful to I. Aronson and F. Ziebert for useful discussions on  the phase field model of cell motility introduced in their paper. The authors wish to thank M. Mizuhara who assisted in the proof-reading of the manuscript.}

\author[pennstate]{Leonid Berlyand}\ead{berlyand@math.psu.edu}
\author[pennstate]{Mykhailo Potomkin}\ead{mup20@ucs.psu.edu}
\author[iltpe]{Volodymyr Rybalko}\ead{vrybalko@ilt.kharkov.ua}


\address[pennstate]{Department of Mathematics, The Pennsylvania State University, University Park, PA 16802, USA}
\address[iltpe]{Mathematical Division, B. Verkin Institute for Low Temperature, Physics and Engineering
	of National Academy of Sciences of Ukraine,
	47 Lenin Ave., 61103 Kharkiv, Ukraine}

\begin{abstract}
We consider a system of two coupled parabolic PDEs introduced in \cite{ZieSwaAra11} to model motility 
of eukaryotic cells. We study the asymptotic behavior of solutions 
in the limit of a small parameter related to the width of the interface
in phase field function (sharp interface limit). 
We formally derive an equation of motion of the interface, which is mean curvature motion with an additional nonlinear term. 
In a 1D model parabolic problem we rigorously justify the sharp interface limit. To this end, a special form of asymptotic expansion is introduced to reduce analysis to a single nonlinear PDE. 
Further stability analysis  reveals a qualitative change in the behavior of the system for small and large values of the coupling parameter. 
Using numerical simulations we also show discontinuities of the  interface velocity  and hysteresis. Also, in the 1D case we establish nontrivial traveling waves 
when the coupling parameter is large enough.
%
%
%
%
%
%
%
%
\end{abstract}


\end{frontmatter}



\section{Introduction}
\label{intro}

The problem of cell motility  has been a classical subject in biology for  several centuries. 
It  dates back to the celebrated discovery by  van Leeuwenhoek in the 17th century  who  
drastically improved the microscope to the extent that he was able   to  observe motion of single celled organisms that moved 
due to contraction and extension.  Three centuries later this problem  continues to  attract  the  attention  of  biologists,   
biophysicists  and, more recently,  applied mathematicians.  A comprehensive  review  of  the mathematical modeling of 
cell motility can be found in \cite{Mog09}.

This work is motivated by the  problem of  motility (crawling motion) of eukaryotic cells on substrates.  The network of  actin  (protein) filaments (which is  a part of the cytoskeleton  in such cells)   plays an important role in cell motility.   We are concerned with cell shape dynamics, caused by extension of the front of the cell due to polymerization of the actin filaments  and contraction of the back of the cell due to detachment of  these filaments.  Modeling of  this process in full generality is at present a formidable challenge because  several  important  biological   ingredients  (e.g., regulatory pathways \cite{Mog09}) are not yet  well understood.

In recent biophysical studies  several simplified  {\it phase field models}  of cell motility have been proposed. Simulations performed for these models demonstrated   good agreement  with experiments (e.g., \cite{ZieSwaAra11,ShaRapLev10} and references therein).  Recall that phase field models are typically used to describe the  evolution of an interface between two phases (e.g., solidification or viscous fingering). The key ingredient of such models is  an auxiliary  scalar field,  which takes two  different values in   domains  describing  the two phases (e.g., $1$ and $0$) with a diffuse interface of a small  width. An alternative approach to cell motility involving  free boundary problems is developed in \cite{KerPinAllBarMarMogThe08,RubJacMog05,BarLeeAllTheMog15,RecTru13,RecPutTru15}. 

We consider the coupled system of parabolic PDEs, which is a modified version of the model 
from  \cite{ZieSwaAra11} in the diffusive scaling  ($t\mapsto \ve^2 t$, $x \mapsto \ve x$):
\begin{equation}
	\frac{\partial \rho_\ve}{\partial t}=\Delta \rho_\ve
	-\frac{1}{\ve^2}W^{\prime}(\rho_\ve)
	-P_\ve\cdot \nabla \rho_\ve +\lambda_\ve(t)
	\text{ in }\
	\Omega,
	\label{eq1}
\end{equation}
\begin{equation}
	\frac{\partial P_\ve}{\partial t}=\ve\Delta P_\ve -\frac{1}{\ve}P_\ve
	-\beta \nabla \rho_\ve
	\qquad\text{in}\ \Omega,
	\label{eq2}
\end{equation}
where
\begin{equation} \label{lagrange}
	\lambda_\ve(t)=\frac{1}{|\Omega|}\int_\Omega\left(\frac{1}{\ve^2}W^\prime(\rho_\ve)
	+ P_\ve\cdot \nabla \rho_\ve \right)\, dx.
\end{equation}
The unknowns here are the scalar phase field function $\rho_\ve$ and the orientation vector 
$P_\ve$;  $\Omega$ is a bounded domain in $\mathbb{R}^2$, $\lambda_\ve$ is the Lagrange multiplier responsible for preservation of volume. 
We study  solutions of system \eqref{eq1}-\eqref{lagrange} in the sharp interface limit, 
when the parameter $\ve>0$ (which is, loosely speaking, the width of the interface) tends to zero. 

While system \eqref{eq1}-\eqref{lagrange} represents a modified version of the model from \cite{ZieSwaAra11}, the main features of the original model are preserved. 
The volume preservation constraint in  \cite{ZieSwaAra11} is imposed by introducing a penalization parameter into the double well potential,
its role in
\eqref{eq1} is recast by the (dynamic) Lagrange multiplier $\lambda_\ve$ given by  \eqref{lagrange}. Both ways of introducing volume preservation are equivalent in the sharp interface limit, see \cite{Alf2010,AlfAli2014,BraBre2011}. Also, for technical simplicity we dropped two terms in the original equation of the orientation field. One of them, responsible for a stronger 
damping  in the phase $\rho_\ve \sim 0$, can be added to \eqref{eq2} without any 
qualitative changes, while the second one, the so-called $\gamma$-term, leads to an enormous 
technical complication, even existence is very hard to prove. Ref. \cite{ZieSwaAra11} qualifies 
this term as a symmetry breaking mechanism, which is important for initiation of motion. 
 Our study, however,  reveals another mechanism for breaking of symmetry in \eqref{eq1}-\eqref{eq2}, emanated 
from asymmetry of the potential $W(\rho)$ (see  Subsection \ref{SubsecMainRes}). That is, 
the effect of $\gamma$-term is replaced, to some extent, by  asymmetry of the potential.       

Heuristically, system \eqref{eq1}-\eqref{lagrange} describes the motion of a  interface caused by the competition between mean curvature motion (due to stiffness of interface) and the push of the  orientation field on the interface curve.  The main issue is to determine  the influence of  this competition  on  the  qualitative behavior of the  sharp interface solution.  The parameter $\beta>0$  models this competition 
which is why it  plays a key role in the analysis of system \eqref{eq1}-\eqref{lagrange}. 


\subsection{Techniques}
Recall  the Allen-Cahn equation  which is at the core of system \eqref{eq1}-\eqref{lagrange},
\begin{equation} \label{AC}
	\frac{\partial \rho_\ve}{\partial t}=\Delta \rho_\ve
	-{\frac{1}{\ve^2}}W^\prime(\rho_\ve), 
\end{equation}
where $W^\prime(\rho)$ is the derivative of a double equal well potential $W(\rho)$. We suppose that 
\begin{equation}
	\label{condpoten}
	W(\,\cdot\,)\in C^3(\mathbb{R}),\ W(\rho)>0\ \text{when} \ \rho\not\in\{0,1\}, \ W(\rho)=W^\prime(\rho)=0\ \text{at}\  \{0,1\},\ W^{\prime\prime}(0)>0,\ W^{\prime\prime}(1)>0,
\end{equation}
e.g. $W(\rho)=\frac{1}{4}\rho^2(\rho-1)^2$.  
Equation \eqref{AC}  was introduced in \cite{AllCah97} to  model  the motion of the phase-antiphase boundary (interface) between two grains in a solid material.
Analysis of \eqref{AC} as $\ve \to 0$  leads to the asymptotic solution that takes values $\rho_\ve\sim 0$ and $\rho_\ve\sim 1$ in the domains  corresponding to  two phases separated by an interface of width of order $\ve$,  the so-called  sharp interface.  Furthermore, it was shown that this sharp interface  obeys mean curvature motion.   Recall that  in this motion the normal component of the velocity of  each point of the surface  is equal to the mean curvature  of the surface at this point. This  motion  has been  extensively studied  in the geometrical community (e.g., \cite{Ham82,Hui84,Gra87,Bra78} and references therein).  It also  received significant attention  in
PDE  literature.  Specifically \cite{CheGigGot91}  and   \cite{EvaSpr91}  established existence of global viscosity  solutions (weak solutions) for the  mean curvature flow. Mean curvature motion of the interface in the limit $\ve \to 0$  was formally derived in \cite{Fif88},\cite{RubSterKel89} and  then justified in \cite{EvaSonSou92}  by using the viscosity solutions techniques. The limit $\ve\to 0$ was also studied for a stochastically perturbed Allen-Cahn equation \eqref{AC}  in \cite{KohOttRezVan06,OttWebWes13}.

Solutions of the stationary Allen-Cahn equation with the volume constraint were studied in \cite{Mod86} by $\Gamma$-convergence  techniques applied to the stationary variational   problem corresponding  to \eqref{AC}. It was established that the $\Gamma$-limiting functional is the interface perimeter  (curve length in 2D  or surface area  in higher dimensions).  Subsequently in the work
\cite{RubSte92}   an evolutionary  reaction-diffusion  equation with double-well potential   and  nonlocal term that describes  the volume constraint was studied.  The following asymptotic formula   for evolution of the interface  $\Gamma$ in the form of volume preserving mean curvature flow was formally derived in \cite{RubSte92}:
\begin{equation} \label{rs_1992}
	V=\kappa -\frac{1}{|\Gamma(t)|}\int_{\Gamma(t)}\kappa\, ds,
\end{equation}
where $V$ stands for the normal velocity of $\Gamma(t)$ with respect to the  inward normal,  $\kappa$
denotes the curvature of $\Gamma(t)$, $|\Gamma(t)|$ is the  curve length.
Formula \eqref{rs_1992} was rigorously justified in the radially symmetric  case in \cite{BroSto97} and in the general case in \cite{CheHilLog10}.

Three main approaches to the study of asymptotic behavior (sharp interface limit) of solutions of phase field equations and systems have been developed. 

When a comparison principle for solutions applies,  a PDE approach based on viscosity solutions techniques  was successfully used in \cite{EvaSonSou92,BarLio03,Gol97, LioKimSle04} and  other works.
This approach can not be applied to  the system \eqref{eq1}-\eqref{lagrange}, because 
\begin{itemize}
	\item equations \eqref{eq1}-\eqref{eq2} are coupled through spatial gradients,
	\item equation \eqref{eq1} contains the nonlocal (volume preservation) term $\lambda_\ve$ given by \eqref{lagrange}.
\end{itemize}

Another technique  used in such problems  is $\Gamma$-convergence (\textcolor{black}{see \cite{Ser10,KohOttRezVan06} and references therein}). This technique also does not work for the system \eqref{eq1}-\eqref{lagrange}. The standard Allen-Cahn equation \eqref{AC} is a gradient flow  (in $L^2$ metric) with Ginzburg-Landau energy functional, which is why one can use the $\Gamma$-convergence approach. However, there is no energy functional such that problem \eqref{eq1}-\eqref{lagrange}   can be  written as a  gradient flow.

When none of the above elegant tools apply, one can use  direct  construction of an asymptotic expansion followed by its justification via energy bounds \cite{MotSha95}.  In Allen-Cahn type problems it typically requires  a number of  terms (e.g., at least  five in \cite{CheHilLog10}) in the expansion. In this work we use some ingredients of this technique.  We  construct an asymptotic  formula with only two terms: the leading one and the corrector (see e.g., \eqref{eq_form}). The main ingredient is an appropriate choice of these terms which allows for good energy bounds for the corrector.  Furthermore  this choice   leads to a reduction of the  coupled system to a single singularly perturbed non-linear PDE 
which for $\ve \to 0$ provides the sharp interface limit.  This approach is rigorously justified in the 1D model problem, however we believe that  this justification  can be carried out in the 2D problem \eqref{eq1}-\eqref{lagrange}.   For  small $\beta$ it  is implemented  via the contraction mapping principle; for large $\beta$  it  requires   more subtle stability analysis of a semigroup generated by a nonlinear nonlocal operator.

\subsection{Main results}
\label{SubsecMainRes}
The  main objectives of this work are: prove well-posedness of \eqref{eq1}-\eqref{lagrange}, 
reveal  the effect of the coupling in  \eqref{eq1}-\eqref{eq2}  on the  sharp interface limit, 
study qualitative behavior of system  \eqref{eq1}-\eqref{eq2} versus  values of the parameter $\beta$.

The first main result, Theorem \ref{wp_theorem}, demonstrates that there is no finite time blow up and that the sharp interface property of the initial data propagates in time.  
Theorem \ref{wp_theorem}   establishes existence of  solutions  to problem \eqref{eq1}-\eqref{lagrange} on  the time-interval $[0,T]$  for any $T>0$  and  sufficiently small $\ve$, $\ve <\ve_0(T)$. It also shows that a sharp (width $\ve$)  interface at $t=0$ remains sharp for $t \in (0,T)$.   This is proved by combining a maximum principle with energy type bounds.

To study  how  coupling of equations \eqref{eq1}-\eqref{eq2} along with the nonlocal volume constraint  \eqref{lagrange} affect the sharp interface limit we use  formal asymptotic expansions following the method of \cite{Fif88}. In this way we derive the equation of motion for the sharp interface,
\begin{equation}\label{motion1}
	V=\kappa +\frac{1}{c_0}\Phi_\beta(V)-\frac{1}{|\Gamma(t)|}\int_{\Gamma(t)}\left(\kappa
	+\frac{1}{c_0}\Phi_\beta(V)\right)\,  ds,
\end{equation}
where $c_0$ is a constant determined by the potential $W(\rho)$  and the function $\Phi(V)$  is  a given function (obtained by solving \eqref{def_for_motion1}).


%

The parameter $\beta$ in \eqref{eq2} can be thought of as the strength of  coupling
in system  \eqref{eq1}-\eqref{lagrange}. If $\beta$ is small, then \eqref{eq1} and \eqref{motion1} can be viewed as a  perturbation of Allen-Cahn equation with volume preserving term 
and curvature driven motion \eqref{rs_1992}, respectively. Results of the work \cite{MizBerRybZha15}, which addresses  \eqref{motion1} for small (subcritical) values of $\beta$, show that curves 
evolving according to \eqref{motion1} behave similarly to those satisfying \eqref{rs_1992}: they become close to circles quite fast exhibiting a little shift compared with curvature driven motion.   
On the other hand, if $\beta$ is not small evolution of sharp interface changes dramatically. 
In this case the function $c_0V-\Phi_{\beta}(V)$ is no longer invertible and one can expect quite complicated behavior of the interface curve. As the first step to study this case, it is natural to look for solutions for \eqref{eq1}-\eqref{lagrange} with steady motion. We can predict existence of such solutions based on our results for a 1D analogue of \eqref{eq1}-\eqref{lagrange}.  We prove that in the 1D case there exist traveling wave solutions with nonzero velocities, provided that $\beta$ is large enough  and the potential $W(\rho)$ has certain asymmetry, e.g. $W(\rho)=\frac{1}{4}\rho^2(\rho-1)^2(\rho^2+1)$. 
Existence of such traveling waves is consistent with experimental observations of motility on keratocyte cells which exhibit self-propagation along the straight line maintaining the same shape over many times of its length \cite{KerPinAllBarMarMogThe08} .

Heuristically, for traveling waves with nonzero velocity, say $V_\ve>0$, 
the push of $P_\ve$ on the front edge of the interface must be stronger than its pullback on the rear edge. 
This asymmetry in $P_\ve$ comes forth with an asymmetry of $W(\rho)$. 
We show that the velocity $V=V_\ve$ 
solves simultaneously equations $c_0V=\Phi_\beta(V)-\lambda$ and $-c_0V=\Phi_\beta(-V)-\lambda$, 
up to a small error. These equations are  obtained in the sharp interface limit on the front and rear edges
of the interface, respectively; $-\Phi_\beta(-V)$  and $-\Phi_\beta(V)$ represent in these equations, loosely speaking, the push (and pullback) of $P_\ve$ on the front and rear edges.  Then eliminating $\lambda$ one derives $2 c_0 V=\Phi_\beta(V)-\Phi_\beta(-V)$, this yields the only solution $V=0$ unless the potential has certain asymmetry (for symmetric potentials, e.g., $W(\rho)=\frac{1}{4}\rho^2(\rho-1)^2$, one has $\Phi_\beta(V)=\Phi_\beta(-V)$). Theorem \ref{Traveltheorem}   justifies the equation $2 c_0 V=\Phi_\beta(V)-\Phi_\beta(-V)$ for velocities of traveling waves in the sharp interface limit $\ve\to 0$. Its proof is based on Schauder's fixed point theorem.

Finally, we study the 1D model parabolic problem without any restrictions on $\beta$, where the effects of curvature and volume preservation are mimicked by a given forcing term $F(t)$.  As already mentioned the main technical trick 
here is to introduce a special (two term) representation of solutions which allows us to reduce the study of the interface velocity to a single singularly perturbed nonlinear equation.  Linearization of this equation and spectral analysis of the corresponding generator lead to a notion of stable and unstable velocities. The main result here, Theorem  \ref{theorem:sharp_interface}, can be informally stated as follows. If the interface velocity $V_{\ve}$ belongs to the domain of stable 
velocities it keeps varying continuously obeying the law $c_0V_\ve(t)=\Phi_\beta(V_\ve(t))-F(t)+o(1)$ until it becomes unstable (if so).  This theoretical result is supplemented by numerical simulations which show that  interface velocities exhibit jumps 
and reveal existence of a hysteresis loop.  Also, our stability analysis predicts that stationary solutions of \eqref{eq1}-\eqref{lagrange} with circular shape of the phase field functions are 
unstable in the case of asymmetric potentials and large enough $\beta$. This conjecture is based on the fact that zero velocity is unstable in this case (see Remark \ref{CircularShapeUnstable}).


%
%
%
%
%
%
%

%
The paper is organized as follows.  Section \ref{well-posedness}  is devoted to the well-posedness of the problem \eqref{eq1}-\eqref{lagrange}.  
In Section \ref{formalderivation} the equation for the interface motion \eqref{motion1} is formally derived. Section \ref{TravelWaveSection} deals with traveling wave solutions.
Section  \ref{limit} contains the rigorous justification of the sharp interface limit in the context of the model 1D problem.

\section{Well-posedness of the problem and formal derivation of the sharp interface limit}


\subsection{Existence of the solution of \eqref{eq1}-\eqref{eq2} with $\ve$-transition layer and no finite time blow up}
\label{well-posedness}

In this section we consider  the system  \eqref{eq1}-\eqref{lagrange} 
supplemented with the Neumann and the Dirichlet boundary conditions on $\partial \Omega$ for $\rho_\ve$  and $P_\ve$, respectively,
\begin{equation}
\label{boundaryCond}
\partial_\nu \rho_\ve=0\ \text{and}\  P_\ve=0\ \text{on}\ \partial\Omega.
\end{equation}

Introduce the following energy-type 
functionals
\begin{equation}
\label{energy}
\begin{array}{l}
E_\ve(t):=\frac{\ve}{2} \int_\Omega |\nabla \rho_\ve(x,t)|^2dx+\frac{1}{\ve}\int_\Omega W( \rho_\ve(x,t) )dx,\\ \\
F_{\ve}(t):= \int_\Omega \Bigl(| P_\ve(x,t)|^2+|P_\ve(x,t)|^4\Bigr)dx.
\end{array}
\end{equation}
\noindent
Assyme that  system \eqref{eq1}-\eqref{eq2} is supplied with initial data that satisfy:
\begin{equation}
\label{rho}
-\ve^{1/4}< \rho_\ve (x,0)<1+\ve^{1/4},
\end{equation}
and
\begin{equation}
E_\ve(0)+F_\ve(0)\leq C.
\label{IniEnBound}
\end{equation}

The first  condition \eqref{rho} is a weakened form of a standard  condition  $0\leq \rho_\ve(x,0)\leq 1$ for the phase field variable. If $\lambda_\ve \equiv 0$, then the maximum principle implies $0\leq \rho_\ve(x,t)\leq 1$ for $t>0$.  The presence of  nontrivial  $\lambda_\ve$ leads to  an  ``extended interval"  for $ \rho_\ve$.\footnote{The exponent $1/4$ in \eqref{rho} can be replaced by any positive number less than $1/2$ as will be seen in the proof the next theorem, {see} Appendix A.3.} The second condition \eqref{IniEnBound} means
that at $t=0$ the function
$\rho_\ve$ has the structure of an ``$\ve$-transition layer", that is, the domain $\Omega$  consists of three subdomains:  one 
where $\rho_\ve\sim 1$ (inside  the cell),
another where $\rho_\ve \sim 0$ (outside the cell), and they are separated by a  transition layer of width $\ve$  (a diffusive interface). Furthermore, it can be shown that the  magnitude of the orientation field  $P_\ve$ is small everywhere except  the $\ve$-transition layer (see \eqref{4aprior}).

\begin{theorem} \label{wp_theorem} ({\it No finite time blow up}) If the initial data $\rho_\ve(x,0)$, $P_\ve(x,0)$ satisfy  \eqref{rho} and
	\eqref{IniEnBound}, then for any $T>0$ the solution of \eqref{eq1}-\eqref{eq2} $\rho_\ve$, $P_\ve$ with boundary conditions \eqref{boundaryCond} exists on the time interval
	$(0,T)$ for sufficiently small $\ve>0$, $\ve<\ve_0(T)$.
	Moreover, $-\ve^{1/4}\leq \rho_\ve (x,t)\leq 1+\ve^{1/4}$ and
	\begin{equation}\label{noblowup1}
	\ve \int_0^T\int_\Omega \Bigl(\frac{\partial\rho_\ve}{\partial t} \Bigr)^2dxdt\leq C,
	\quad
	E_\ve(t)+F_\ve(t)\leq C\quad \forall t\in(0,T),
	\end{equation}
	where $C$ is independent of $t$ and $\ve$.
\end{theorem}

\begin{remark}  This theorem implies that if the initial data are well-prepared  in the sense of  \eqref{rho}-\eqref{IniEnBound},  then for $0<t <T$  the solution exists and  has the structure of  an $\ve$-transition layer.  Moreover, the bound  on initial data \eqref{rho}   remains true for $t>0$. While it relies on a maximum principle  argument, it also  requires additional estimates on $\lambda_\ve$ as seen from \eqref{max_prl} below.
\end{remark}

\noindent \begin{proof}$\;$\\
	 First multiply  \eqref{eq1} by $\partial_t \rho_\ve$ and integrate over $\Omega$:
	\begin{equation}
	\begin{aligned}
	\int_\Omega |\partial_t \rho_\ve|^2dx+\frac{\rm d}{{\rm d} t}\int_\Omega\bigl(\frac{1}{2}|\nabla \rho_\ve|^2+\frac{1}{\ve^2}W(\rho_\ve)\bigr)dx
	&= -\int_\Omega P_\ve\cdot \nabla \rho_\ve\, \partial_t\rho_\ve dx\\
	&\leq
	\frac{1}{2}\int_\Omega|\partial_t \rho_\ve|^2dx+\frac{1}{2}\int_\Omega |P_\ve|^2\, |\nabla \rho_\ve|^2dx.
	\label{1aprior}
	\end{aligned}
	\end{equation}
	
	Here we used the fact that, due to \eqref{lagrange}, the integral of $\partial_{t}\rho_{\varepsilon}$ over $\Omega$ is zero and thus
	\begin{equation*}
	\int\limits_\Omega \lambda_{\varepsilon}(t)\partial_{t}\rho_{\ve}dx=0.
	\end{equation*}
	Next, using the maximum principle in
	\eqref{eq1} we get:
	\begin{equation}
	\label{max_prl}
	-2\ve^2 \sup_{\tau\in(0,t]} |\lambda_\ve(\tau)| \leq \rho_\ve \leq 1+2\ve^2 \sup_{\tau\in(0,t]} |\lambda_\ve(\tau)|.
	\end{equation}
	Let $T_\ve>0$ be the maximal time such that
	\begin{equation}
	\label{min_max}
	-\ve^{1/4}\leq \rho_\ve\leq 1+\ve^{1/4}, \quad \text{when}\ t\leq T_\ve,
	\end{equation}
	and from now on assume that $t\leq T_\ve$.
	
	Using \eqref{1aprior}, \eqref{min_max} and integrating by parts we obtain
	\begin{equation}\label{2aprior}
	\frac{d}{{ d}t}
	E_{\varepsilon}+\frac{\varepsilon}{4}\int_{\Omega}|\partial_t \rho_{\ve}|^2dx\leq \varepsilon\int \left(|P_{\varepsilon}|^2|\Delta\rho_{\ve}|+|\nabla |P_{\ve}|^2||\nabla \rho_{\ve}|\right)dx
	\end{equation}
	
	We proceed by deriving an upper bound for the integral in the right hand side of (\ref{2aprior}). By (\ref{eq1}) we have
	\begin{equation}
   \begin{aligned}
	\int_\Omega(|\Delta \rho_\ve|\, |P_\ve|^2 +|\,\nabla |P_\ve|^2\,|\,\,|\nabla \rho_\ve|)&dx\leq \int_\Omega |\partial_t \rho_\ve| |P_\ve|^2dx
	+\int_\Omega |P_\ve\cdot \nabla\rho_\ve| |P_\ve|^2dx\\
	&+\int_\Omega |\nabla \rho_\ve| |\,\nabla |P_\ve|^2|dx
	+\frac{1}{\ve^2} \int_\Omega |W^\prime(\rho_\ve)| |P_\ve|^2dx+|\lambda_\ve|\int_\Omega |P_\ve|^2dx
	=:\sum\limits_{i=1}^{5}I_i.
\end{aligned}
	\label{raspisali}
	\end{equation}
	The following bounds are obtained  by routine application of the Cauchy-Schwarz and Young's inequalities.
	For the sum of the first three terms in (\ref{raspisali}) we get,
	$$
	\sum_1^3I_i\leq \ve \int_\Omega(\partial_t \rho_\ve)^2dx + \ve\int_\Omega |P_\ve|^2 \, |\nabla\rho_\ve|^2dx+
	\frac{1}{2\ve}
	\int_\Omega |P_\ve|^4dx
	+ \int_\Omega|\,\nabla \bigl|P_\ve|^2\bigr|^2dx+\frac{1}{\ve}E_\ve.
	$$
	Since $(W^\prime(\rho_\ve))^2\leq C W(\rho_\ve)$ we also have
	$$
	\frac{1}{\ve^2}\int_\Omega|W^\prime(\rho)|\, |P_\ve|^2dx
	\leq \frac{C}{\ve^2}\int_\Omega W(\rho)dx+\frac{1}{2\ve^2}\int_\Omega|P_\ve|^4dx\leq \frac{C}{\ve}E_\ve+\frac{1}{2\ve^2}
	\int_\Omega |P_\ve|^4dx.
	$$
	Finally, in order to bound $I_5$ we first derive, 
	\begin{equation}
	\begin{aligned}
	|\lambda_\ve(t)|\leq\frac{C}{\ve^2}\Bigl(\int_\Omega W(\rho_\ve)dx\Bigr)^{1/2}
	+\Bigl(\int_\Omega |\nabla \rho_\ve|^2dx\Bigr)^{1/2}\Bigl(\int_\Omega | P_\ve|^2dx\Bigr)^{1/2}
	\leq\frac{C}{\ve} \Bigl(\frac{ E_\ve}{\ve}\Bigr)^{1/2}+\Bigl(\frac{2 E_\ve}{\ve}\int_\Omega | P_\ve|^2dx\Bigr)^{1/2},
	\end{aligned}
	\label{lambound}
	\end{equation}
	then
	\begin{eqnarray*}
		I_5&\leq&
		\frac{C}{\ve}\left(\frac { E_\ve}{\ve}\right)^{1/2} \int_\Omega |P_\ve|^2dx+
		\left(\frac{2E_\ve}{\ve}\right)^{1/2}\left(\int_\Omega |P_\ve|^2dx\right)^{3/2}\\&\leq&
		\frac{C}{\ve} E_\ve+\frac{1}{2\ve^2}\int_\Omega |P_\ve|^4dx+
		E_\ve^2 +\frac{C}{\ve^{2/3}}\int_\Omega |P_\ve|^4dx.
	\end{eqnarray*}
	Thus,
	\begin{equation*}
		\sum_1^5 I_i
		\leq
		\frac{C}{\ve} E_\ve +E_\ve^2 +\frac{1+O(\ve)}{\ve^2}\int_\Omega |P_\ve|^4dx +
		\ve \int_\Omega(\partial_t \rho_\ve)^2dx
		+ \ve\int_\Omega |P_\ve|^2 \, |\nabla\rho_\ve|^2dx
		+\int_\Omega|\,\nabla \bigl|P_\ve|^2\bigr|^2dx,
	\end{equation*}
	and using this inequality, (\ref{raspisali}) and (\ref{min_max}) in (\ref{2aprior}),
	then substituting the resulting bound in (\ref{1aprior}) we obtain, for sufficiently small $\ve$,
	\begin{equation}
	\label{3aprior}
	\frac{1}{4}\int_\Omega |\partial_t \rho_\ve|^2dx+
	\frac{\rm d}{{\rm d} t}\frac{E_\ve}{\ve}
	\leq \frac{C}{\ve}E_\ve+ E_\ve^2+\frac{1}{\ve^2}\int_\Omega |P_\ve|^4dx
	+\int_\Omega|\,\nabla \bigl|P_\ve|^2\bigr|^2dx.
	\end{equation}
	
    Now we obtain a bound for the last two terms in  (\ref{3aprior}).
	Taking the scalar product of \eqref{eq2} with $2k P_\ve+4|P_\ve|^2 P_\ve$, $k>0$, integrating over $\Omega$ and using \eqref{min_max} we get
	\begin{equation*}
	\begin{aligned}
\frac{d}{{ d} t}\int_\Omega
(k|P_\ve|^2+|P_\ve|^4)dx&+\ve\int_\Omega (2k |\nabla P_\ve|^2+ 4|\nabla P_\ve|^2\,|P_\ve|^2+2\bigl|\nabla |P_\ve|^2\bigr|^2)dx
+\frac{2}{\ve}\int_\Omega (k|P_\ve|^2+2|P_\ve|^4)dx
\\
&= -2\beta k \int_\Omega P_\ve\cdot \nabla\rho_\ve dx+4\beta\int_\Omega\rho_\ve\,{\rm div}(P_\ve|P_\ve|^2)dx
\\
&\leq k C\ve\int_\Omega|\nabla \rho_\ve|^2 dx+\frac{k}{\ve}\int_\Omega |P_\ve|^2 dx
+\ve\int_\Omega |\nabla P_\ve|^2\,|P_\ve|^2 dx+\frac{C_1}{\ve} \int_\Omega |P_\ve|^2 dx.
	\end{aligned}
	\end{equation*}
	We chose $k:=C_1+1$ to obtain
	\begin{equation}
	\label{4aprior}
	\ve\int_\Omega \bigl|\nabla |P_\ve|^2\bigr|^2dx+\frac{1}{\ve}\int_\Omega |P_\ve|^4dx\leq C E_\ve-
	\frac{ d}{{ d} t}\int_\Omega
	(k|P_\ve|^2+|P_\ve|^4)dx.
	\end{equation}
	
 Finally, introducing $G_\ve=E_\ve+\int_\Omega(4k|P_\ve|^2+|P_\ve|^4)\, dx$, by (\ref{3aprior}) and (\ref{4aprior})
	we have the differential
	inequality,
	\begin{equation}
	\label{dif_ineq}
	\frac{{ d} G_\ve}{{ d} t}\leq C G_\ve+\ve G_\ve^2,
	\end{equation}
	with a constant $C>0$ independent of $\ve$. Considering the bounds
	on the initial data and assuming that $\ve$ is sufficiently small,
	one can easily construct  a  bounded  supersolution $\tilde G$ of
	(\ref{dif_ineq}) on $[0,T]$ such that $\tilde G(0)\geq G_\ve$. We
	now have,  $G_\ve\leq C$ on $[0,T_\ve]$ for sufficiently small
	$\ve$. By  (\ref{max_prl}) and (\ref{lambound}) we then conclude
	that $T_\ve$ in (\ref{min_max}) actually coincides  with $T$ when
	$\ve$ is small. The theorem is proved. \qed
\end{proof}

\subsection{Formal derivation of the  Sharp Interface Equation \eqref{motion1}}\label{formalderivation}

In this section we formally derive equation \eqref{motion1} for the 2D   system  \eqref{eq1}-\eqref{eq2}.
While the derivation is analogous  to the single  Allen-Cahn equation (e.g., \cite{RubSterKel89},
\cite{CheHilLog10}), the gradient coupling in \eqref{eq1}-\eqref{eq2} results in a nonlinear term that modifies the mean curvature motion. 

Assume that that initial data $\rho_{\ve}(x,0)$ converge to the characteristic function of a smooth subdomain $\omega_0\subset \Omega$ as $\ve\to 0$. Next we want to describe the evolution of the interface $\Gamma(t)=\partial \omega_t$ with $t$, where $\omega_t$ is the support of $\lim\limits_{\ve\to 0} \rho_\ve(x,t)$.  We will assume  that
the initial data coincide with initial values of asymptotic expansions for $\rho_\ve$ and $P_\ve$ to be constructed below.


Let $X_0(s,t)$ be a parametrization of $\Gamma(t)$. In a vicinity of $\Gamma(t)$  the parameters $s$ and the  signed distance   $r$ to $\Gamma(t)$  will be used as local coordinates, so that
\begin{equation}\nonumber
x=X_0(s,t)+r \n(s,t)=X(r,s,t), \qquad \text{where  $\n$ is an inward normal to $\Gamma(t)$.}
\end{equation}
The inverse mapping to $x=X(r,s,t)$ is given by
\begin{equation}\nonumber
r=\pm\text{dist}(x,\Gamma(t)),\;\;s=S(x,t),
\end{equation}
where in the formula for $r$ we choose $+$ if $x\in \omega_t$ and $-$, if $x\notin \omega_t$.
Recall that $\Gamma(t)$ is the limiting location of interface as $\ve \to 0$. 
Next we seek $\rho_\ve$ and $P_{\ve}$ in the following forms in local coordinates $(r,s)$:
\begin{equation}
\rho_\ve(x,t)=\tilde{\rho}_\ve \left(\frac{r(x,t)}{\varepsilon},S(x,t),t\right)\text { and }P_\ve(x,t)=\tilde{P}_\ve \left(\frac{r(x,t)}{\varepsilon},S(x,t),t\right).
\end{equation}

Introduce asymptotic expansions in local coordinates:
\begin{eqnarray} \label{expansion_5}
\tilde\rho_{\varepsilon}(z,s,t)&=&\theta_0(z,s,t)+\varepsilon \theta_{1}(z,s,t)+...\\
\tilde P_{\varepsilon}(z,s,t)&=&\Psi_0(z,s,t)+\dots
\\
\lambda_{\varepsilon}(t)&=&\frac{\lambda_0(t)}{\ve}+
\lambda_1(t)+\ve\lambda_2(t)+...
\label{expansion9}
\end{eqnarray}
Now, substitute \eqref{expansion_5}-\eqref{expansion9} into \eqref{eq1} and \eqref{eq2}. Collecting terms with likewise powers of $\ve$ ($\ve^{-2}$ and $\ve^{-1}$) and equating them to zero we successively get,
\begin{equation}\label{theta_0z}
\frac{\partial^2 \theta_0}{\partial z^2}=W'(\theta_0),
\end{equation}
and 
\begin{eqnarray}
-\frac{\partial^2 \theta_1}{\partial z^2}+W''(\theta_0)\theta_1&=&V_0\frac{\partial\theta_0}{\partial z}-\frac{\partial \theta_0}{\partial z}\kappa(s,t)-(\Psi_0\cdot \n)\frac{\partial \theta_{0}}{\partial z}+\lambda_{0}(t),\label{solv_for_theta_1}\\
-V_0\frac{\partial \Psi_0}{\partial z}&=& \frac{\partial^2 \Psi_0}{\partial z^2}-\Psi_0-\beta \frac{\partial \theta_0}{\partial z}\n,\label{def_of_psi}
\end{eqnarray}
where $\kappa(s,t)$ is the curvature of $\Gamma_0(t)$ and $V_0(t):=-\partial_t r$ is the limiting velocity. The curvature $\kappa$ appears in the equation when one rewrites the Laplace operator in \eqref{eq1} in local coordinates $(r,s)$. 

It is well-known that 
there exists a standing wave solution $\theta_0(z)$ of \eqref{theta_0z} which tends to $1$ as $z\to \infty$ and to $0$ as $z\to-\infty$, respectively. Moreover, all derivatives of the function $\theta_0(z)$  exponentially decay to $0$ as $|z|\to\infty$ and $\theta_0'(z)$ is an eigenfunction of the linearized Allen-Cahn operator $\mathcal{L}u:=-u''+W''(\theta_0)u$ corresponding to the eigenvalue $0$.  Then multiplying \eqref{solv_for_theta_1} by $\theta_0'(z)$ and integrating over $z$ we are lead to the solvability condition for 
\eqref{solv_for_theta_1}:
%
%
%
\begin{equation}\label{formalmotion1}
c_0 V_0(s,t)=c_0 \kappa(s,t)+\int (\Psi_0\cdot \n)\left(\frac{\partial \theta_0}{\partial z}\right)^2 dz
- \lambda_0(t), \text { where }c_0=\int_{\mathbb R}\left(\frac{\partial \theta_0}{\partial z}\right)^2 dz.
\end{equation}
Next we obtain the formula for $\lambda_0(t)$. It follows from  \eqref{lagrange} that $\int_{\Omega}\partial_t \rho_{\ve}=0$. Substitute expansion \eqref{expansion_5}  for $\rho_{\ve}$ into $\int_{\Omega}\partial_t \rho_{\ve}=0$ and take into account the fact that $$\partial_t\rho_{\ve}=-\theta_0'\left(\frac{r}{\ve}\right)\frac{V_0(s,t)}{\ve}+O(1).$$ Thus, in order to satisfy the condition  $\int_{\Omega}\partial_t \rho_{\ve}=0$ to the leading order, $V_0(s,t)$ must have $$\int V(s,t)|\frac{\partial }{\partial s}X_0(s,t)|ds~=~0.$$ 
Using this fact and integrating \eqref{formalmotion1} with respect to $s$ with the weight $|\frac{\partial }{\partial s}X_0(s,t)|$, we get
\begin{equation}\label{def_lambda_0}
\lambda_0(t)=\int \left\{c_0 \kappa(s,t)+\int (\Psi_0\cdot \n)\left(\frac{\partial \theta_0}{\partial z}\right)^2 dz\right\} |\frac{\partial }{\partial s}X_0(s,t)|ds.
\end{equation}

Finally, the unique solution of \eqref{def_of_psi}  is given by $\Psi_0(z,s,t)=\psi(z,V_0(s,t))\n (s,t)$ where $\psi=\psi(z,V)$ is the unique (bounded) solution of 
\begin{equation}\label{psi_in_formal}
\partial^2_z\psi +V\partial_z \psi-\psi-\beta\theta_0'=0.
\end{equation}
The representation $\Psi_0(z,s,t)=\psi(z;V_0(s,t))\n (s,t)$ yields 
\begin{equation}\label{def_for_motion1}
\int \Psi_0 \cdot \n (\theta_0')^2 dz=\Phi_\beta (V) : = \int \psi(z,V) (\theta_0^\prime)^2 dz,
\end{equation}
where we  have also taken into account the linearity of \eqref{psi_in_formal} in $\beta$.
Now substitute \eqref{def_for_motion1} and \eqref{def_lambda_0} into equation \eqref{formalmotion1} to  conclude the derivation of sharp interface equation \eqref{motion1}.


\section{Traveling waves in 1D}
\label{TravelWaveSection}

In this section we study special solutions of system \eqref{eq1}-\eqref{eq2} in the 1D case. Specifically, we look for traveling waves (traveling pulses). Therefore it is natural to switch to the entire space  $\mathbb{R}^1$ setting. 
We show that,  not surprisingly, there are nonconstant stationary solutions, standing waves. However, we prove  that apart from standing waves there are true traveling waves 
 when the parameter $\beta$ is large enough and the potential $W(\rho)$ has certain asymmetry,  e.g. $W(\rho)=\frac{1}{4}(\rho^2+\rho^4)(\rho-1)^2$, see also the discussion in Remark \ref{KosoiPotential}. 


We are interested in (localized in some sense) solutions of \eqref{eq1}-\eqref{eq2} with
$\rho_{\ve} =\rho_{\ve}(x-Vt)$, $P_{\ve}=P_{\ve}(x-Vt)$. They satisfy the following  stationary
equations with unknown constant velocity $V$ and constant $\lambda$:
\begin{eqnarray}
0&=&\partial^2_x \rho_{\ve} +V\partial_x\rho_{\ve}-\frac{W'(\rho_{\ve})}{\ve^2}-P_{\ve}\partial_x\rho_{\ve} +\dfrac{\lambda}{\ve}\label{tw_rho},\\
0&=&\ve \partial_x^2P_{\ve}+V\partial_xP_{\ve} -\frac{1}{\ve} P_{\ve} -\beta\partial_x\rho_{\ve}. \label{tw_P}
\end{eqnarray}

Let us postulate an ansatz for the phase field function $\rho_\ve$. Given $a>0$, we look for solutions of \eqref{tw_rho}-\eqref{tw_P} for sufficiently small
$\ve>0$ with $\rho_\ve$ having the form
\begin{equation}
\label{repr}
\rho_\ve=\phi_\ve+\ve\chi_\ve+ \ve u,
\end{equation}
where 
$$
\phi_\ve:=\theta_0((x+a)/\ve)\theta_0((a-x)/\ve),\quad  \chi_\ve:=\chi^{-}_\ve+(\chi^{+}_\ve-\chi^{-}_\ve)\phi_\ve,
$$ 
constants $\chi^{-}_\ve$ and $\chi^{+}_\ve$ are the smallest (in absolute value) solutions of
$W^\prime(\ve\chi^{-}_\ve)=\ve\lambda$ and $W^\prime(1+\ve\chi^{+}_\ve)=\ve\lambda$, respectively, and $u$ is the
new unknown function vanishing at $\pm\infty$. The role of the constant $\chi^{-}_\ve$ in \eqref{repr}
is to amend the first term of the representation so that $u$ decays at $\pm \infty$. Similarly, $\chi^{+}_\ve$ is introduced to end up with $u$ which is exponentially close to one in $(-a,a)$ away 
from points $\pm a$ (see also Fig.~\ref{fig:sketch_chi}).

\begin{figure}
	\begin{center}
		\includegraphics[width=\textwidth]{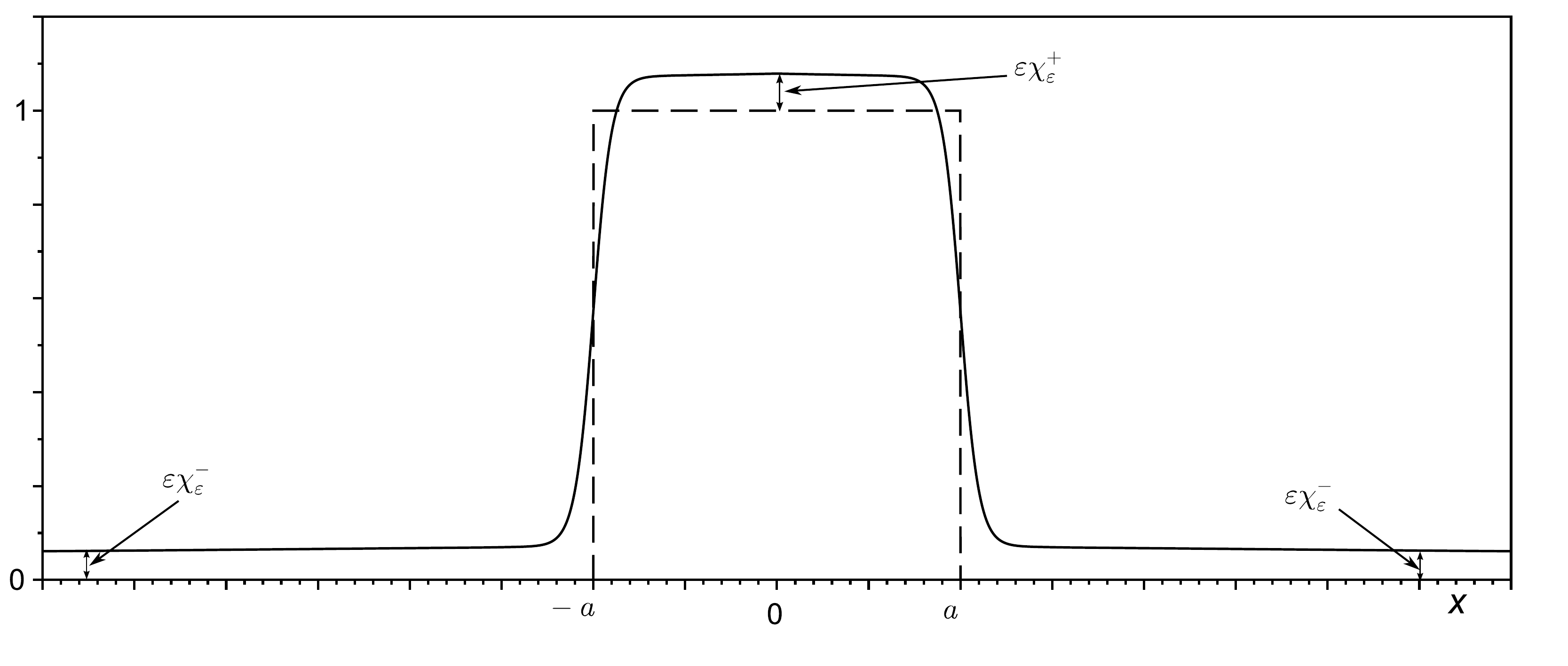}
		\caption{Illustration of the ansatz \eqref{repr}. Function $\rho_\ve$ decays to a non-zero constant of the order $\ve$  for $x\to\pm\infty$ and to a constant slightly different from 1 for $-a\leq x\leq a$ (solid). Dashed line represents the limiting profile, which is the characteristic function of $(-a,a)$.}
		\label{fig:sketch_chi}
	\end{center}
\end{figure}

Substitute representation \eqref{repr} in \eqref{tw_rho}-\eqref{tw_P} to find after rescaling the variable $y:=x/\ve$ and rearranging terms,
\begin{equation}
\begin{aligned}
\partial_y^2  u-W''(\phi_{\ve})u&=-V\partial_y \phi_{\ve}
+P_{\ve}(\partial_y\phi_{\ve}+\ve \partial_y u)
-\lambda
+\frac{1}{\ve}\Bigl(W^{\prime}(\phi_{\ve}+\ve\chi_\ve)-\partial^2_y(\phi_{\ve}+\ve \chi_\ve)\Bigr)\\
&+
\frac{1}{\ve}\Bigl(
W^{\prime}(\phi_{\ve}+\ve\chi_\ve+\ve u)
-W^{\prime}(\phi_{\ve}+\ve\chi_\ve)
-\ve W''(\phi_{\ve})u
\Bigr)
-\ve V\partial_y( u+\chi_\ve),
\end{aligned}
\label{eq_for_u_tw}
\end{equation}
\begin{equation}
\partial_y^2P_\ve+V\partial_y P_\ve -P_\ve=\beta \partial_y \phi_\ve
+\ve \beta\partial_y(\chi_\ve+u).
\label{eq_for_P_tw}
\end{equation}
Note that the ansatz \eqref{repr} yields the characteristic function 
of the interval $(-a,a)$ in the limit $\ve\to 0$, provided that $u=u_\ve$ remains bounded. 
In this sense we seek solutions with localized profiles of the phase field function $\rho_\ve$. The idea of the construction of traveling wave solutions is based on the observation that solvability of  the above equations   \eqref{eq_for_u_tw} and \eqref{eq_for_P_tw} can be handled by local analysis near the points $\pm a$. Indeed, setting  $z=y+a$  
\eqref{eq_for_u_tw}-\eqref{eq_for_P_tw} and keeping only leading order terms  
we (formally) obtain 
$$
\partial_z^2u-W^{\prime\prime}(\theta_0(z))u=-V\theta_0^\prime(z) +P_\ve\theta_0^\prime-\lambda\ \ \text{and}\ \ \partial_z^2P_\ve+V\partial_z P_\ve -P_\ve=\beta \theta_0^\prime(z).
$$
Resolve the second equation to obtain $P_\ve(z)=\psi(z,V)$, then solvability of the first equation (recall that $\partial_z^2\theta_0^\prime(z)-W^{\prime\prime}(\theta_0(z))\theta_0^\prime(z)=0$) requires that 
$c_0V=\Phi_\beta(V)-\lambda$, where we have used \eqref{def_for_motion1}. Similarly, local analysis near the point $a$ leads to  the equation $-c_0V=\Phi_\beta(-V)-\lambda$. Thus, we have reduced the infinite dimensional system \eqref{eq_for_u_tw}-\eqref{eq_for_P_tw} to a two dimensional one. 

In order to transform the above heuristics into a rigorous analysis we 
reset \eqref{eq_for_u_tw}-\eqref{eq_for_P_tw} as a fixed point problem.
To this end rewrite \eqref{eq_for_u_tw}
in the following  form, introducing  auxiliary functions ${\theta}_\ve^{(1)}(y)=\theta_0^\prime(y+a/\ve)+\theta_0^\prime(a/\ve-y)$
and $\theta_\ve^{(2)}(y)=\theta_0^\prime(y+a/\ve)-\theta_0^\prime(a/\ve-y)$,
\begin{equation}
\label{pereoboznachili}
\partial_y^2  u-W''(\phi_{\ve})u-
\frac{\partial^2_y\theta_\ve^{(1)}-W^{\prime\prime}(\phi_\ve)\theta_\ve^{(1)}}{\theta_\ve^{(1)}}u+H_\ve\int u\theta_\ve^{(2)} dy=G(\lambda,V,P_\ve,u),
\end{equation}
where
$$
H_\ve=\frac{1}{\displaystyle\int \left(\theta_\ve^{(2)}\right)^2 dy}\Biggl(W''(\phi_{\ve})\theta_\ve^{(2)}-\partial_y^2\theta_\ve^{(2)}  +
\frac{\partial^2_y\theta_\ve^{(1)}-W^{\prime\prime}(\phi_\ve)\theta_\ve^{(1)}}
{\theta_\ve^{(1)}}\theta_\ve^{(2)}
\Biggr)
$$
and
$$
\begin{aligned}
G(\lambda,V,P_\ve,u)&=H_\ve\int u\theta_\ve^{(2)} dy
-\frac{\partial^2_y\theta_\ve^{(1)}-W^{\prime\prime}(\phi_\ve)\theta_\ve^{(1)}}
{\theta_\ve^{(1)}}u
-V\partial_y \phi_{\ve}
+P_{\ve}(\partial_y\phi_{\ve}+\ve \partial_y u)-\ve V \partial_y (u+\chi_\ve)
-\lambda\\
&+\frac{1}{\ve}\Bigl(W^{\prime}(\phi_{\ve}+\ve\chi_\ve)-\partial^2_y(\phi_{\ve}+\ve \chi_\ve)\Bigr)
+
\frac{1}{\ve}\Bigl(
W^{\prime}(\phi_{\ve}+\ve\chi_\ve+\ve u)
-W^{\prime}(\phi_{\ve}+\ve\chi_\ve)
-\ve W''(\phi_{\ve})u
\Bigr).
\end{aligned}
$$
Note that the operator $\mathcal{Q}_\ve$ in the left hand side of \eqref{pereoboznachili},
\begin{equation*}
\mathcal{Q}_\ve u: = \partial_y^2  u-W''(\phi_{\ve})u-
\frac{\partial^2_y\theta_\ve^{(1)}-W^{\prime\prime}(\phi_\ve)\theta_\ve^{(1)}}{\theta_\ve^{(1)}}u+H_\ve\int u\theta_\ve^{(2)} dy
\end{equation*}
has
two eigenfunctions $\theta_\ve^{(1)}$ and $\theta_\ve^{(2)}$
corresponding to the zero eigenvalue.

	\begin{lemma}\label{lemma:bound_on_q}
		Let $v_\ve$ be orthogonal to both $\theta^{(1)}_{\ve}$ and $\theta^{(2)}_{\ve}$ in $L^2(\mathbb R)$. Assume also that  $f_\ve:=\mathcal{Q}_{\ve} v_\ve$ belongs to $L^2(\mathbb R)$. Then 
		\begin{equation}\label{lemma_ineq_for_tw}
			\|v_\ve\|_{H^1}\leq C \|f_\ve\|_{L^2}+Ce^{-r/\ve},
		\end{equation}
		where constants $C$ and $r>0$ are independent from $\ve$.	
	\end{lemma}
	\begin{proof}
		Multiplying $\mathcal{Q}_\ve
		v_\ve$ by $v_\ve$ in $L^2(\mathbb R)$ and representing $v_\ve$ as
		$v_\ve=\theta_\ve^{(1)} w_\ve$ (note that $\theta_\ve^{(1)}>0$)
		we derive
		\begin{equation}
			\label{nettretegosostvznach}
			(\mathcal{Q}_\ve v_\ve, v_\ve)_{L^2}=
			\int\left(2\partial_y\theta_\ve^{(1)} \partial_y w_\ve+\theta_\ve^{(1)}\partial^2_y w_\ve\right)
			\theta_\ve^{(1)} w_\ve dy
			=-\int\left(\theta_\ve^{(1)}\right)^2
			(\partial_y w_\ve)^2 dy,
		\end{equation}   
		where the latter equality is obtained via integrating by parts, and the term with $H_\ve$ vanishes  
		thanks to orthogonality of $v_\ve$ to $\theta^{(2)}_\ve$. 
		Substituting the definition of $f_\ve$ into \eqref{nettretegosostvznach}, we obtain
		\begin{equation}\label{fv_bound}
			\int\left(\theta_\ve^{(1)}\right)^2
			(\partial_y w_\ve)^2 dy\leq \|f\|_{L^2} \|v_\ve\|_{L^2}.
		\end{equation}
		The statement of Lemma \ref{lemma:bound_on_q} immediately follows if we prove the following inequality
		\begin{equation}
			\label{poincare_prime}
			\int v_\ve^2 dy \leq C\int (\theta^{(1)}_{\ve})^2 (\partial_y w)^2 dy + Ce^{-r/\ve}
		\end{equation}
		with $C,r>0$ independent from $\ve$. Indeed, using the estimate \eqref{poincare_prime} and the Cauchy  inequality in the right hand side of \eqref{fv_bound} we get 
		\begin{equation}\label{what_we_need_in_lemma}
			\int \left(\theta_\ve^{(1)}\right)^2
			(\partial_y w_\ve)^2 dy \leq C\|f\|^2 + Ce^{-r/\ve},
		\end{equation} 
		which together with 
		\begin{equation*}
			\|v_\ve\|_{H^1}^2\leq \int \left(\theta_\ve^{(1)}\right)^2
			(\partial_y w_\ve)^2 dy+ C\int v_\ve^2 dy
		\end{equation*}
		implies \eqref{lemma_ineq_for_tw}.
		
		To prove \eqref{poincare_prime} we  use the Poincar\'e inequality (see \ref{appendix_poincare})
		\begin{equation}
			\int 
			\left(\theta^\prime_0(a/\ve\pm y)\right)^2
			\left| w_\ve - \langle w_\ve\rangle_{\pm} \right|^2 dy
			\leq
			C\int\left(\theta^{(1)}_\ve\right)^2 (\partial_y w_\ve)^2dy
			\label{ner-vo_tipa_poinc}
		\end{equation}
		with a constant $C$ independent of $\ve$ and  
		\begin{equation*}
			\langle w_\ve \rangle_{\pm}:=\frac{\int \left(\theta^\prime_0(a/\ve\pm y)\right)^2w_\ve d y}{\int \left(\theta^\prime_0\right)^2 dy}.
		\end{equation*}
		Due to orthogonality of $\theta_\ve^{(1)} w_\ve$ to $\theta_\ve^{(1)}$ and $\theta_\ve^{(1)}$, 
		we have
		$$
		\int \left(\theta^\prime_0(a/\ve\pm y)\right)^2w_\ve d y=- \int \theta^\prime_0(y+a/\ve)
		\theta^\prime_0(a/\ve-y) w_\ve d y.
		$$
		Thanks to  the exponential decay of $\theta_0'$, $\theta_0'(y)\leq \alpha_0 e^{-\kappa |y|}$   (see, e.g., \cite{MotSha95}), it follows that 
		\begin{equation}
			\Bigl|
			\int \left(\theta^\prime_0(a/\ve\pm y)\right)^2 w_\ve d y
			\Bigr|
			\leq e^{-r/\ve} \Bigl(\int \left(\theta^{(1)}_\ve\right)^2 w_\ve^2dy\Bigr)^{1/2}
			\label{almost_orthogonal}
		\end{equation}
		for some $r>0$ independent of $\ve$. Combining \eqref{almost_orthogonal} and \eqref{ner-vo_tipa_poinc} we obtain \eqref{poincare_prime}, the lemma is proved. \qed
	\end{proof}
	\begin{proposition} 
		\label{propSpectralapriori}	
		For sufficiently small $\ve$ the operator  $\mathcal{Q}_\ve^\ast$
		adjoint to $\mathcal{Q}_\ve$ (with respect to the scalar product
		in $L^2(\mathbb{R})$) has two eigenfunctions
		$\theta_\ve^{(1)}$ and
		$\theta_\ve^{(3)}=\theta_\ve^{(2)}+q_\ve$
		corresponding to the zero eigenvalue, with
		$\|q_\ve\|_{H^1}=o(\ve)$. Moreover the equation $\mathcal{Q}_\ve
		u=f$ has a solution if and only if $f\in L^2(\mathbb{R})$ is
		orthogonal to the eigenfunctions $\theta_\ve^{(1)}$  and
		$\theta_\ve^{(3)}$  of $\mathcal{Q}_\ve^\ast$.
	\end{proposition}
	
	\begin{proof}
		Given $f\in L^2(\mathbb{R})$, consider the equation $\mathcal{Q}_\ve u=f$ rewriting it 
		in the form
		\begin{equation}
			\label{vspomogeq}
			\partial_y^2u-W^{\prime\prime}(0)u
			+\left(W^{\prime\prime}(0)-W^{\prime\prime}(\phi_{\ve})\right)u-
			\frac{\partial^2_y\theta_\ve^{(1)}-
				W^{\prime\prime}(\phi_\ve)\theta_\ve^{(1)}}{\theta_\ve^{(1)}}u+H_\ve\int
			u\theta_\ve^{(2)} dy=f.
		\end{equation}
		Since $W^{\prime\prime}(0)>0$, the equation 
		$\partial_y^2u-W^{\prime\prime}(0)u=\tilde f$ has the unique solution 
		$u=G\tilde f$ for every $\tilde f\in L^2(\mathbb{R})$ with a bounded
		resolving operator $G:L^2(\mathbb{R}) \to L^2(\mathbb{R})$. Moreover, by applying the operator $G$ to
		\eqref{vspomogeq} we reduce this equation to $u+Ku
		=Gf$, where $K$ is a compact operator (this can be easily shown
		using the properties of the function $\theta_0$). Thus we can
		apply the Fredholm theorem to study the solvability of
		\eqref{vspomogeq}. Note that $\mathcal{Q}_\ve$ does not have other
		eigenfunctions corresponding to the zero eigenvalue besides
		$\theta_\ve^{(1)}$ and ${\theta}_\ve^{(2)}$.
		Indeed, existence of such an eigenfunction $v_\ve$ orthogonal to
		$\theta_\ve^{(1)}$, ${\theta}_\ve^{(2)}$ in
		$L^2(\mathbb{R})$ and normalized by $\int v^2_\ve dy=1$ would contradict \eqref{fv_bound} derived in the proof of Lemma~\ref{lemma:bound_on_q}. 
		
		
		Consider now the eigenfunction $\theta^{(3)}_\ve $ of $\mathcal{Q}_\ve^\ast$ orthogonal to $\theta^{(1)}_\ve$, and represent it as $\theta^{(3)}_\ve =\theta^{(2)}_\ve +q_\ve$ 
		with $q_\ve$ orthogonal to both $\theta^{(1)}_\ve$ and $\theta^{(2)}_\ve$. 
		Then combining the equality 
		\begin{equation*}
			\mathcal{Q}_\ve q_\ve = H_\ve \int  (\theta^{(2)})^2 dy -\theta^{(2)} \int H_\ve (\theta^{(2)}_\ve - q_\ve) dy 
		\end{equation*} 
		with Lemma \ref{lemma:bound_on_q} we obtain that 
		$\|q_\ve\|_{H^1}=o(\ve)$ as $\ve\to 0$.
		\qed
	\end{proof}
	
	\bigskip
	
	Let us consider now for a given $u\in H^1(\mathbb{R})$, $V$ and $\lambda$ a solution $P_\ve$ of \eqref{eq_for_u_tw},  assuming that $\ve$ is sufficiently small
	and $\|u\|_{H^1}\leq M$, $|\lambda|\leq M$, $|V|\leq M$ for some finite $M$. 
	We represent $P_\ve$ in the form 
	\begin{equation}
		P_\ve(y)=\psi_0(y+a/\ve,V)-\psi_0(a/\ve-y,-V)+B_\ve
		\label{P_repreent}
	\end{equation}
	and observe that $B_\ve$ can be estimated as follows,
	$$
	\int(\partial_y B_\ve)^2dy+\int(B_\ve)^2dy\leq 
	\ve CM\|B_\ve\|_{L^2}\ \ \text{hence} \ \ \|B_\ve\|_{H^1}\leq \ve C_1 M.
	$$ 
	Now consider $u$ in the left hand side of \eqref{pereoboznachili} as an unknown 
	function to write down the solvability condition 
	\begin{equation}
		\label{SOLVABILITY}
		\int G(\lambda,V,P_\ve,u)\theta_{\ve}^{(k)}dy=0,\ k=1,3.
	\end{equation}
	Calculate leading terms of \eqref{SOLVABILITY} for small $\ve$ taking into account the fact that 
	\begin{equation}
		\label{Bnd43}
		W^{\prime}(\phi_{\ve}+\ve\chi_\ve+\ve u)
		-W^{\prime}(\phi_{\ve}+\ve\chi_\ve)
		-\ve W''(\phi_{\ve})u=O(\ve^2)
	\end{equation}
	and 
	\begin{equation}
		\label{Bnd44}
		W^{\prime}(\phi_{\ve}+\ve\chi_\ve)
		-\partial^2_y (\phi_\ve+\ve\chi_\ve)=
		\ve(W^{\prime\prime}(\phi_{\ve})\chi_\ve
		-\partial^2_y \chi_\ve)+ 
		O(\ve^2),
	\end{equation}
	where $O(\ve^2)$ in \eqref{Bnd43} and \eqref{Bnd44} stand for functions whose $L^\infty$-norm is bounded by $C\ve^2$. Note also that integrals 
	$$
	\int (W^{\prime\prime}(\phi_{\ve})\chi_\ve
	-\partial^2_y\chi_\ve)\theta_{\ve}^{(k)}dy=\int (W^{\prime\prime}(\phi_{\ve})\theta_{\ve}^{(k)}
	-\partial^2_y\theta_{\ve}^{(k)})\chi_\ve dy,\  k=1,3
	$$
	 tend to zero, when $\ve \to 0$. Thus  \eqref{SOLVABILITY}
	can be rewritten as
	\begin{equation}
		0=\Phi_\beta(V)+\Phi_\beta(-V)-2\lambda+\ve\tilde \Phi_1(V,\lambda,u) \quad \text{and}\quad 2c_0V=\Phi_\beta(V)-\Phi_\beta(-V)+\ve\tilde \Phi_2(V,\lambda,u), 
		\label{Reduced_system_tw}
	\end{equation}
	where functions $\tilde \Phi_1$, $\tilde \Phi_2$ and their first  partial derivatives in $V$ and $\lambda$ are uniformly 
	bounded by some constant depending on $M$ only.  Note 
	that if $V_0$ is a nondegenerate root of the equation $2c_0V=\Phi_\beta(V)-\Phi_\beta(-V)$ then
	for sufficiently small $\ve$, in a neighborhood of $V_0$ and $\lambda_0=\frac{1}{2}(\Phi_\beta(V_0)+\Phi_\beta(-V_0))$ there exists a unique pair $V_\ve(u)$ and $\lambda_\ve(u)$ solving \eqref{Reduced_system_tw} and 
	depending continuously on $u$. 
	
	\begin{theorem}
		\label{Traveltheorem}
		 (Existence of traveling waves)
		Assume that the equation $2c_0V=\Phi_\beta(V)-\Phi_\beta(-V)$ has a nondegenerate root $V_0$.
		Then for sufficiently small $\ve$ there exists a function $u_\ve$, with $\|u_\ve\|_{H^1}\leq C$ and 
		$C$ being independent of $\ve$, a function $P_\ve$ and constants $V=V_\ve$, $\lambda=\lambda_\ve$ such 
		that $\rho_\ve$ given by \eqref{repr} and $P_\ve$ are solutions of  \eqref{tw_rho}-\eqref{tw_P}. 
		Moreover,
		the velocity $V_\ve$ and the constant $\lambda_\ve$  converge to $V_0$ and $\lambda_0:=\frac{1}{2}(\Phi_\beta(V_0)+\Phi_\beta(-V_0))$ as $\ve\to 0$. 
	\end{theorem}
	
	\begin{proof}
		Consider the mapping $u\mapsto \mathcal{Q}_\ve^{-1}G(\lambda_\ve,V_\ve, P_\ve,u)$, where $\lambda_\ve=\lambda_\ve(u)$ and $V_\ve=V_\ve(u)$ solve \eqref{Reduced_system_tw}, and $P_\ve$
		is the solution of \eqref{eq_for_P_tw} with $V=V_\ve$ and $\lambda=\lambda_\ve$. 
		Since the operator $\mathcal{Q}_\ve$ has two eigenfunctions $\theta_\ve^{(1)}$ and $\theta_\ve^{(2)}$ corresponding to the zero eigenvalue, we can choose  $v:=\mathcal{Q}_\ve^{-1}G(\lambda_\ve,V_\ve, P_\ve,u)$ to be orthogonal to $\theta_\ve^{(1)}$ and $\theta_\ve^{(2)}$ in $L^2(\mathbb{R})$. Then, 
		using Lemma \ref{lemma:bound_on_q}
		one can show that for large enough $M$  and sufficiently small $\varepsilon$ it holds 
		that  if $\|u\|_{H^1}\leq M$ then $\|\mathcal{Q}_\ve^{-1}G(\lambda_\ve,V_\ve, P_\ve,u)\|_{H^1}\leq M$. 
		Also, the mapping $u\mapsto \mathcal{Q}_\ve^{-1}G(\lambda_\ve,V_\ve, P_\ve,u)$ is continuous in $H^1$. Thus, we can apply the Schauder fixed point theorem provided we establish the compactness of the mapping under consideration. To this end we consider a 	subset of functions $u$ which  decay exponentially with their first derivatives:
\begin{equation}
	\label{CLASS}
	\mathcal{K}_{M,r}:=\{u;\, \|u\|_{H^1}\leq M,\, |u|\leq Me^{-r(|y|-2a/\ve)},\, |\partial_y u|\leq  Me^{-r(|y|-2a/\ve)}\ \text{when}\ |y|\geq 2a/\ve  \}.
\end{equation}    
 We claim that for some $M>0$ and $r>0$ the solution $v$ of the equation $\mathcal{Q}_\ve v=G(\lambda_\ve,V_\ve, P_\ve,u)$ 	(orthogonal to $\theta_\ve^{(1)}$ and $\theta_\ve^{(2)}$) belongs to $\mathcal{K}_{M,r}$ for every $u\in \mathcal{K}_{M,r}$, when $\ve$ is sufficiently small.  Indeed, the required bound for the norm of $v$ in $H^1(\mathbb{R})$ is already established. It remains to prove
that $v$ and $\partial_y v $ decay exponentially  when $|y|\geq 2a/\ve$. To this end we observe first that 
\begin{equation}
\label{expdecayP}
|P_\ve| \leq C(1+\ve M) e^{-r_1(|y|-2a/\ve)} \ \text{for}\ |y|\geq 2a/\ve,
\end{equation}  
with $C>0$ independent of $M$ and $\ve$, and $r_1>0$ depending on $M$ only. The proof of \eqref{expdecayP} is carried out in two steps. First, we multiply \eqref{eq_for_P_tw} by $P_\ve$, integrate on $\mathbb{R}$   
and apply the Cauchy-Schwarz inequality. As a result we get $\|P_\ve\|_{L^\infty}\leq C\|P_\ve\|_{H^1}\leq C_1(1+\ve M)$. Second, observe that the function $\theta_0^\prime(y)$ decays exponentially when $y\to\pm\infty$. Therefore there exists $C_2 \geq  C_1(1+\ve M)$ and $r_1>0$
such that the functions $P_{\pm}(y):=\pm C_2 e^{-r_1(|y|-2 a/\ve)}$ satisfy 
$$
\mp \partial_y^2P_\ve \mp V\partial_y P_\ve \pm P_\ve\geq \beta \partial_y \phi_\ve
+\ve \beta\partial_y(\chi_\ve+u)\ \text{for} \ |y| \geq 2a/\ve.
$$
This yields pointswise bounds $-C_2 e^{-r_1(|y|-2a/\ve)}\leq P_\ve(y)\leq  C_2 e^{-r_1(|y|-2a/\ve)}$ for all $y\leq -2a/\ve$ and $y\geq 2a/\ve$.
%
Next using \eqref{expdecayP} in the equation $\mathcal{Q}_\ve v=G(\lambda,V,P_\ve, u )$ and arguing similarly one can establish that 
$|v|\leq C(1+\ve C_1(M))e^{-r_2(|y|-2a/\ve)}$ for $|y|\geq 2a/\ve$. Finally, taking an integral
from $-\infty$ to $y$ (or from $y$ to $+\infty$) of the equation $\mathcal{Q}_\ve v=G(\lambda,V,P_\ve, u )$ we get the required bound for $\partial_y v$ on
$(-\infty,-2a/\ve]$ (or $[2a/\ve,+\infty)$). 

Thus the image of the convex closed set $\mathcal{K}_{M,r}$ under the mapping $u\mapsto \mathcal{Q}_\ve^{-1}G(\lambda_\ve,V_\ve, P_\ve,u)$ is contained 
in $\mathcal{K}_{M,r}$. Also the restriction of this mapping to $\mathcal{K}_{M,r}$ is clearly compact. Thus there exists a fixed point of the mapping $u\mapsto \mathcal{Q}_\ve^{-1}G(\lambda_\ve,V_\ve, P_\ve,u)$ in $\mathcal{K}_{M,r}$. Since the principal part $0=\Phi_\beta(V)+\Phi_\beta(-V)-2\lambda$ and $2c_0V=\Phi_\beta(V)-\Phi_\beta(-V)$ of the system  \eqref{Reduced_system_tw}
is nondegenerate in the neighborhood of $V_0$ and $\lambda_0$, we have 
$V_\ve\to V_0$ and $\lambda_\ve\to \lambda_0$ as $\ve \to 0$.  	
\qed	
\end{proof}

\begin{remark}
Note that $V_0=0$ and $\lambda_0=\Phi_\beta(0)$ are always solutions of the principal part of the system \eqref{Reduced_system_tw}. Moreover one can establish a traveling wave (in fact standing wave) solution with $V_\ve$ 
equal to zero exactly, by following the line of Theorem \ref{Traveltheorem} 
but considering subspace of even functions $u\in H^1(\mathbb{R})$. 
The existence of nontrivial traveling waves (with nonzero velocities) is granted by
Theorem \ref{Traveltheorem} in the case when the equation $2c_0V=\Phi_\beta(V)-\Phi_\beta(-V)$ has a nonzero (nondegenerate) root. Such a solution does not exist for 
the standard potential $W(\rho)=\frac{1}{4}\rho^2(\rho-1)^2$, in this case $\Phi_\beta(V)=\Phi_\beta(-V)$ for all $V$ due to the fact that $\theta_0^\prime$ is an odd function. However, if the potential has two equally deep wells but possesses certain asymmetry, e. g. $W(\rho)=\frac{1}{4} (\rho^2+\rho^4)(\rho-1)^2$, we have $\Phi_\beta(V)>\Phi_\beta(-V)$ for $V>0$, so that nontrivial solutions of  $2c_0V=\Phi_\beta(V)-\Phi_\beta(-V)$ do exist for sufficiently large $\beta$, 
$\beta>\beta_{\rm critical}$. The plot of the function $\Phi_\beta(V)$
for $\beta=1$ and  $W(\rho)=\frac{1}{4}(\rho^2+\rho^4)(\rho-1)^2$ is depicted on Fig. \ref{AsymmetrictPotential}, as well as the corresponding standing wave.     
\end{remark}

\begin{figure}[]
	\centerline{
		\includegraphics[width=.45\textwidth]{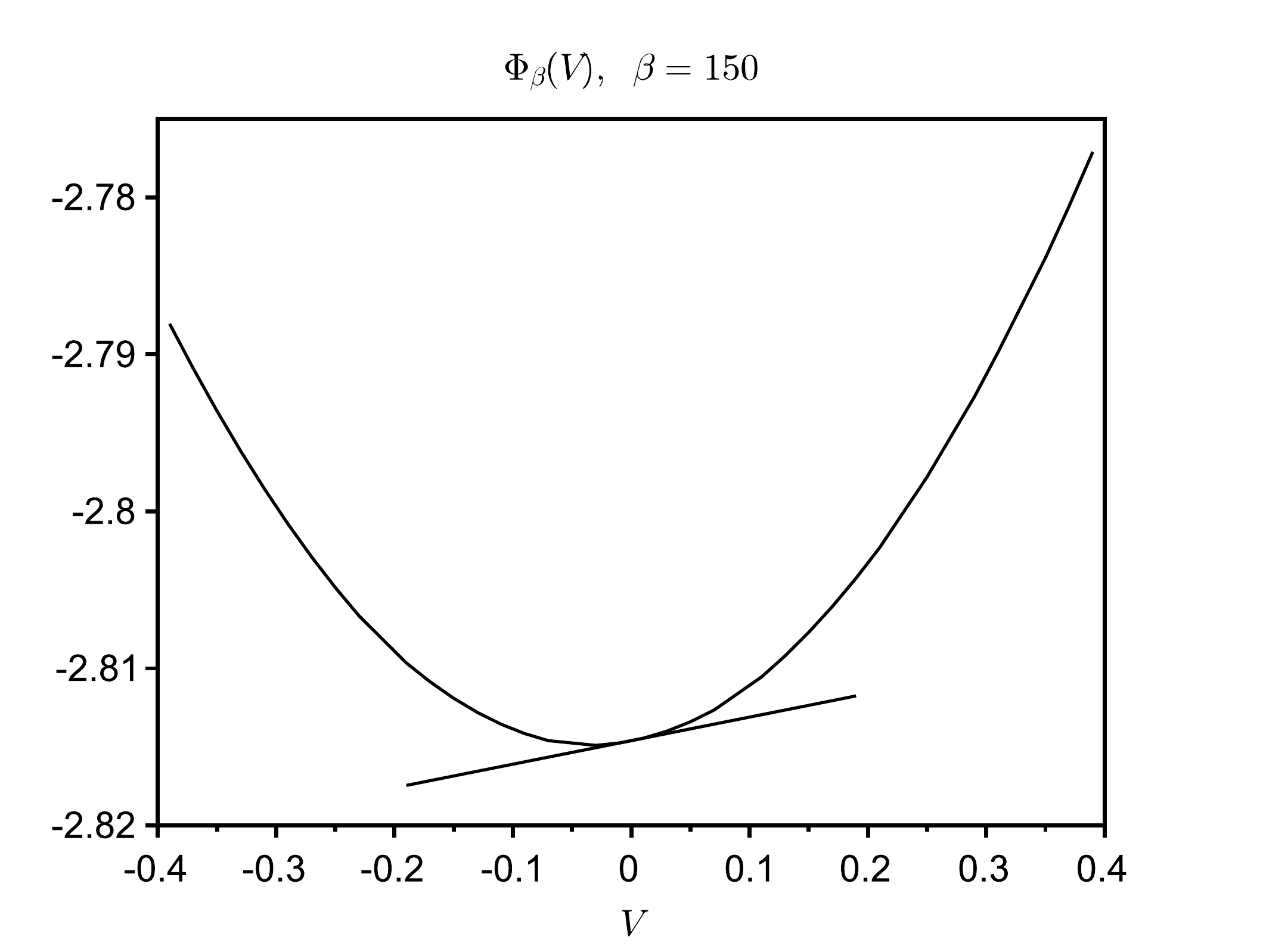}
		\includegraphics[width=.45\textwidth]{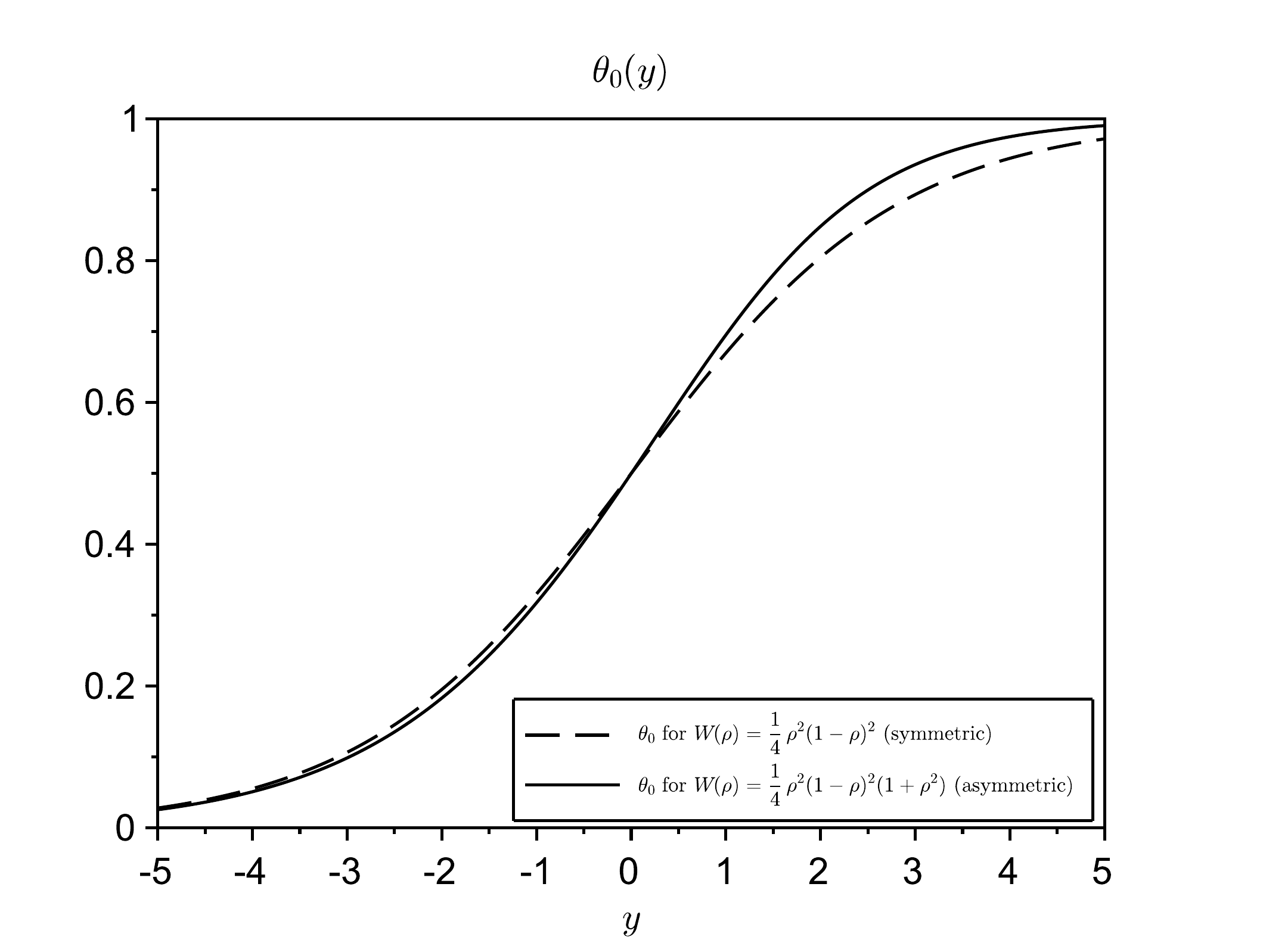}
		}
	\caption{{\it Left:} $\Phi_\beta(V)$ for $\beta=150$ and $W(\rho)=\frac{1}{4}\rho^2(1+\rho^2)(\rho-1)^2$, positive slope illustrates $\Phi'_\beta(0)>0$ ; {\it Right:} $\theta_0$,  standing wave  for the Allen-Cahn equation for $W=\frac{1}{4}\rho^2(\rho-1)^2$ (dashed) and $W(\rho)=\frac{1}{4}\rho^2(1+\rho^2)(\rho-1)^2$ (solid).} 
	\label{AsymmetrictPotential}
\end{figure}

\begin{remark}
	\label{KosoiPotential}
As already mentioned, nontrivial traveling waves appear in the case 
when $W(\rho)$ has certain asymmetry, that, in particular, makes the derivative $\Phi_\beta^\prime(V)$ of $\Phi_\beta(V)$ to be positive at $V=0$. The function $\Phi_\beta(V)$ depends on the potential $W(\rho)$ in a complex way. In order to have an idea 
about  this dependence assume that the diffusion coefficient in equation \eqref{eq2} for $P_\ve$ is given by $\delta\ve$, where 
$\delta$ is a positive parameter independent of $\ve$. This leads to redefining $\Phi_\beta(V)$ as follows,
$$
\Phi_\beta(V)=\int \chi(\theta_0^\prime)^2 dy, \quad  -\delta \partial_y^2 \psi-V\partial_y \psi+\psi=-\beta\theta_0^\prime.
$$
One can write down an asymptotic expansion of $\psi$ and its derivative $\psi_V$ with respect to $V$ at $V=0$ for sufficiently small $\delta>0$
$$
\psi=-\beta\theta_0^\prime-\delta\beta\theta_0^{\prime\prime\prime}
+\dots, \quad \psi_V=-\beta\theta_0^{\prime\prime}-2\delta\beta\theta_0^{(\rm{iv})}+\dots. 
$$
Then we have 
$$
\Phi^{\prime}_\beta(0)=-2\delta\beta\int \theta_0^{(\rm{iv})} (\theta_0^\prime)^2 dy+O(\beta\delta^2),
$$
which yields, after integrating by parts and using the relations $(\theta^\prime)^2=2W(\theta)$, $\theta^{\prime\prime}=W^\prime(\theta)$,
\begin{equation}\label{asymmetrmeasure}
\Phi^{\prime}_\beta(0)=\frac{8\sqrt{2}}{3}\delta\beta\int_0^1 W^{\prime\prime}(\rho)\,d W^{3/2}(\rho) +O(\beta\delta^2).
\end{equation}
The integral  in \eqref{asymmetrmeasure} can be interpreted as a measure of asymmetry of the potential $W(\rho)$, and nontrivial traveling waves emerge if this integral is positive and 
\begin{equation*}\beta>\beta_{\rm critical}=\frac{3 c_0}{8\sqrt{2}\delta\int_0^1 W^{\prime\prime}(\rho)\,d W^{3/2}(\rho) +O(\delta^2)}.
\end{equation*}
\end{remark}

\section {Sharp interface limit in 1D model problem}
\label{limit}

The equation of motion  \eqref{motion1} formally derived in Subsection \ref{formalderivation} exhibits qualitative changes for large values    
of the parameter $\beta$. This is indicated, in particular,  by the fact that the equation    
\begin{equation}\label{motion1_simple}
c_0V-\Phi_{\beta}(V)=-F,
\end{equation} 
may have multiple roots $V$. 
Note that combining the curvature and integral (constant) terms in \eqref{motion1} yields the equation of the form \eqref{motion1_simple} with 
$
F:=\frac{1}{|\Gamma|}\int_{\Gamma}\left(\kappa+\Phi_\beta(V)\right)ds-\kappa$.

In this Section we analyze a 1D analogue of the original model and rigorously derive a law of motion in the sharp interface limit. For given $F(t)\in C[0,T]$ we consider bounded solutions of the system   
\begin{empheq}[left=\empheqlbrace]{align}
\frac{\partial \rho_{\varepsilon}}{\partial t}&=\partial^2_{x}\rho_{\varepsilon}-\frac{W'(\rho_{\varepsilon})}{\varepsilon^2}-P_{\varepsilon}\partial_x\rho_{\varepsilon}+\frac{F(t)}{\varepsilon},\quad x\in\mathbb{R}^1,\;t>0,\label{P1}
\\
\frac{\partial P_{\varepsilon}}{\partial t}&=\varepsilon \partial_{x}^2P_{\varepsilon}-\frac{1}{\varepsilon}P_{\varepsilon}-\beta \partial_{x}\rho_{\varepsilon}.
\label{P2}
\end{empheq}


Analysis of the 1D problem \eqref{P1}-\eqref{P2} is a necessary step for understanding the original problem \eqref{eq1}-\eqref{eq2}.  
Observe that motion of the interface in the 2D system \eqref{eq1}-\eqref{eq2} occurs in the normal direction, and therefore it is essentially one-dimensional. Thus, the 1D model \eqref{P1}-\eqref{P2} is anticipated to capture the main features of 
\eqref{eq1}-\eqref{eq2}. 
The effects of curvature and mass conservation  in \eqref{motion1} 
are modeled by a given function $F(t)$.  
We believe that qualitative conclusions obtained for the 1D problem \eqref{P1}-\eqref{P2} apply for the 2D model \eqref{eq1}-\eqref{eq2}.

We study  the asymptotic behavior of solutions to the system \eqref{P1}-\eqref{P2} as $\ve\to 0$  with "well-prepared"  initial data 
for $\rho_\ve$,
\begin{equation}
\label{ic}
\rho_{\ve}(x,0)=\theta_0(x/\ve)+\ve v_{\ve}(x/\ve),
\end{equation}
where $\theta_0$
is a standing wave solution of the Allen-Cahn equation \eqref{theta_0z} 
such that $\theta_0(z)\to 0$ as $z\to-\infty$ and $\theta_0(z)\to 1$ 
as $z\to+\infty$. We seek $\rho_\ve$ in the form
\begin{equation}\label{form_rho_ve_1}
\rho_{\ve}(x,t)= \theta_0\left(\frac{x-x_{\ve}(t)}{\ve}\right)+\ve v_{\ve}\left(\frac{x-x_{\ve}(t)}{\ve},t\right).
\end{equation}
The 
$x_{\ve}(t)$ in \eqref{form_rho_ve_1} can be viewed as a location 
of the interface.
Remark \ref{remark:perturbed_interface} explains that a  choice of $x_\ve$ is not unique, however it is well defined in the limit $\ve\to 0$.


The main goal of this Section is to prove that $x_{\ve}(t)$ converges  as $\ve\to 0$ to $x_0(t)$, whose velocity $V_0(t)=\dot{x}_0(t)$ solves the sharp interface equation 
\begin{equation}\label{sharp_interface}
c_0V_0(t)=\Phi_{\beta}(V_0(t))-F(t), 
\end{equation}
where $\Phi(V)$ is the known nonlinear function given by \eqref{def_for_motion1}. This equation can be formally obtained in the limit $\ve \to 0$  as in the Section \ref{formalderivation}. 
%
%

Next for reader's convenience we summarize key steps of the asymptotic analysis of \eqref{P1}-\eqref{P2}:
\begin{itemize}	
	\item[(i)] {\it Choice of a special representation.} The 
	function $\rho_{\ve}$ is represented in the form 
	\begin{equation}
\label{repr_prev}
		\rho_{\ve}(x,t)=\theta_0(y)+\ve\chi_{\ve}(y,t)+\ve u_{\ve}(y,t),\ P_\ve(x,t)=Q_\ve(y,t), \;\;y=\frac{x-x_{\ve}(t)}{\ve},
	\end{equation}
	\noindent where $\theta_0$ and $\chi_\ve$ are known, and $u_{\ve}$, $Q_\ve$ are the new unknown functions. Existence of 
	$x_{\ve}(t)$ with estimates on $u_{\ve}$ uniform in $\ve$ and $t$ are established in Section \ref{section:reduction}. 
	\item[(ii)] {\it Reduction of the system to a single equation.}
	The unknown function $u_{\ve}$ is eliminated by showing that  the third term  in representation \eqref{repr_prev} is small. Next, we split $Q_{\ve}$ into two parts, $Q_{\ve}=A_{\ve}+B_{\ve}$, where $B_{\ve}$ depends on $u_{\ve}$ but is small, and $A_{\ve}$ does not depend on $u_\ve$. Thus, the original system \eqref{P1}-\eqref{P2}  is reduced to 
	\begin{empheq}[left=\empheqlbrace]{align}
	(c_0+o(1))V_\ve(t)&=\int (\theta_0')^2 A_{\ve}dy-F(t) +o(1),\label{eq_for_V_key_steps}\\
	\ve\frac{\partial A_{\ve}}{\partial t}&= \partial_y^2A_{\ve}+V_{\ve}(t)\partial_y A_{\ve}-A_{\ve}-\beta {\theta_0^\prime}.\label{eq_for_A_key_steps}
	\end{empheq}
	Taking the limit $\ve\to 0$ in the system \eqref{eq_for_V_key_steps}-\eqref{eq_for_A_key_steps} is non-trivial because of the product term $V_{\ve}(t)\partial_y A_{\ve}$.
	\item[(iii)] {\it Analysis of reduced problem.} For sufficiently small $\beta$ we prove that $x_{\ve}(t)\to x_0(t)$ as $\ve\to 0$  by the contraction mapping principle. For larger $\beta$, system  \eqref{eq_for_V_key_steps}-\eqref{eq_for_A_key_steps} further reduces to a singularly perturbed non-linear non-local equation. The limiting transition in this equation is based on the stability analysis of the semigroup generated by the linearized operator.    
\end{itemize}

\subsection{Asymptotic representation for $\rho_{\ve}$}
In order to pass to the limit $\ve \to 0$ in \eqref{P1}-\eqref{P2} we further specify $v_{\ve}$ in \eqref{form_rho_ve_1}. Namely,  we introduce the representation 
\begin{equation}\label{eq_form}
\rho_{\ve}(x,t)=\theta_0\left(\frac{x-x_{\ve}(t)}{\ve}\right)+\ve\chi_{\ve}\left(\frac{x-x_{\ve}(t)}{\ve},t\right)+\ve u_{\ve}\left(\frac{x-x_{\ve}(t)}{\ve},t\right),
\end{equation}
with the new unknown function $u_{\ve}$ satisfying 
\begin{equation}\label{ortogonality}
\int \theta_0'(y) u_{\ve}(y,t) dy=0,
\end{equation} and $\chi_{\ve}(y,t)$ defined by
\begin{equation}\nonumber
\chi_{\ve}(y,t)=\chi^-_{\ve}(t)+\theta_0(y)(\chi^+_{\ve}(t)-\chi^-_{\ve}(t)),
\end{equation}
where $\chi^+$ and $\chi^-$ are solutions of the following ODEs
\begin{eqnarray}\label{eq_for_psi_pm}
\ve^2
\partial_t \chi^+_{\ve}=-\frac{W'(1+\ve\chi_{\ve}^+)}{\ve}+F(t),\;\;\;
\ve^2\partial_t \chi_{\ve}^-=-\frac{W'(\ve\chi_{\ve}^-)}{\ve}+F(t)
\end{eqnarray}
with the initial data $\chi^+_{\ve}(0)=F(0)/W^{\prime\prime}(1)$ and 
$\chi^-_{\ve}(0)=F(0)/W^{\prime\prime}(0)$.

The 
idea
of the decomposition of the lower order term in \eqref{form_rho_ve_1} into two parts is 
suggested by the 
observation that it is the most important to control behavior of $\rho_\ve$ 
in the vicinity of the interface. So, 
ideally we would like to localize the analysis by considering functions that are negligibly small outside the interface. However, the right hand side $F(t)$ 
prevents $\rho_\ve$ from being localized. 
The function $\chi_\ve$ absorbs this nonlocal part of $\rho_\ve$: 
the new unknown function $u_\ve$ decays at infinity and, therefore, it allows one to work in Sobolev spaces on $\mathbb R$. 
Note that the standard ODE methods yield the following bounds
\begin{equation}\label{bounds_on_psi}
|\chi_{\ve}(y,t)|+|\partial_y\chi_{\ve}(y,t)|+|\partial_y^2\chi_{\ve}''(y,t)|\leq C \;\;\forall t\in[0,T],\;y\in{\mathbb R},
\end{equation}
moreover, thanks to the continuity of $F(t)$ and a particular choice of the initial values $\chi^{\pm}_\ve(0)$ we have
\begin{equation}
\label{partia_t_psi}
\ve^2\|\partial_t\chi_\ve\|_{L^\infty}\to 0\quad \text{uniformly} \ \text{on}\ [0,T]\ \text{as}\ \ve\to 0.
\end{equation}

Finally, we set $Q_\ve(y,t):=P_\ve(x_\ve+\ve y, t)$.

\begin{remark}
	\label{remark:perturbed_interface}
The choice of $x_\ve$ in the representation \eqref{form_rho_ve_1} is not unique, 
e.g. its perturbation with a term of order $\ve^2$ still leads to an 
expansion of the form  \eqref{form_rho_ve_1}. We introduced the additional orthogonality condition \eqref{ortogonality} which implicitly  
specifies $x_\ve(t)$. 
This  condition 
allows us to use 
Poincar\'{e} type inequalities (see \ref{appendix_poincare}) when deriving 
various bounds for $u_\ve$. If the initial value of $u_\ve$ in the expansion \eqref{eq_form}
does not satisfy \eqref{ortogonality}, it can be fixed by perturbing the initial value $x_\ve(0)=0$
with a higher order term. Indeed, this amounts to 
solving the equation
$$
\int \left(\theta_0(y+x_\ve(0)/\ve)-\theta_0(y)\right)\theta_0^\prime(y) dy =
\ve\int\left(\chi_\ve(y,0)-v_\ve(y+x_\ve(0)/\ve)\right)\theta_0^\prime(y) dy.
$$
If $\|v_\ve\|_{L^2}\leq C$ then the latter equation has a solution $x_\ve(0)$ and $|x_\ve(0)|=o(\ve)$. 
\end{remark}

\medskip

\subsection{Reduction of the system to a single equation} 
\label{section:reduction}

The following theorem justifies the expansions \eqref{eq_form} and will be used to obtain a reduced system for unknowns $x_\ve(t)$ and $Q_\ve(y,t)$ by eliminating $u_{\ve}$.

\begin{theorem}\label{ansatz_existence}
(Validation of representation \eqref{eq_form}-\eqref{ortogonality}) Let $\rho_{\ve}$ and $P_{\ve}$ be solutions of problem 
\eqref{P1}-\eqref{P2} with initial data $\rho_\ve(x,0)=\theta_0(x/\ve)+\ve v_\ve(x/\ve)$ 
and $P_{\ve}(x,0)=p_\ve (\frac{x}{\ve})$, where
	\begin{equation}
	\|v_\ve\|_{L^2}<C,\;\; \|v_\ve\|_{L^{\infty}}\leq C/\ve, \;\; 
	\end{equation} 
	and 
	\begin{equation}
	\|p_\ve\|_{L^2(\mathbb R)}+ 
	\| \partial_y p_\ve\|_{L^{2}}\leq C. 
	\end{equation}
	Then there exists $x_{\ve}(t)$ such that the expansion \eqref{eq_form}-\eqref{ortogonality}  holds with $\|u_{\ve}(\cdot,t)\|_{L^2}~\leq~C$ for $t\in [0,T]$.
\end{theorem}

\proof{$\;$}\\
\rm
\noindent STEP 1. ({\it coupled system for $u_\ve$, $Q_\ve$ and $V_\ve:=\dot x_\ve$}) Note that the maximum principle applied 
to \eqref{P1} yields $\|\rho_\ve\|_{L^\infty}\leq C$. This bound in conjunction 
with \eqref{bounds_on_psi} allow one to write down the expansion
$$
W^\prime\left(\theta_0+\ve(\chi_{\ve}+u_{\ve})\right)=
W^\prime(\theta_0+\ve\chi_{\ve})+\ve W^{\prime\prime}(\theta_0)u_\ve+
\ve^2W^{\prime\prime\prime}(\xi_\ve)\chi_\ve u_\ve+\frac{\ve^2}{2}
 W^{\prime\prime\prime}(\overline{\xi}_\ve) u_\ve^2,
$$
where $\xi_\ve$ and $\overline{\xi}_\ve$ are some bounded functions 
(while $\xi_\ve$ and $\overline{\xi}_\ve$ depend on $\theta_0$, $\chi_\ve$ and $u_\ve$, this dependence is omitted for brevity).
Then substituting  the expansion \eqref{eq_form} into equation
\eqref{P1} leads to 
\begin{equation}
\begin{aligned}
\ve^2\frac{\partial u_{\ve}}{\partial t}=&\partial_y^2 u_{\ve}-W^{\prime\prime}(\theta_0)u_{\ve}+V_\ve\theta_0^\prime-{Q_{\ve}}\theta_0^\prime+\partial^2_y\chi_{\ve}
+\frac{W^\prime(\theta_0)-W^\prime(\theta_0+\ve\chi_\ve)}{\ve}+F(t)-\ve^2\frac{\partial \chi_{\ve}}{\partial t}\\
&-\ve W^{\prime\prime\prime}(\xi_\ve)\chi_\ve u_\ve-\frac{\ve}{2}
W^{\prime\prime\prime}(\overline{\xi}_\ve)u_\ve^2
-\ve{Q_{\ve}}(\partial_y\chi_{\ve}+\partial_y u_{\ve})+\ve V_{\ve} (\partial_y\chi_{\ve}+\partial_y u_{\ve}).
\end{aligned}
\label{eq_for_u_ve_full}
\end{equation}
This equation is coupled with that for $Q_\ve$
\begin{equation}
\label{EqForPperepisano}
\ve\frac{\partial Q_{\ve}}{\partial t}=\partial_y^2 Q_{\ve}+V_\ve\partial_y Q_\ve-
 Q_{\ve}-\beta\theta_0^\prime-\ve\beta(\partial_y\chi_{\ve}+\partial_y u_{\ve}). 
\end{equation}
Finally, considering the solution $\rho_\ve$ as a given function we differentiate \eqref{eq_form} in time, multiply by $\theta_0^\prime(y)$ and integrate in $y$ over $\mathbb R$ to obtain the equation for $V_\ve$. Thanks to  \eqref{ortogonality} we get
\begin{equation}
\label{original_eq_for_interface}
V_\ve\left(c_0-\ve\int  (u_\ve+\chi_\ve) \theta_0^{\prime\prime} dy\right) =\ve^2\int \partial_t\chi_{\ve} \theta_0^\prime dy-
\ve\int \partial_t \rho_{\ve}(x_{\ve}(t)+\ve y,t)\theta_0^\prime dy.
\end{equation}
Note that if we obtain a uniform in $t$ a priori bound of the form $\|u_\ve\|_{L^2}\leq C$ with $C$ independent of $\ve$, \eqref{original_eq_for_interface} can be resolved with respect to $\dot x_\ve=V_\ve$ to come up with a well posed system \eqref{eq_for_u_ve_full}-\eqref{original_eq_for_interface}.

\medskip

\noindent STEP 2. ({\it energy estimates for
	$u_{\ve}$ and $Q_\ve$}) Represent $u_\ve$
as  $u_{\ve}=\theta_0^\prime w_{\ve}$, then multiply 
the equation \eqref{eq_for_u_ve_full} by $u_{\ve}$ and integrate in $y$ over $\mathbb R$. Since
$$
\int \left(-\partial_y^2 u_\ve+W^{\prime\prime}(\theta_0)u_{\ve}\right)u_\ve dy = \int (\theta_0^\prime)^2(\partial_y w_{\ve})^2 dy,\ \text{and}\  \int \theta_0^\prime u_{\ve} dy=0,\ 
\int \partial_y\chi_\ve u_{\ve} dy=0,\  \int \partial_y u_\ve u_{\ve} dy=0,
$$
we get
\begin{equation}
\begin{aligned}
\frac{\ve^2}{2} \frac{d}{dt}\int u_{\ve}^2dy+
\int (\theta_0')^2(\partial_y w_{\ve})^2 dy\leq &
\int \left(R_1- Q_\ve \theta_0^\prime -\ve Q_\ve \partial_y \chi_\ve\right) u_{\ve}dy \\
&-\ve \int Q_\ve \partial_y u_{\ve} u_\ve dy
+C\ve\int (u_{\ve}^2+|u_\ve|^3)dy,
\end{aligned}
\label{FirstEnBforu_ve}
\end{equation}
where  $R_1=\partial_y^2\chi_\ve +\dfrac{W^\prime(\theta_0)-W^\prime(\theta_0+\ve \chi_{\ve})}{\ve}-\ve^2\dfrac{\partial \chi_{\ve }}{\partial t}$.
Due to the construction of $\chi_\ve$ we have, $\|R_1\|_{L^2}\leq C$ with $C$ independent of $\ve$ and $t$. Also, by a  Poincar\'{e} type inequality {(see \ref{appendix_poincare})}
\begin{equation*}
\int (\theta_0^\prime)^2(\partial_y w_{\ve})^2 dy\geq C_{\theta_0}\|u_\ve\|_{H^1}^2 
\end{equation*}
 with $C_{\theta_0}>0$ independent of $u_\ve$. Thus \eqref{FirstEnBforu_ve} implies that
 \begin{equation}
 \begin{aligned}
 \frac{\ve^2}{2} \frac{d}{dt}\| u_{\ve}\|_{L^2}^2+\frac{C_{\theta_0}}{2}
 \|u_\ve\|_{H^1}^2&\leq C 
+C_1 \|Q_\ve\|^2_{L^2} 
 +\frac{\ve}{2} \int \partial_y Q_\ve  u_{\ve}^2  dy
 +C\ve\int|u_\ve|^3dy+\frac{C_{\theta_0}}{2}
 \left(\frac{\|u_\ve\|^2_{L^2}}{2}-\|u_\ve\|_{H^1}^2\right)\\
 &\leq C+C_1 \|Q_\ve\|^2_{L^2}+\ve\|\partial_y Q_\ve\|_{L^2}^2+C_2\ve\|u_\ve\|^6_{L^2}
 %
 \end{aligned}
 \label{FirstEnBforu_ve_bis}
 \end{equation}
 where we have also used the interpolation inequality $\int |u|^4dy\leq C\|u\|_{H^1}\|u\|_{L^2}^3$ which yields $\int |u|^4dy\leq C(\|u\|_{H^1}^2+\|u\|_{L^2}^6)$. Next we derive differential inequalities  
 \begin{equation}\label{dif_ineq_P}
 \ve\frac{d}{dt} \|Q_\ve\|^2_{L^2}+\|\partial_y Q_\ve\|_{L^2}^2+
  \|Q_\ve\|^2_{L^2}\leq C+C\ve^2\|u_\ve\|_{L^2}^2,
 \end{equation}
 \begin{equation}\label{dif_ineq_partial_y_P}
 \ve\frac{d}{dt} \|\partial_y Q_\ve\|^2_{L^2}+\|\partial_y^2 Q_\ve\|_{L^2}^2+
 \|\partial_y Q_\ve\|^2_{L^2}\leq C+C\ve^2\|u_\ve\|_{H^1}^2,
 \end{equation}
 by multiplying \eqref{EqForPperepisano} by $Q_\ve$ and $\partial^2_y Q_\ve$, 
 and integrating on $\mathbb{R}$.
 
 \medskip
 
 \noindent STEP 3. ({\it uniform  bound for $\|u_\ve\|_{L^2}$}) We show that differential inequalities \eqref{FirstEnBforu_ve_bis}-\eqref{dif_ineq_partial_y_P} imply  that $\|u_\ve\|^2_{L^2}$ remains uniformly bounded on $[0,T]$ when $\ve>0$
is small. To this end fix $M>\max\{1,\|u_\ve(\,\cdot\,,0)\|^2_{L^2}\}$, to be specified later, and consider the first time $t=\overline t\in (0,T)$ when $\|u_\ve(\,\cdot\,,t)\|^2_{L^2}$ reaches $M$ (if any). We have,
$\|u_\ve(\,\cdot\,,t)\|^2_{L^2}<M$ on $(0,\overline{t})$  and 
 \begin{equation}
 \label{good_proizvodnaya}
 \frac{d}{dt}\|u_\ve\|^2_{L^2}\geq 0 \quad \text{at}\ t=\overline{t}. 
 \end{equation}
 It follows from \eqref{dif_ineq_P} that 
$ \|Q_\ve\|^2_{L^2}\leq C+C\ve^2 M-\ve\frac{d}{dt} \|Q_\ve\|^2_{L^2}$; the same bound also holds for $\|\partial_y Q_\ve\|^2_{L^2}$. Substitute these bounds in \eqref{FirstEnBforu_ve_bis}
 and integrate from $0$ to $\overline{t}$ to conclude that 
 \begin{equation}
\int_0^{\overline{t}} \|u_\ve\|^2_{H^1} dt \leq C\left(\overline{t} +\ve^2 
\|u_\ve(\,\cdot\,,0)\|^2_{L^2}+\ve \|Q_\ve(\,\cdot\,,0)\|^2_{L^2}+\ve \overline{t} M^3)\right)
\label{FTFJVHGGH}
\end{equation}
with a constant $C$ independent of $\overline{t}$, M and $\ve$.  Now integrate 
\eqref{dif_ineq_partial_y_P}  from $0$ to $t$, in view of  \eqref{FTFJVHGGH} this results in the following pointwise inequality
$$
\|\partial_y Q_\ve\|^2_{L^2}\leq C\left(\frac{1}{\ve} +\ve^3 
 \|u_\ve(\,\cdot\,,0)\|^2_{L^2}+\ve^2 \|Q_\ve(\,\cdot\,,0)\|^2_{L^2}+\ve^2 M^3\right)+\|\partial_y Q_\ve(\,\cdot,,0)\|^2_{L^2} \quad \forall t \in(0,\overline{t}).
$$ 
Also, Gronwall's inequality applied to  \eqref{dif_ineq_P} yields 
$$
\| Q_\ve\|^2_{L^2}\leq C(1+\ve^2M) 
+ \|Q_\ve(\,\cdot\,,0)\|^2_{L^2} \quad \forall t \in(0,\overline{t}).
$$
We substitute the latter two bounds into  \eqref{FirstEnBforu_ve_bis} and consider the resulting inequality at $t=\overline{t}$. In view of \eqref{good_proizvodnaya} we have
$$
\|u\|_{L^2}^2(\,\cdot\,, \overline{t})\leq\|u(\,\cdot\,, \overline{t})\|_{H^1}^2\leq C(1+\|Q_\ve(\,\cdot\,,0)\|^2_{L^2}+\ve\|\partial_y Q_\ve(\,\cdot,,0)\|^2_{L^2}+\ve^4 
\|u_\ve(\,\cdot\,,0)\|^2_{L^2}+\ve M^3),
$$
where $C$ is independent of $\overline{t}$, M and $\ve$.
Thus, taking $M$ bigger than 
$$
\overline{M}=\max\{\|u_\ve(\,\cdot\,,0)\|^2_{L^2},C(1+\|Q_\ve(\,\cdot\,,0)\|^2_{L^2}+\ve\|\partial_y Q_\ve(\,\cdot,,0)\|^2_{L^2}+\ve^4 
\|u_\ve(\,\cdot\,,0)\|^2_{L^2}) \},
$$ 
e.g. $M:=2\overline{M}$, and considering sufficiently small $\ve>0$ 
we see that $\|u(\,\cdot\,, \overline{t})\|_{L^2}^2<M$. This shows 
that $\|u(\,\cdot\,, \overline{t})\|_{L^2}^2<M$ on $[0,T]$, and the Theorem is proved. $\qed$

Note that as a bi-product of the above  proof we obtained the integral bound
\begin{equation}
\label{integ_u_ve_H1}
\int_0^T\|u_\ve\|_{H^1}^2 d t\leq C,
\end{equation}
which plays an important role in the following derivation of a reduced system for $V_\ve$ and $Q_\ve$.  

The special form 
of the representation 
\eqref{eq_form}
(cf. \eqref{ortogonality}) together with estimates of Theorem \ref{ansatz_existence} and \eqref{integ_u_ve_H1} allow us 
to derive a system of the form 
\eqref{eq_for_V_key_steps}-\eqref{eq_for_A_key_steps} 
for $V_\ve$ and $Q_\ve$. To this end multiply \eqref{eq_for_u_ve_full} by $\theta_0^\prime(y)$ and integrate in $y$ over $\mathbb R$, this results in 
\begin{eqnarray}\label{interf_prelim}
\left(c_0+\ve\tilde{\mathcal{O}}_{\ve}(t)\right)V_{\ve}(t)-\int (\theta_0')^2A_{\ve}dy + F(t)=\ve\mathcal{O}_{\ve}(t)+\tilde {o}_\ve(t),
\end{eqnarray}
where $A_\ve$ is the solution of 
\begin{equation}\label{interf_prelim_A}
\ve\frac{\partial A_{\ve}}{\partial t}= \partial_y^2A_{\ve}+V_{\ve}(t)\partial_y A_{\ve}-A_{\ve}-\beta {\theta_0^\prime}
\end{equation}
 with  the initial condition $A_\ve(y,0)=p_\ve(y)$($=Q_\ve(y,0)$) and 
\begin{eqnarray}
\tilde{\mathcal{O}}_{\ve}(t)&:=&-\int (\chi_{\ve}+u_{\ve})\theta_0^{\prime\prime} dy, 
\nonumber
\\
\mathcal{O}_{\ve}(t)&:=&
\int \left(\frac{1}{2}W^{\prime\prime\prime}(\tilde{\xi}_\ve) \chi_\ve^2+
W^{\prime\prime\prime}(\xi_\ve)\chi_\ve u_{\ve}+\frac{1}{2}W^{\prime\prime\prime}(\overline{\xi}_\ve)u_{\ve}^2\right)
\theta_0^\prime dy \nonumber\\
&&+\frac{1}{\ve}\int(Q_\ve-A_\ve)(\theta_0^\prime)^2dy+\int Q_{\ve}\partial_y(\chi_{\ve}+u_{\ve})\theta_0^\prime dy,
\label{cal_O}\\
\nonumber
\tilde o_\ve(t)&:=&\ve^2\int \frac{\partial \chi_\ve}{\partial t}\theta_0^\prime dy
\end{eqnarray}
with $\tilde \xi_\ve$ being a bounded function (as well as ${\xi}_\ve$ and $\overline{\xi}_\ve$). It follows from \eqref{partia_t_psi} that $\tilde{o}_\ve$ uniformly converges to $0$  
as $\ve\to 0$ ($|\tilde{o}_\ve|\leq C\ve$ if $F$ is Lipschitz or $W^{\prime\prime}(0)=W^{\prime\prime}(1)$). Next we show that $\mathcal{O}_{\ve}(t)$ is bounded in $L^{\infty}(0,T)$ uniformly in $\ve$ .

\begin{proposition} Let 
	conditions of Theorem \ref{ansatz_existence} be satisfied, then $\mathcal{O}_{\ve}(t)$ introduced in \eqref{cal_O}  is  bounded uniformly in $t\in [0,T]$ and $\ve$. 
\end{proposition}
\proof $\;$\\
By Theorem  \ref{ansatz_existence} the first term in  
\eqref{cal_O} is bounded. To estimate the remaining terms 
represent  $Q_\ve$ as $Q_\ve=A_\ve+B_\ve$, where 
$B_\ve$ solves \begin{equation}
\label{Eq_for_B}
\ve\frac{\partial B_{\ve}}{\partial t}=\partial_y^2 B_{\ve}+V_\ve\partial_y B_\ve-
B_{\ve}-\ve\beta(\partial_y\chi_{\ve}+\partial_y u_{\ve})
\end{equation}
  with zero 
initial condition. Multiply this equation by $B_\ve$  and integrate on $\mathbb{R}$, then multiply   \eqref{Eq_for_B} by $\partial^2_y B_\ve$ and integrate on $\mathbb{R}$ to obtain 
\begin{equation}
\label{dif_ineq_B}
\ve\frac{d}{dt} \|B_\ve\|^2_{L^2} +\|B_\ve\|^2_{L^2}\leq C\ve^2(1+\|u_\ve\|_{L^2}^2),
\end{equation}
$$
\frac{d}{dt} \|\partial_y B_\ve\|^2_{L^2}\leq C\ve(1+\|u_\ve\|_{H^1}^2).
$$
After integrating these inequalities from $0$ to $t$ we make use of \eqref{integ_u_ve_H1} to derive
$\|B_\ve\|_{H^1}^2\leq C\ve$. Also, Gronwall's inequality applied to \eqref{dif_ineq_B}  yields $\|B_\ve\|_{L^2}^2\leq C\ve^2$. 
Similarly, in order to bound $\|A_\ve\|_{L^2}$ and $\|\partial_y A_\ve\|_{L^2}$ we first get 
\begin{equation*}
\ve\frac{d}{dt} (\|A_\ve\|^2_{L^2} +\|\partial_y A_\ve\|^2_{L^2})+(\|A_\ve\|^2_{L^2} +\|\partial_y A_\ve\|^2_{L^2})\leq C,
\end{equation*} 
then apply  Gronwall's inequality to conclude that  $\|A_\ve\|_{H^1}^2\leq C$. Thus,
\begin{equation*}
\begin{aligned}
\frac{1}{\ve}\int|Q_\ve-A_\ve|(\theta_0^\prime)^2dy+\left|\int Q_{\ve} \partial_y(\chi_{\ve}+u_{\ve}) \theta_0^\prime dy\right|&= 
\frac{1}{\ve}\int|B_\ve|(\theta_0^\prime)^2dy+\left|\int  (\chi_{\ve}+u_{\ve}) \partial_y (Q_{\ve}\theta_0^\prime) dy\right|\\
&\leq \frac{C}{\ve}\|B_\ve\|_{L^2}+C(1+\|u\|_{L^2})(\|A_\ve\|_{H^1}+\|B_\ve\|_{H^1})\leq C_1.
\end{aligned}
\end{equation*}
$\square$

From now on $\tilde{\mathcal{O}}_\ve$, $\tilde o_\ve$ and ${\mathcal{O}}_\ve$ 
are  regarded as given functions in the reduced system \eqref{interf_prelim}-\eqref{interf_prelim_A}, and their influence 
on the behavior of the system is small.  Observe that taking the formal limit as $\ve \to 0$ in  the system \eqref{interf_prelim}-\eqref{interf_prelim_A}
leads to 
\eqref{sharp_interface}.  Indeed, the formal limit as $\ve\to 0$  in \eqref{interf_prelim_A} is nothing but \eqref{psi_in_formal} whose unique solution is $\psi(y;V(t))$. Then substituting this function into the limit of  \eqref{interf_prelim} yields \eqref{sharp_interface}.



\subsection{Sharp Interface Limit for small $\beta$ by contraction mapping principle.}

The following Theorem establishes the sharp interface limit for sufficiently small $\beta$.  We assume that initial data $P_\ve(\ve y,0)=A_\ve(y,0)$ are bounded in $L^2(\mathbb{R})$ by a constant $C$ independent of $\ve$:
\begin{equation}\label{cond_on_A_ve}
\|A_{\ve}(\,\cdot\,,0)\|_{L^2}<C.
\end{equation} 
\begin{theorem} \label{theorem:contraction} (Sharp Interface Limit for subcritical $\beta$)
	Let $A_{\ve}$, $V_{\ve}$ be solution of the reduced system \eqref{interf_prelim}-\eqref{interf_prelim_A} with $\tilde{\mathcal{O}_{\ve}},\mathcal{O}_{\ve}\in L^{\infty}(0,T)$ and 
	$\tilde o_\ve$ converging to $0$ in $L^\infty(0,T)$ as $\ve\to 0$.
	Assume also that 
	 \eqref{cond_on_A_ve} holds. Then there exists $\beta_0>0$ (e.g., $\forall$ $0< \beta_0<2/\max\{\|(\theta_0^\prime)^2\|_{L^2}, \sqrt{c_0}\}$) such that for $0\leq\beta<\beta_0$
	\begin{equation}\label{v_converg}
	V_{\ve}(t)\rightarrow V_0(t) \text{ in }L^{\infty}(\delta,T)\text{ as }\ve \to 0,\;\;\forall\delta>0,
	\end{equation}
	where  $V_0$ is the unique solution of (\ref{sharp_interface}).     
\end{theorem}

\proof {$\;$}\\
STEP 1 ({\it Study of the boundary layer at $t=0$}). 
We show that the function $\eta_{\ve}(y,t)=A_\ve(y,t)-\psi(y,V_0(0) )$ behaves as a boundary layer at $t=0$. Since $\psi$ satisfies  $\partial_y^2\psi+V_0(0)\partial_y\psi-\psi=\beta\theta_0^\prime$,
$c_0V_0(0)=\int (\theta_0')^2 \psi dy-F(0)$ and $A_\ve$, $V_\ve$ solve \eqref{interf_prelim}-\eqref{interf_prelim_A}, we have 
\begin{equation}
\begin{aligned}
\ve \partial_t \eta_\ve =\partial_y^2\eta_{\ve}&+V_\ve\partial_y\eta_{\ve}-\eta_\ve+\frac{1}{c_0}
\partial_y\psi\int (\theta_0')^2 \eta_\ve dy\\
&+\frac{\partial_y \psi}{c_0+\ve\tilde{\mathcal{O}}_{\ve}}
\left(
F(0)(1+\ve\tilde{\mathcal{O}}_{\ve}/c_0)-F(t)-\ve\frac{\tilde{\mathcal{O}}_{\ve}}
{c_0} \int (\theta_0')^2A_\ve+\ve\mathcal{O}_{\ve}+\tilde {o}_\ve
\right).
\end{aligned}
\label{pde_for_eta}
\end{equation}
Multiply \eqref{pde_for_eta} by $\eta_\ve$ and integrate on $\mathbb{R}$,
\begin{equation*}
\frac{\ve}{2}\frac{d}{dt}\|\eta\|_{L^2}^2 +\|\partial_y\eta\|_{L^2}^2 +
	\|\eta\|_{L^2}^2 \leq \frac{1}{c_0} \|(\theta^\prime)^2\|_{L^2} \|\partial_y\psi\|_{L^2}
	\|\eta\|_{L^2} ^2+C
	\left(|F(0)-F(t)|+|\tilde o_\ve|+\ve\right)(1+\|\eta\|_{L^2}^2 ).
\end{equation*}
Note that $\|\partial_y \psi\|_{L^2}^2+\|\psi\|_{L^2}^2=-\beta\int\theta_0^\prime\psi dy$, therefore $\|\partial_y \psi\|_{L^2}^2\leq \beta^2\| \theta_0^\prime\|_{L^2}^2/4=\beta^2 c_0/4$. Thus, if $\beta \|(\theta^\prime)^2\|_{L^2}<2$, then for sufficiently small $\ve$ and $0<t<\sqrt{\ve}$ we have 
\begin{equation}
\label{energy_eq_for_eta}
\frac{\ve}{2}\frac{d}{dt}\|\eta\|_{L^2}^2 +\omega  \|\eta\|_{L^2} ^2\leq 
C\left(|F(0)-F(t)|+|\tilde o_\ve|+\ve\right)
\end{equation}  
with some $\omega>0$ independent of $\ve$. Now apply Gronwall's inequality
to  \eqref{energy_eq_for_eta} to obtain that 
\begin{equation*}
\| \eta_\ve\|_{L^2}^2  \leq Ce^{-2\omega t/\ve}+C\max_{\tau\in(0,t)}\left(|F(0)-F(t)|+|\tilde o_\ve|\right)+C\ve\quad
\forall  t\in[0,\sqrt{\ve}],
\end{equation*}  
in particular, 
\begin{equation}
\label{bound_on_bl}
\|A_\ve(\,\cdot\,,\sqrt{\ve})-\psi(\,\cdot\,,V_0(0) )\|_{L^2}\to 0\ \ \text{as}\ \ve\to 0.
\end{equation}
%
%
STEP 2 ({\it Resetting of \eqref{interf_prelim}-\eqref{interf_prelim_A} as a fixed point problem}). Consider an arbitrary $V\in L^\infty(\sqrt{\ve},T)$ and define $\mathcal{F}_{\ve}:L^{\infty}(\sqrt{\ve},T)\mapsto L^{\infty}(\sqrt{\ve},T)$ by 
\begin{equation}
\label{def_of_big_F}
\mathcal{F}_{\ve}(V):=\frac{1}{c_0+\ve \tilde{\mathcal{O}}_\ve}\left[ \int (\theta_0')^2 (A+\tilde{\eta}_\ve) dy-F(t)+\ve \mathcal{O}_\ve+\tilde o_\ve\right],
\end{equation}
where $A$ is the unique solution of 
\begin{empheq}[left=\empheqlbrace]{align}
\ve \partial_t A&=\partial^2_yA+V\partial_y A-A-\beta\theta_0^\prime,\label{def_of_big_f_pde_for_A}\\
A(y,\sqrt{\ve})&=\psi(y,V_0(0))\label{def_of_big_f_ic_for_A}
\end{empheq}
on $\mathbb{R}\times (\sqrt{\ve},T]$ and $\tilde{\eta}_\ve$ solves
\begin{empheq}[left=\empheqlbrace]{align}
\ve \partial_t \tilde{\eta}_\ve&=\partial^2_y\tilde{\eta}_\ve+V\partial_y \tilde{\eta}_\ve-\tilde{\eta}_\ve,\nonumber\\
\tilde{\eta}_\ve(y,\sqrt{\ve})&=A_\ve(y,\sqrt{\ve})-\psi(y,V_0(0)).\nonumber
\end{empheq}
Note that thanks to \eqref{bound_on_bl},
\begin{equation}
\label{unif_b_for_eta_tide_ve}
\max_{t\in[\sqrt{\ve},T]}\|\tilde{\eta}_\ve\|_{L^2}\to 0,\quad \text{as}\ \ve\to 0.
\end{equation}
%
%

It follows from the construction of $\mathcal{F}_\ve$ that $V_\ve$ is a fixed point of this mapping. Next we prove that, for sufficiently small $\beta$, $\mathcal{F}_\ve$ is a contraction mapping. Consider $V_1,V_2\in L^{\infty}(\sqrt{\ve},T)$ and let $A_1$, $A_2$ be solutions of \eqref{def_of_big_f_pde_for_A}-\eqref{def_of_big_f_ic_for_A} with $V=V_1$ and $V=V_2$, respectively. The function $\bar{A}:=A_1-A_2$ solves the following problem
\begin{empheq}[left=\empheqlbrace]{align}
\ve \partial_t \bar{A}&=\partial_y^2\bar{A}+V_1 \partial_y\bar{A}-\bar{A}+(V_1-V_2)\partial_y A_2,\label{contr_map_pde_for_bar_A}\\
\bar{A}(y,\sqrt{\ve})&=0.\label{contr_map_ic_for_bar_A}
\end{empheq}
Multiplying equation \eqref{contr_map_pde_for_bar_A} by $\bar{A}$ and integrating in $y$ we get 
\begin{equation}
\begin{aligned}
\frac{\ve}{2}\frac{\text{d}}{\text{d}t}\| \bar{A} \|^2_{L^2} +\| \bar{A} \|^2_{L^2} +\| \partial_y\bar{A} \|^2_{L^2} &=(V_1-V_2)\int  \bar{A}\partial_y A_2 dy
=(V_2-V_1)\int A_2 \partial_y \bar{A} dy
\\
&\leq |V_1-V_2|^2\frac{ \|A_2\|_{L^2}^2}{4}+ \|\partial_y \bar{A} \|_{L^2}^2.\label{contr_map_energy_eq_for_A_bar}
\end{aligned}
\end{equation}
On the other hand every solution solution $A$ of 
\eqref{def_of_big_f_pde_for_A}-\eqref{def_of_big_f_ic_for_A}, in particular $A_2$, satisfies  
\begin{equation}\label{contr_map_bound_on_A}
 \|A\|^2_{L^2} <c_0\beta^2, \;\;\;t\in [\sqrt{\ve},T].
\end{equation}
Indeed, multiplying  \eqref{def_of_big_f_pde_for_A} by $A$ and integrating in $y$ we get
\begin{equation*}
\ve\frac{d}{dt}\|A\|^2_{L^2}+2\|A\|^2_{L^2}+2\|\partial_y A\|^2_{L^2} =-2\beta \int \theta_0'A dy\leq c_0\beta^2+\|A\|^2_{L^2},
\end{equation*}
which yields $\ve\frac{d}{dt}\|A\|^2_{L^2}+\|A\|^2_{L^2}\leq c_0\beta^2$, the latter inequality in turn implies that 
$\|A\|^2_{L^2}\leq \|\psi\|^2_{L^2}e^{-t/\ve}+\beta^2c_0(1-e^{-t/\ve})$
for $t\in[\sqrt{\ve}, T]$. Observing that $\|\psi\|_{L^2}^2\leq\beta^2c_0$, we are led  to \eqref{contr_map_bound_on_A}.


Substitute now \eqref{contr_map_bound_on_A} in  \eqref{contr_map_energy_eq_for_A_bar} to conclude that
\begin{equation}
\|\mathcal{F}_{\ve}(V_1)-\mathcal{F}_\ve(V_2)\|^2_{L^{\infty}(\sqrt{\ve},T)} \leq \frac{c_0}{4}\beta^2 \|V_1-V_2\|^2_{L^{\infty}(\sqrt{\ve},T)}.
\end{equation}
Thus, for $\beta<2/\sqrt{c_0}$, $\mathcal{F}_{\ve}$ is a contraction mapping.

\medskip 

\noindent STEP 3.
Since  $V_\ve$ is a fixed point of the mapping $\mathcal{F}_{\ve}$, we have
\begin{equation*}
\begin{aligned}
	\|V_\ve -V_0\|_{L^\infty(\sqrt{\ve},T)}&= \|\mathcal{F}_{\ve}(V_{\ve})-\mathcal{F}_\ve (V_0)\|_{L^\infty(\sqrt{\ve},T)}
	+ \|\mathcal{F}_{\ve}(V_0)-V_0\|_{L^\infty(\sqrt{\ve},T)}
	\\
	&
	\leq \frac{\sqrt{c_0}}{2}\beta \|V_\ve -V_0\|_{L^\infty(\sqrt{\ve},T)}+ \|\mathcal{F}_{\ve}(V_0)-V_0\|_{L^\infty(\sqrt{\ve},T)}.
\end{aligned}
\end{equation*}
Thus, 
\begin{equation}
\label{contr_map_one_minus_q}
\|V_\ve -V_0\|_{L^\infty(\sqrt{\ve},T)}\leq \frac{1}{1-\sqrt{c_0}\beta/2} \|\mathcal{F}_{\ve}(V_0)-V_0\|_{L^\infty(\sqrt{\ve},T)}.
\end{equation}
It remains  to prove that  
\begin{equation}\label{contr_left_to_prove}
\text{$\|\mathcal{F}_{\ve}(V_0)-V_0\|_{L^\infty(\sqrt{\ve},T)}\to 0$ as $\ve\to 0$.}
\end{equation}

\medskip

\noindent STEP 4 ({\it Proof of \eqref{contr_left_to_prove}}).
First, we approximate $V_0(t)$, which can be a non-differentiable function, by a smooth function. Namely, construct $V_{0\ve}(t)\in C^1[0,T]$, e.g., as a mollification  of $V_0(t)$, such that 
\begin{equation}\label{def_V_ve_0}
\lim\limits_{\ve\to 0}V_{0\ve}= V_0\text{ in }C[0,T]\text{ and }\left|\frac{d}{dt}V_{0\ve}\right|<\frac{C}{\sqrt{\ve}}, \;\forall t\in [0,T].
\end{equation} 
%

Let $A$ be the solution of \eqref{def_of_big_f_pde_for_A}-\eqref{def_of_big_f_ic_for_A} with $V=V_0(t)$. Consider $D_\ve(y,t):=A(y,t)-\psi(y,V_{0\ve}(t))$, it 
satisfies the following equality 
\begin{equation}\label{Eq_for_D_ve}
\ve\partial_t D_\ve- \partial^2_y D_\ve- V_0 \partial_y D_\ve +D_\ve=-\ve  \frac{\partial \psi}{\partial V}(y;V_{0\ve}(t))\frac{d}{dt}V_{0\ve}+(V_0-V_{0\ve})\partial_y \psi(y,V_{0\ve}(t)) 
\end{equation}
on $\mathbb{R}\times (\sqrt{\ve},\infty)$.
Since the right hand side of \eqref{Eq_for_D_ve} converges to $0$ in $L^\infty([0,T], L^2(\mathbb{R}))$ and the norm of initial values $\|D_\ve(y,\sqrt{\ve})\|_{L^2}=
\|\psi(y,V_0(0))-\psi(y,V_{0\ve}(\sqrt{\ve}))\|_{L^2}\to 0$ as $\ve\to 0$, we have
\begin{equation}\label{bound_on_D_3}
\max\limits_{t\in[\sqrt{\ve},T]}\| D_\ve\|_{L^2}=0 \quad \text{when}\ \ve\to 0. 
\end{equation}
Finally, since 
$\int (\theta_0^\prime)^2\psi(y,V_{0\ve})dy=c_0V_0+F(t)+O(|V_{0\ve}-V_0|)$ we see that   
\begin{equation*}
|\mathcal{F}_{\ve}(V_0)-V_0|\leq
C(|V_{0\ve}-V_0|+\|D_\ve\|_{L^2}+\|\tilde \eta_\ve\|_{L^2}+|\tilde o_\ve|+\ve) 
\end{equation*}
Then combining  \eqref{unif_b_for_eta_tide_ve},\eqref{def_V_ve_0} and \eqref{bound_on_D_3} we establish \eqref{contr_left_to_prove}, and the Theorem is proved.  
$\square$

%
%

\subsection{Sharp Interface Limit  for arbitrary $\beta$ via stability analysis}
\label{arbitrary_beta}

\subsubsection{Reduction to a stability problem}

For larger $\beta$ the contraction principle no longer applies and both analysis and the results become  more complex. 
Here the stability analysis of the semigroup generated by a non-local non self-adjoint operator is used in place of the contraction mapping principle.

In the case where $\beta$ is not small, solutions of \eqref{sharp_interface} are no longer unique, see Fig. \ref{fig:Phi}. However, the original PDE problem  \eqref{P1}-\eqref{P2} (as well as the reduced system \eqref{interf_prelim}-\eqref{interf_prelim_A} has the unique solution. This indicates that analysis for large $\beta$ must be complemented by  a criterion of how to select the limiting solution of equation \eqref{sharp_interface} among all solutions of this equation.

\begin{figure}[]
	\begin{center}
		\includegraphics[width=0.45 \textwidth]{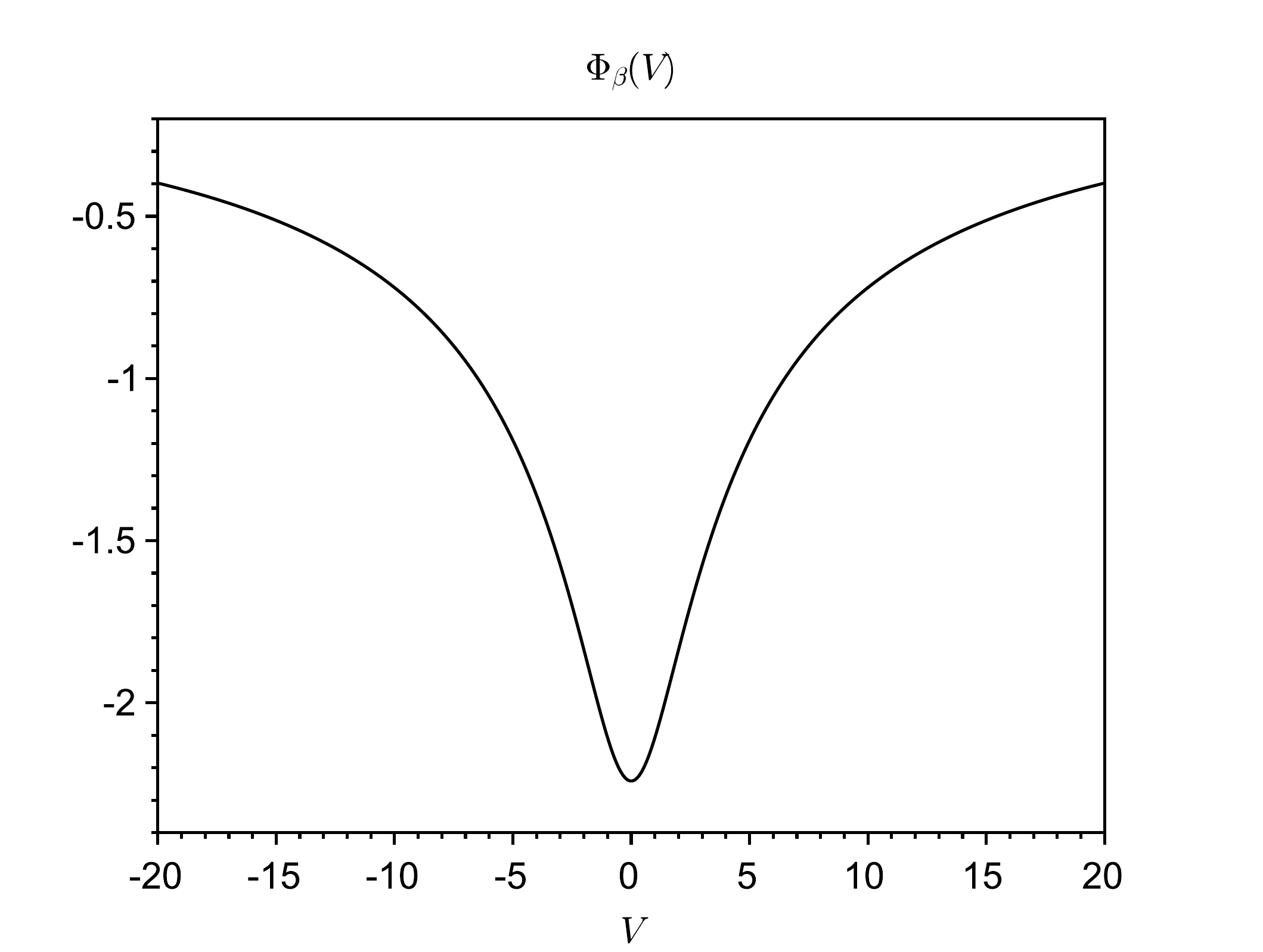}
		\includegraphics[width=0.45 \textwidth]{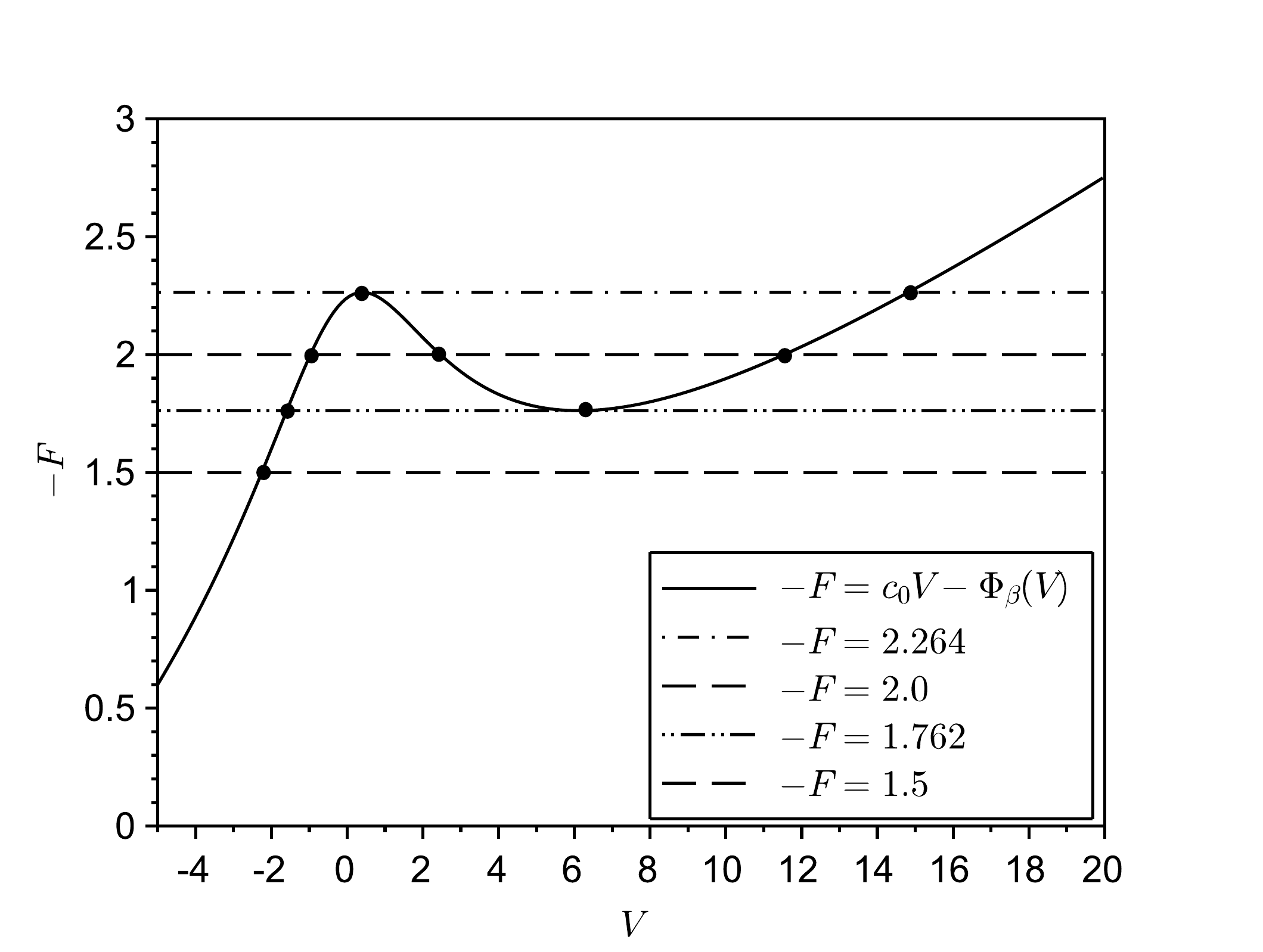}
		\caption{{\it Left}: Plot of function $\Phi_\beta(V)$ for $\beta =150>\beta_{\text{cr}}$; {\it Right}: Plot $c_0V- \Phi_\beta(V)$ for $\beta=150$ vs $F$. For $-F=1.5$ there is one intersection (\eqref{sharp_interface} has one root). For each $-F=1.762$ and $-F=2.264$ there are two intersections (\eqref{sharp_interface} has two roots). For $-F=2$ there are three intersections (\eqref{sharp_interface} has three roots).}
		\label{fig:Phi}
	\end{center}
\end{figure}


As a first step, 
we neglect terms $\ve\tilde{\mathcal{O}}_{\ve}(t)$, $\ve\mathcal{O}_{\ve}(t)$ and $\tilde o_\ve(t)$ in the reduced system \eqref{interf_prelim}-\eqref{interf_prelim_A} and study the system
\begin{empheq}[left=\empheqlbrace]{align}
c_0V_{\ve}(t)&=\int (\theta_0'(y))^2f_{\ve}(y,t) dy -F(t),\label{eq_for_V_ve_sys}\\
\ve \partial_t f_{\ve}&=\partial_y^2 f_{\ve}+V_{\ve}(t)\partial_y f_{\ve}-f_{\ve}-\beta\theta_0'
\label{eq_for_f_ve_sys}
\end{empheq}
(in \eqref{eq_for_V_ve_sys}-\eqref{eq_for_f_ve_sys}, $f_\ve$ replaces $A_\ve$ from \eqref{interf_prelim}-\eqref{interf_prelim_A}).
Substitute \eqref{eq_for_V_ve_sys} into \eqref{eq_for_f_ve_sys} to rewrite the \eqref{eq_for_V_ve_sys}-\eqref{eq_for_f_ve_sys} as a single equation
\begin{equation}\label{eq_for_f_ve}
\ve \partial_t f_{\ve}=\partial_y^2 f_{\ve}+\frac{1}{c_0}\left(\int (\theta_0')^2f_{\ve} dy-F(t)\right) \partial_yf_{\ve}-f_{\ve}-\beta\theta_0'.
\end{equation}
In the limit $\ve\to 0$ this equation (formally)  leads to the  PDE  
\begin{equation}\label{eq_for_f_0}
0=\partial_y^2 f_{\ve}+\frac{1}{c_0}\left(\int (\theta_0')^2f_{\ve} dy-F(t)\right) \partial_yf_{\ve}-f_{\ve}-\beta\theta_0'.
\end{equation}
Taking the formal limit is justified below for passing from \eqref{eq_for_f_ve} to \eqref{eq_for_f_0}. 

\begin{remark}
	\label{RiemarkAboutODES}
Equation \eqref{eq_for_f_ve} is a singular perturbation of \eqref{eq_for_f_0} and both equations are non-autonomous. It is well-known that singular limit problems, including non-autonomous equations, can be reduced to the analysis of large time behavior of autonomous equations. To illustrate this, recall a standard example of an ODE with a small parameter $\ve$ from \cite{Hol13},
\begin{equation}\label{ODE_from_Tik}
\ve \frac{dz_{\ve}}{dt}=\mathcal{F}(z_{\ve},t), \;t\in[0,T].
\end{equation}
Assume that there exists the unique root $\phi(t)$ of $\mathcal{F}$, i.e., $0=\mathcal{F}(\phi(t),t),\;t\in[0,T]$.
Then the singular limit
$\phi(t)=\lim\limits_{\ve\to 0}z_{\ve}(t)$ holds
provided that $\phi(t)$ is a stable root, i.e., all solutions $u(\tau)$ of an autonomous problem $\frac{du(\tau)}{d\tau}=\mathcal{F}(u(\tau),t)$ ($t$ is fixed) converge to the large time limit $\phi(t)$:
$
\lim\limits_{\tau\to \infty}u(\tau)=\phi(t).$
Note that the problem \eqref{ODE_from_Tik} has two time scales: a slow time $t$ and a fast time $\tau$. Also the large-time limit corresponds to $\tau \to \infty$ for a fixed parameter $t$.   

Note that the equivalence of singular and large-time limits is straightforward for the singularly perturbed  autonomous problems ($\mathcal{F}$ does not depend on $t$ in \eqref{ODE_from_Tik}). In this case, the simple rescaling \begin{equation*}\text{$\tau:=t/\ve$ $\;\;\;\;\;$ $u(\tau):=z_{\ve}(\ve \tau)$}\end{equation*} reduces the singular limit problem 
to a problem of stability of steady state.
\end{remark}

To justify the transition  from \eqref{eq_for_f_ve} to  \eqref{eq_for_f_0} we introduce three time scales: slow, fast, and intermediate.
More precisely, we employ the following three step procedure: 
(I) partition the interval $[0,T]$ by segments of length $\sqrt{\ve}$ on which the equation \eqref{eq_for_f_ve} is ``almost" autonomous ($F(t)$ is ``almost" constant  on each of these intervals); 
(II) on the first interval $(0,\sqrt{\ve})$, by appropriate scaling $\tau=t/\ve$ and stability analysis find large-time  asymptotics $\tau\to\infty$
(here we used equivalence of singular and large-time limits for autonomous equations);
(III) use the  asymptotics found in (II) as initial conditions for the next interval $(\sqrt{\ve},2\sqrt{\ve})$, repeat step (II) on this interval, and continue to obtain global asymptotics on $[0,T]$.       
A crucial ingredient here is an exponential stability of the linearized problem which 
prevents accumulating of errors (see bound\eqref{stab_eq_3} in Lemma \ref{l_stab}).   

\subsubsection{Spectral analysis of the linearized operator}


Rescale the ``fast" time $\tau=t/\ve $ in the unknown $f_\ve$ in \eqref{eq_for_f_ve}  and ``freeze" time $t$ in $F(t)$ (as described in step (II) above)
\begin{equation} \label{eq_for_f}
\partial_{\tau} f=\partial_y^2f+\frac{1}{c_0}\left(\int (\theta_0'(y))^2f(y,\tau) dy-F(t)\right)\partial_y f-f-\beta\theta_0',
\end{equation}
here $t\in[0,T]$ is considered to be a fixed parameter. Steady states of 
\eqref{eq_for_f} are solutions of \eqref{eq_for_f_0}. Let $f_0$ be such a 
solution, we define its velocity by 
\begin{equation}\label{velocity_Vs}
V_0:=\frac{1}{c_0}\left(\int (\theta_0'(y))^2f_0 dy-F(t)\right),
\end{equation}
then $f_0(y)=\psi(y,V_0)$, where $\psi(y;V)$ is defined in \eqref{psi_in_formal}.
Linearizing equation
\eqref{eq_for_f} around $f_{0}$ we obtain
\begin{equation}\label{eq_for_f_linearized}
\partial_\tau f +\mathcal{T}(V_0) f=0,
\end{equation}
where $\mathcal{T}(V):L^2(\mathbb R)\mapsto L^2(\mathbb R)$ is a
linear operator parameterized by $V\in \mathbb{R}$ and 
given by
\begin{equation}
\label{def_of_T}
\mathcal{T}(V) f:= -\partial_y^2 f -V \partial_y f + f -\frac{1}{c_0}\left(\int (\theta_0')^2 f dy \right)\partial_y\psi(y,V).
\end{equation}


Operator $\mathcal{T}(V)$ is a perturbation of a local operator $\mathcal{A}(V) f:= -\partial_y^2 f -V \partial_y f + f $ by a non-local rank one operator $\mathcal{P}(V)f= -\partial_y\psi(y,V)\frac{1}{c_0}(f,(\theta_0^\prime)^2)_{L^2}$, where 
$(\,\cdot\,,\,\cdot\,)_{L^2}$ stands for the standard inner product in $L^(\mathbb{R})$. 
The spectrum $\sigma(\mathcal{A}(V))$ of operator $\mathcal{A}(V)$ is described by the following straightforward proposition.
\begin{proposition}
	The spectrum $\sigma(\mathcal{A}(V))$ consists only of its essential part:
	\begin{equation*}
	\sigma(\mathcal{A}(V))=\sigma_{\text{ess}}(\mathcal{A}(V))=\left\{k^2-iVk +1;\, k\in \mathbb R\right\}.
	\end{equation*}
\end{proposition}
The spectrum $\sigma(\mathcal{T}(V))$ of  $\mathcal{T}(V)$ is described in
\begin{theorem}\label{spectral_theorem} (On spectrum of the linearized operator)
	Consider the part $\sigma_p(\mathcal{T}(V))$ of the spectrum $\sigma(\mathcal{T}(V))$ laying in $\mathbb C\backslash \sigma_{\text{ess}}(\mathcal{A})$. Then  $\sigma_p(\mathcal{T}(V))$ is given by 
	\begin{equation*}
	\sigma_p(\mathcal{T}(V))=\left\{\lambda\in \mathbb C\backslash \sigma_{\text{ess}}(\mathcal{A});\, \left(\,(\mathcal{A}(V)-\lambda)^{-1}\partial_y \psi,(\theta_0^\prime)^2\right)_{L^2}=c_0\right\}.
	\end{equation*}
	Moreover, all $\lambda$ from $\sigma_{\text{p}}(\mathcal{T}(V))$ are  eigenvalues with finite algebraic multiplicities, and geometric multiplicity one. 	
\end{theorem}

\begin{proof} We suppress dependence of $\mathcal{A},\mathcal{T}$, and $\eta$ on $V$  for brevity. Consider  $\lambda \notin \sigma_{\text{ess}}(\mathcal{A}) \cup\sigma_p(\mathcal{T})$ and $g\in L^2(\mathbb R)$. There exists the solution $f$ of 
\begin{equation*}
(\mathcal{A}-\lambda)f-\frac{1}{c_0}\partial_y \psi 
(f, (\theta_0^\prime )^2)_{L^2} = g,\ \ \text{or}\ 
f=\frac{1}{c_0}(\mathcal{A}-\lambda)^{-1} \partial_y\psi (f, (\theta_0^\prime )^2)_{L^2}+(\mathcal{A}-\lambda)^{-1}g
\end{equation*}
which can be represented as
\begin{equation}
\label{Resolvent_applied}
f=\frac{1}{c_0}(\mathcal{A}-\lambda)^{-1} \partial_y\psi (f, (\theta_0^\prime )^2)_{L^2}+(\mathcal{A}-\lambda)^{-1}g.
\end{equation}
Eliminate $(f, (\theta_0^\prime )^2)_{L^2}$ from the latter equation to find that 
\begin{equation*}
f=(\mathcal{A}-\lambda)^{-1}\partial_y\psi
\frac{\left((\mathcal{A}-\lambda)^{-1}g,(\theta_0^\prime )^2 \right)_{L^2}}{c_0-\left((\mathcal{A}-\lambda)^{-1} \partial_y\psi,(\theta_0^\prime )^2 \right)_{L^2}}
+(\mathcal{A}-\lambda)^{-1}g.
\end{equation*}
Thus, if $\lambda\notin\left\{\lambda\in \mathbb{C}; ((\mathcal{A}-\lambda)^{-1}\partial_y\psi,(\theta_0^\prime )^2)_{L^2}= c_0 \right\}\cup \sigma_{\text{ess}}(\mathcal{A})$, then 
$\lambda$ belongs to the resolvent set of $\mathcal{T}$. 

Now suppose that $\lambda\in \sigma_p(\mathcal{T}(V))$. Then 
$((\mathcal{A}-\lambda)^{-1}\partial_y\psi,(\theta_0^\prime )^2)_{L^2}= c_0$ and by Fredholm's theorem applied to  \eqref{Resolvent_applied},
$\lambda$ is an eigenvalue of finite multiplicity. Let $f$ be a corresponding eigenfunction, then by \eqref{Resolvent_applied}
\begin{equation*}
f=\frac{1}{c_0}(\mathcal{A}-\lambda)^{-1} \partial_y\psi (f, (\theta_0^\prime )^2)_{L^2}
\end{equation*}
Take the scalar product of this equality with $(\theta_0^\prime)^2$ to conclude that $\lambda \in \sigma_p(\mathcal{T})$ if and only if
\begin{equation}\label{point_spectrum}
\left((\mathcal{A}-\lambda)^{-1}\partial_y\psi,(\theta_0')^2\right)_{L^2}
=c_0, 
\end{equation}
and $(\mathcal{A}-\lambda)^{-1} \partial_y\psi$ is the unique (up to multiplication by a constant) eigenfunction. 
\end{proof}
{$\square$}

Thus, Theorem \ref{spectral_theorem} reduces the study of the part of the spectrum $\sigma_p(\mathcal{T})=\sigma(\mathcal{T})\setminus \sigma_{\text{ess}}(\mathcal{A})$ of operator $\mathcal{T}(V)$ to the equation \eqref{point_spectrum}. Next, using the obtained characterization of the $\sigma_p(\mathcal{T})$ we study the stability of  $\mathcal{T}$. 

\begin{proposition}\label{prop_unstable}
	If $\Phi^\prime_\beta(V)\geq c_0$, then there exists a real non
	positive eigenvalue $\lambda\in \sigma_p(\mathcal{T}(V))$. 
\end{proposition}
\begin{proof}
Consider the function $\zeta(\lambda):=
\left((\mathcal{A}(V)-\lambda)^{-1}\partial_y\psi,(\theta_0')^2\right)_{L^2}$ for real $\lambda\in (-\infty,0]$. We claim that $\zeta(0)=\Phi^\prime(V)$. Indeed, differentiate  \eqref{psi_in_formal}
to find that 
\begin{equation*}
-\partial_y^2\psi_{V}-V\partial_y\psi_{V}+\psi_V =\partial_y \psi,
\end{equation*}
where $\psi_{V}$ denotes the partial derivative of $\psi$ in $V$. Thus
$$
\zeta(0)=\left([\mathcal{A}(V)]^{-1}\partial_y\psi,(\theta_0')^2\right)_{L^2}
=\left(\psi_V,(\theta_0')^2\right)_{L^2}=\Phi^\prime_\beta(V).
$$
On the other hand it is easy to see that $\zeta(\lambda)\to 0$ as 
$\lambda\to-\infty$, consequently $\zeta(\lambda)=c_0$ for 
some $\lambda\in (-\infty,0]$. By Theorem  \ref{spectral_theorem}
this $\lambda$ is a non positive eigenvalue of $\mathcal{T}(V)$.
\end{proof}{$\square$}

\medskip 

\begin{definition} Define the {\it set of stable velocities} $\mathcal{S}$ by
	\begin{equation}\label{def_of_sd}
	\mathcal{S}:=\left\{V\in \mathbb R: \forall \lambda \in\sigma (\mathcal{T}(V))\ \text{has positive real part}\,
		\right\},
	\end{equation}
	where $\mathcal{T}(V)$ is the linearized operator given by \eqref{def_of_T}.
\end{definition}

\begin{remark}
	\label{CircularShapeUnstable}
In the case of 2D sytem \eqref{eq1}-\eqref{lagrange} one can expect (yet to be proved) that
there exist standing wave solutions with circular symmetry when $\Omega$ is a disk.  However our preliminary reasonings show that these solutions are not stable if $\Phi_\beta^\prime(0)>c_0$ 
(this latter inequality holds for asymmetric potentials $W(\rho)$ and sufficiently large $\beta$).
This conjecture originates from the fact that zero velocity and its small perturbations does not belong to the set of stable velocities as shown in Proposition \ref{prop_unstable}.  	
\end{remark}

Proposition \ref{prop_unstable} implies that the inequality 
\begin{equation}\label{T_is_stable}
\Phi'_\beta(V)< c_0
\end{equation}
is a necessary condition for stability of $V$. We hypothesize that \eqref{T_is_stable} is also a sufficient condition, and therefore \eqref{T_is_stable} describes the set $\mathcal{S}$, that is, 
\begin{equation}\label{S_is_set}
\mathcal{S}=\left\{V\in\mathbb R: \;\Phi'_\beta(V)< c_0\right\}.
\end{equation}

To support our hypothesis we consider $W(\rho)=\frac14\rho^2(\rho-1)^2$. In this case, the set $\left\{V\in\mathbb R: \;\Phi'_\beta(V)< c_0\right\}$ is the complement to the open interval $(V_{\text{min}}, V_{\text{max}})$, where $V_{\text{min}}$ and $V_{\text{max}}$ are the local maximum and minimum, respectively (see  Fig. \ref{fig:Phi} and the sketch of $c_0V-\Phi_\beta(V)$ in Fig. \ref{fig:two_branches}). Numerical simulations clearly show that \eqref{S_is_set} holds. We can also rigorously prove that there exist such $V_1$ and $V_2$ that the set of stable velocities $\mathcal{S}$ is non-empty and, moreover, contains the compliment to the open interval $(V_1,V_2)$. This is done by means of Fourier analysis which allows us to rewrite \eqref{point_spectrum} as an integral equation for a complex number $\lambda$. 
Details are relegated to 
\ref{appendix_integral_equation}.


\subsubsection{Main result for 1D interface limit}

In this subsection we formulate the main result on the 1D sharp interface limit in the  system \eqref{P1}-\eqref{P2}  for arbitrary $\beta>0$.
 Introduce the following conditions:
\begin{list}{}{}
	
	\item{\bf (C1)} 
	Let $V_0 \in \mathcal{S}$ solve $c_0V_0- \Phi_\beta(V_0)=F(0)$  and let  $[0,T_\star]$ be a time interval such that there exists $V(t)\in \mathcal{S}$  a continuous  solution of \begin{equation}\label{117}
	c_0V(t)-\Phi_\beta(V(t))=-F(t),\;\;t\in[0,T_\star],\; V(0)=V_0.
	\end{equation}

	\item{\bf (C2)} 
	Assume that $P_{\ve}(x, 0)=p_{\ve}(y/\ve)$ and 
	$\|p_\ve(\, \cdot\, )-\psi(\,\cdot\, ,V_0)\|_{L^2}< \delta$ with a small constant $\delta>0$ independent of $\ve$ (the function $\psi=\psi(y;V_0)$ is defined by \eqref{psi_in_formal}).
	
\end{list}

\medskip
\begin{theorem}\label{theorem:sharp_interface}
	{(\it Sharp Interface Limit for all $\beta$)} Let $x_{\ve}$ be as in Theorem \ref{ansatz_existence} and assume that conditions ({\bf C1}) and ({\bf C2}) hold along with the conditions of Theorem \ref{ansatz_existence}. Then $x_{\ve}(t)$ converges to $x_0(t)$ in $C^1[0,T_\star]$, where  $V(t):= \dot{x}_{0}(t)$  is the solution of  \eqref{117} as  defined in
	{\bf (C1)}.
\end{theorem}

Theorem \ref{theorem:sharp_interface} justifies the sharp interface equation \eqref{117} for any $\beta$. Its proof consists of two steps: (i) reduction to a  single equation (nonlinear, singularly perturbed) which is done in Section \ref{section:reduction} and (ii) passage to the limit in this equation based on stability analysis presented below, which is the main ingredient of the proof.

\begin{proof} {\it of Theorem \ref{theorem:sharp_interface}.}
	Rewrite \eqref{interf_prelim}-\eqref{interf_prelim_A} in the form of  the single PDE
	\begin{equation}\label{stab_eq_1}
	\ve \frac{\partial A_\ve}{\partial t}=\partial_y^2A_\ve+\frac{1}{c_0+\ve\tilde{\mathcal{O}}_{\ve}(t)}\left(\int (\theta_0')^2A_{\ve}dy -F(t)+\ve\mathcal{O}_{\ve}(t)+\tilde {o}_\ve(t)\right)\partial_y A_{\ve}-A_\ve -\beta\theta_0',
	\end{equation}
Recall that $\tilde{\mathcal{O}}_{\ve}(t)$ and ${\mathcal{O}}_{\ve}(t)$ are uniformly bounded functions, $\tilde {o}_\ve(t)$ tends to $0$ uniformly on $[0,T_{\ast}]$ as $\ve\to 0$.
	We next pass to the limit in equation \eqref{stab_eq_1} using exponential stability (established in \eqref{exp_stable_semigroup}) of the semigroup corresponding to the linearized operator. The following local stability result plays the crucial role in the proof.
	
	\begin{lemma}
		\label{l_stab} 
		There exist $\omega>0$ and $\delta>0$ such that if 
		\begin{equation}
		\|A_\ve(\,\cdot\,,t)-\psi(\,\cdot\,,V(t))\|_{L^2}\leq \delta, 
		\end{equation}
then for any $0<r<1$ and sufficiently small $\ve$, $\ve<\ve_0 (T_\ast)$, the function $\eta_\ve(y,t,\tau)=A_{\ve}(y,t+\ve \tau)-\psi(y, V(t))$ satisfies
\begin{equation}
\label{stab_eq_2}
\|\eta_\ve(\,\cdot\,t,\tau)\|^2_{L^2}\leq C\left(e^{-\frac{\omega}{2}\tau}\|\eta_\ve(\,\cdot\,t,0)\|^2_{L^2}+
\max_{s\in [t,t+\ve \tau]}\left(|F(t)-F(s)|^2+\tilde o_\ve^2(s)\right) +\ve^2\right)
		\end{equation}
for $0\leq \tau \leq \frac{1}{\ve^r}$. The constants $\omega, \delta$ and $C$ in 
	are independent of  $t$, $\tau$ and $\ve$.  
	\end{lemma}
	
	This Lemma shows that if the initial data are at distance at most $\delta$ from $\psi$ (in the $L^2$-norm), then the solution $A_\ve(y,t+\ve \tau)$ approaches $\psi$ exponentially fast in $\tau$ (first term in the RHS of \eqref{stab_eq_2}) with  a deviation that is bounded from above independently of $t$ (described by the second and the third terms in the RHS of \eqref{stab_eq_2}).   
	The conclusion of Theorem \ref{theorem:sharp_interface} immediately follows from this Lemma. Indeed, consider the time interval $(0,t_1)$, $t_1:=\sqrt{\ve}$. Then by Lemma \ref{l_stab} we obtain 
	\begin{equation}
	\label{stab_eq_3}
	\| A_{\ve}(\,\cdot\,,t_1)-\psi(\,\cdot\,, V(t_1))\|_{L^2}^2\leq C \left(e^{-\frac{\omega}{2\sqrt{\ve}}}\delta+m^2(\sqrt{\ve}) +\max_{s\in [0,T_{\ast}]}\tilde o_\ve^2(s)+\ve^2 \right)+C_1\ve,
	\end{equation}
	where $m$ denotes the modulus of continuity  of $F$ on $[0,T_\star]$. Choose $\ve$ small enough so that $\log\frac{1}{\ve}\leq \frac{\omega}{2}\sqrt{\ve}$ and the right hand side  of \eqref{stab_eq_3} is bounded by $\delta$.
Similarly, for intervals $(t_1,t_2)$, where $t_2:=2\sqrt{\ve}$, $(t_2,t_3$, where $t_3:=3\sqrt{\ve}$, etc., we obtain
	\begin{equation*}
	\|A_{\ve}(\,\cdot\,,t_i)-\psi(\,\cdot\,,V(t_i))\|^2_{L^2}\leq C\left(\ve +m^2(\sqrt{\ve})+\max_{s\in [0,T_{\ast}]}\tilde o_\ve^2(s)+\ve^2\right)+C_1\ve<\delta.
	\end{equation*}
	To complete the proof of Theorem \ref{theorem:sharp_interface} we again use Lemma \ref{l_stab} to bound $\|A_{\ve}(\,\cdot\,,t)-\psi(\,\cdot\,,V(t))\|^2_{L^2}$ for $t\in (t_i,t_{i+1})$, $i=1,2\dots$. $\square$
	
\medskip
	
	\noindent{\it Proof of Lemma \ref{l_stab}.}
	As in the first step of the proof of Theorem \ref{theorem:contraction}, consider the function $\eta_{\ve}(y,\tau):=A_{\ve}(y,t+\ve\tau)-\psi(y,V(t))$, hereafter $t$ is considered as a fixed parameter. It follows from \eqref{stab_eq_1} and \eqref{psi_in_formal} that  $\eta_{\ve}$ satisfies the following PDE 
\begin{equation}
	\frac{\partial \eta_{\ve}}{\partial \tau}+\mathcal{T}\eta_{\ve}=\frac{\partial_y\eta_\ve}{c_0+\tilde O_\ve} \int(\theta_0^\prime)^2\eta_\ve dy +
	\frac{\Lambda_\ve}{c_0+\tilde O_\ve}\partial_y\eta_\ve-\frac{\ve \tilde O_\ve}{c_0(c_0+\tilde O_\ve)}\partial_y\psi \int(\theta_0^\prime)^2\eta_\ve dy
	+\frac{\Lambda_\ve}{c_0+\tilde O_\ve}\partial_y\psi,
\label{DLINNOE_URAVNENIE}	
\end{equation}
where 
$$
\Lambda_\ve(t,\tau):=F(t)-F(t+\ve\tau)+\ve O_\ve(t+\ve\tau)+\tilde o_\ve(t+\ve\tau) +\ve\frac{\tilde O_\ve(t+\ve\tau)}{c_0}\left(\int(\theta_0^\prime)^2\psi(y,V(t)) dy\right).
$$
Introduce the semigroup operator $e^{-\mathcal{T}\tau},\tau>0$ in $L^2(\mathbb R)$, then by Duhamel's principle
	\begin{equation}\label{semigroup_representation}
	\eta_\ve (\,\cdot\, , \tau )=e^{-\mathcal{T}\tau}\eta_\ve(\,\cdot\,, 0) +\int_0^{\tau} e^{-\mathcal{T}(\tau - \tau')} R_\ve (\,\cdot\,,\tau')d\tau',
	\end{equation}
where $R_\ve (y,\tau)$ denotes the right hand side of \eqref{DLINNOE_URAVNENIE}.	

In order to proceed with the proof of Lemma \ref{l_stab} we first prove exponential stability of the semigroup $e^{-\mathcal{T}t}$ and establish its consequences in the following 
	
\begin{lemma}
		\label{exp_stability} 
		There exists $\omega>0$ such that 
\begin{itemize}
	\item[{\rm(i)\ }] the following inequality holds
	\begin{equation}
	\label{exp_stable_semigroup}
	\|e^{-\mathcal{T}\tau}\|\leq Me^{-\omega \tau}, \;\;\tau\geq 0,
	\end{equation} 
	where $\|e^{-\mathcal{T}\tau}\|$ stands for the operator norm of $e^{-\mathcal{T}\tau}$ in $L^2(\mathbb R)$;
	\item[{\rm(ii)}] for every $g(y,t)$,
	\begin{equation}
	\label{stabb_ineq}
	\Bigl\|\int\limits_0^{\tau} e^{-\mathcal{T}(\tau-\tau')}\frac{\partial^k g}{\partial y^k}(\,\cdot\,,\tau')d\tau'\Bigr\|^2_{L^2}
	\leq C \int_0^{\tau} e^{-\omega(\tau-\tau')}\|g(\,\cdot\,,\tau')\|^2_{L^2}d\tau', \quad k=0,1
	\end{equation}
	with a constant $C$ independent of $g$.
\end{itemize}		
		 Moreover,  constants $\omega$, $M$ and $C$ can be chosen independently of $t$ (recall that $\mathcal{T}=\mathcal{T}(V(t))$ depends on $t$). 
	\end{lemma}
\begin{proof} {\it of Lemma \ref{exp_stability}.} 
\noindent STEP 1 ({\it proof of (i)}).
		For every fixed $V\in \mathcal{S}$, it follows from Gerhardt-Pr\'{u}ss theorem (see, e.g., \cite{Gea78,Pru84}) that 
		\eqref{exp_stable_semigroup} holds with some constants 
		$M$ and $\omega>0$. However, for later use we need a stronger result, we prove that these constants can be chosen 
		independently of $V=V(t)$ for $t\in[0,T_\ast]$. To this end we establish the following bound
		\begin{equation}
		\label{sector_bound}
		\|(\mathcal{T}(V(t))-\lambda-\omega)^{-1}\|\leq \frac{C}{|\lambda|} \ \text{for}\ \lambda\in \Pi_{\varphi_0}:=
		\{-re^{i\varphi}\,;\, |\varphi|\leq \pi/2+\varphi_0,r> 0 \}, 
		\end{equation}
		with constants $\omega>0$, $\varphi_0>0$ and $C$ all independent of $t\in [0,T_\ast]$. 
		Then Theorem I.7.7 from \cite{Paz83} yields the inequality $\|e^{-(\mathcal{T}(V(t))-\omega)\tau}\|\leq M$ for $\tau>0$
		with constants $\omega>0$ and  $M$ independent of 
		$t$, and this latter inequality is equivalent to  \eqref{exp_stable_semigroup}.
		
Set $\mathcal{T}^\prime(t,\omega):=\mathcal{T}(V(t))-\omega$ and  $\mathcal{A}^\prime(t,\omega):=\mathcal{A}(V(t))-\omega$. 
To prove \eqref{sector_bound} we first derive by Fourier analysis,
\begin{equation}
		\label{sector_bound_for_A}
		\|(\mathcal{A}^\prime(t,\omega)-\lambda)^{-1}\| \leq \max_{k\in\mathbb{R}} \frac{1}{|k^2-iV(t)k+1-
			\lambda-\omega|}\leq \frac{C}{|\lambda|+1}
		\ \text{for}\ \lambda\in \Pi_{\overline{\varphi}}, 
\end{equation}
		where $\overline{\varphi}=\frac{1}{2}\arctan \frac{1}{\max_t |V(t)|}$, constant $C$ is independent of 
		both $t\in [0,T_\ast]$ and $0\leq\omega<1/2$. Next we make use of the 
representation (cf. Theorem \ref{spectral_theorem}) 
		\begin{equation}
		\begin{aligned}
		(\mathcal{T}^\prime(t,\omega)-\lambda)^{-1} \, \cdot\, 
		=\frac{((\mathcal{A}^\prime(t,\omega)-\lambda)^{-1}\, \cdot\, ,(\theta_0^\prime)^2)_{L^2}}
		{\mu(\lambda;t,\omega)}
		(\mathcal{A}^\prime(t,\omega)-\lambda)^{-1}\partial_y \psi
		+ (\mathcal{A}^\prime(t,\omega)-\lambda)^{-1}\cdot\ ,
		\end{aligned}
		\label{repr_resolv_mu}
		\end{equation}
		where $\mu(\lambda;t,\omega)=c_0-
	\left((\mathcal{A}^\prime(t,\omega)-\lambda)^{-1}\partial_y \psi,(\theta_0^\prime)^2\right)_{L^2}$.

		It follows  from \eqref{sector_bound_for_A}
		that the family of holomorphic functions $\mu(\cdot;t,\omega):\Pi_{\overline{\varphi}}\to \mathbb C$ satisfies $|\mu|>1/2$ everywhere but on a fixed bounded subset $K$ of $\Pi_{\overline{\varphi}}$ which is independent of $0\leq \omega\leq 1/2$ and $t\in [0,T_{\ast}]$. 
		On the other hand the functions $\mu(\lambda;t,\omega)$ 
		are uniformly bounded  in $\{\lambda\in\mathbb{C};\, {\rm Re} \lambda<1/4\}$ and they depend continuously on $t$ and $\omega$.
		%
		Now taking into account the fact that $V(t)\in \mathcal{S}$ for all $t\in [0,T_{\ast}]$ we show that $|\mu(\lambda;t,0)|\geq \mu_0$ when $\lambda \in K$ and $|\text{Re}\lambda|\leq 2\omega$ for some $1/2\geq\mu_0>0$ and $1/2\geq \omega>0$. Indeed, otherwise there is a sequence $t_k\to t_0$, $\lambda_k\to \lambda_0$ such that $\text{Re}\lambda_0=0$ and $\mu(\lambda_k;t_k,0)\to 0$. Then, by Montel's theorem, up to extracting a subsequence $\mu(\lambda_k;t_k,0)\to\mu(\lambda_0;t_0,0)$, but  $\mu(\lambda_0;t_0,0)\neq 0$ as $V(t_0)\in \mathcal{S}$ (cf. proof of Theorem \ref{spectral_theorem}).  
		Thus there are  $\varphi_0>0$ 
		($\varphi_0\leq\overline{\varphi}$)
		such that $|\mu(\lambda;t,\omega)|\geq \mu_0$ for $\lambda\in \Pi_{\varphi_0}$. 
		Using this fact and inequality  \eqref{sector_bound_for_A} to bound terms in 
		in \eqref{repr_resolv_mu}  we get \eqref{sector_bound}, and therefore \eqref{exp_stable_semigroup} holds for some $\omega>0$ 
		and $M$, both being independent of $t$. This result immediately yields \eqref{stabb_ineq} for $k=0$.

		\noindent STEP 2 ({\it proof of (ii)}). To prove \eqref{stabb_ineq} for $k=1$ consider first $\tau\geq 1$ and show that 
		\begin{equation}
		\label{parabolic_regularity}
		\|e^{-\mathcal{T}\tau}\partial_y g\|_{L^2}\leq Ce^{-\omega\tau}\|g\|^2_{L^2}.
		\end{equation}
		
		The idea here is to establish a short time parabolic regularization property. Consider $f:=e^{-\mathcal{T}s}\partial_y g$, it can be represented as  $f=\partial_y v$ with $v$ solving
		\begin{empheq}[left=\empheqlbrace]{align}
		\partial_{s}v&=\partial_y^2 v+V\partial_y v-v -\frac{2\psi}{c_0}\int \theta_0''\theta_0' v dy,
		\label{eq_for_v_star_1}\\
		v(y,0)&=g(y).\label{ic_for_v_star_2}
		\end{empheq}
		In a standard way, multiplying \eqref{eq_for_v_star_1} by $v$ and integrating in $y$ we get 
		\begin{equation}
		\label{soviet_star}
		\frac{1}{2}\frac{d}{ds}\|v\|^2_{L^2}+\|\partial_y v\|_{L^2}^2\leq C \|v\|^2_{L^2}.
		\end{equation}
		Then an application of Gronwall's  inequality yields the uniform bound
		\begin{equation*}
		\|v\|_{L^2}\leq C \|g\|_{L^2}\;\text{ for }\;0\leq s \leq 1.
		\end{equation*}
		Using this bound in \eqref{soviet_star} we derive
		\begin{equation*}
		\int_0^1\|\partial_y v(\,\cdot\,,s)\|^2_{L^2}ds \leq C\|g\|^2_{L^2}.
		\end{equation*} 
		It follows that $\|\partial_y v(\,\cdot\,,s_0)\|_{L^2}\leq C_1\|g\|_{L^2}$ for some $0<s_0\leq 1$.   Then by the semigroup property we have
		\begin{eqnarray}
		\|e^{-\mathcal{T}\tau}\partial_y g\|_{L^2} =\|e^{-\mathcal{T}(\tau-s_0)}\partial_y v(\,\cdot\,,s_0)\|_{L^2}\leq M e^{-\omega(\tau-s_0)}C_1 \|g\|_{L^2}\leq C_2e^{-\omega \tau}\|g\|_{L^2}\quad\quad\text{ for }\tau\geq 1, 
		\end{eqnarray}
		where we have used \eqref{exp_stable_semigroup}. The bound  \eqref{parabolic_regularity} being established, we conclude with the estimate
	\begin{equation}
	\label{first_bound}
	\begin{aligned}
	\Bigl\|\int_0^{\tau-1}e^{-\mathcal{T}(\tau-\tau')}\partial_y g(\,\cdot\,,\tau')d\tau'\Bigr\|^2_{L^2}&\leq 
	\Bigl(C \int_0^{\tau-1} e^{-\omega(\tau-\tau')}\|g(\,\cdot\,\tau')\|_{L^2}d\tau'\Bigr)^2
	\\
	&\leq 
	C_1
	\int_0^{\tau-1} e^{-\omega(\tau-\tau')}\|g(\,\cdot\,\tau')\|_{L^2}^2d\tau'.
		\end{aligned}
		\end{equation}

		
		To complete the proof of \eqref{stabb_ineq} consider
		\begin{equation*}
		\tilde{f}(\,\cdot\,):=\int_{\tau-1}^{\tau}e^{-\mathcal{T}(\tau-\tau')}\partial_y g(\,\cdot\,,\tau')d\tau'=\int_{0}^{1}e^{-\mathcal{T}(1-s)}\partial_y g(\,\cdot\,,\tau-1+s)ds
		\end{equation*}
	(if $\tau<1$, we set $g(y,\tau^\prime)\equiv 0$ for $\tau^\prime<0$).	It follows from the definition of $\tilde{f}$ that $\tilde{f}(y)=\tilde{v}(1,y)$, where $v$ solves 
		\begin{empheq}[left=\empheqlbrace]{align}
		\partial_s\tilde{v}+\mathcal{T}\tilde{v}&=\partial_yg(y,\tau-1+s),\label{eq_for_v_s}\\
		\tilde{v}(0,y)&=0.
		\end{empheq}
		Multiply equation \eqref{eq_for_v_s} by $v$ and integrate in $y$ to obtain 
		\begin{eqnarray*}
			\frac{1}{2} \frac{d}{ds}\|\tilde{v}\|^2_{L^2}+\|\partial_y\tilde{v}\|^2_{L^2}&\leq& -\int g(y,\tau-1+s)\partial_y\tilde{v}(y,s)dy+C\|\tilde{v}\|^2_{L^2}\\
			&\leq&\|\partial_y\tilde{v}\|^2_{L^2}+\frac{1}{4}
			\|g(\,\cdot\,,\tau-1+s)\|^2_{L^2}+C\|\tilde{v}\|^2_{L^2}.
		\end{eqnarray*}
		Now apply Gonwall's inequality. As a result we get
		\begin{equation*}
		\|\tilde{v}(\,\cdot\,,1)\|^2_{L^2}\leq C \int_0^1\|g(\,\cdot\,,\tau-1+s)\|_{L^2}^2 ds.
		\end{equation*}
		Thus
		\begin{eqnarray}
		\Bigl\|\int_{\tau-1}^{\tau}e^{-\mathcal{T}(\tau-\tau')}\partial_y g(\,\cdot\,,\tau')d\tau'\Bigr\|^2_{L^2}
		\leq C \int_0^1\|g(\,\cdot\,,\tau-1+s)\|^2_{L^2}ds
		\leq
		C_1 \int_{\tau-1}^{\tau}e^{-\omega(\tau-\tau')}\|g(\tau')\|^2_{L^2}d\tau'. \label{second_bound}
		\end{eqnarray}
		Combining \eqref{second_bound} with \eqref{first_bound} completes the proof of Lemma \ref{exp_stability}.
	\end{proof}$\square$
	
\begin{figure}[t]
	\begin{center}
		\includegraphics[width=.45\textwidth]{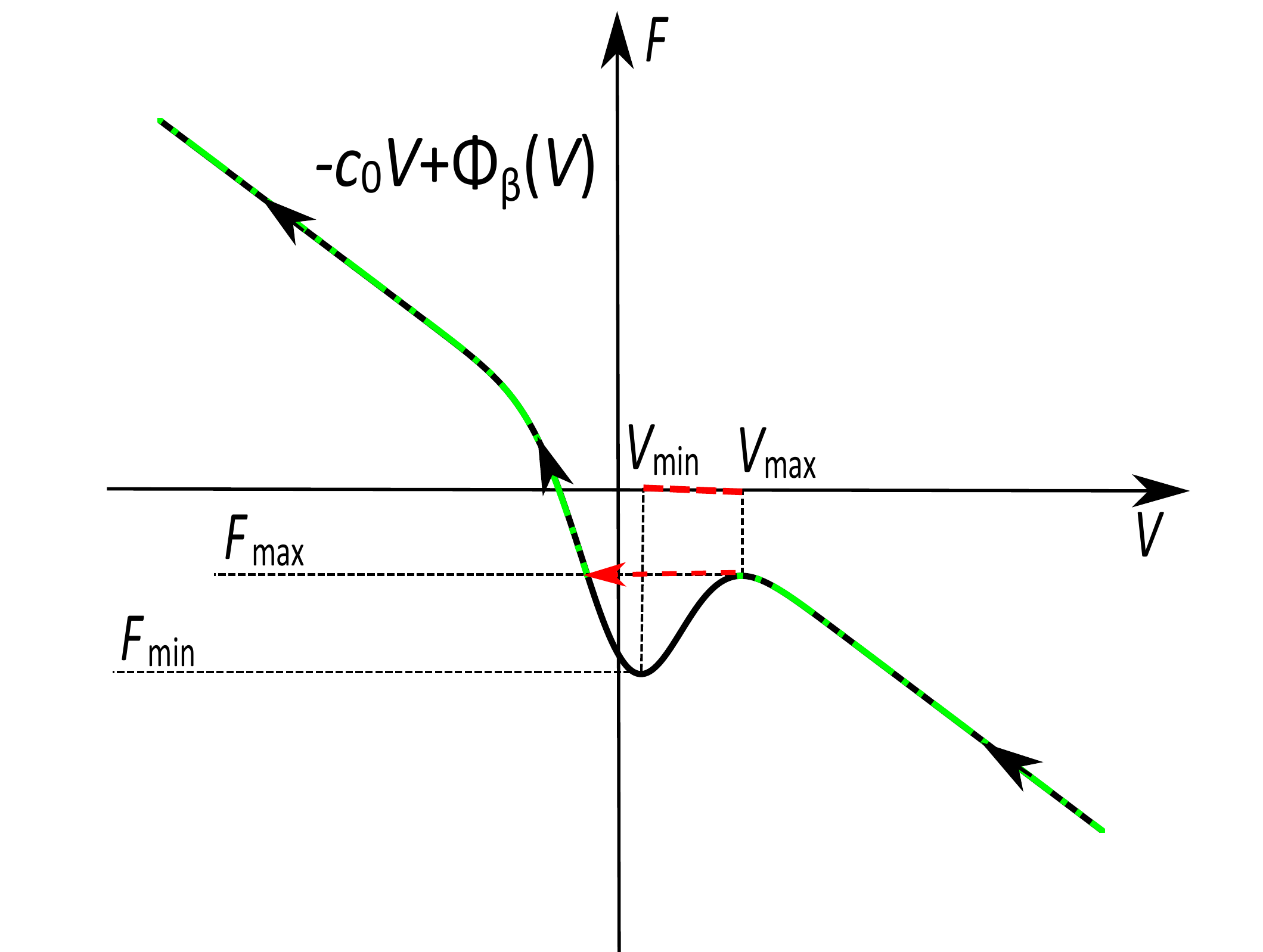}
		\includegraphics[width=.45\textwidth]{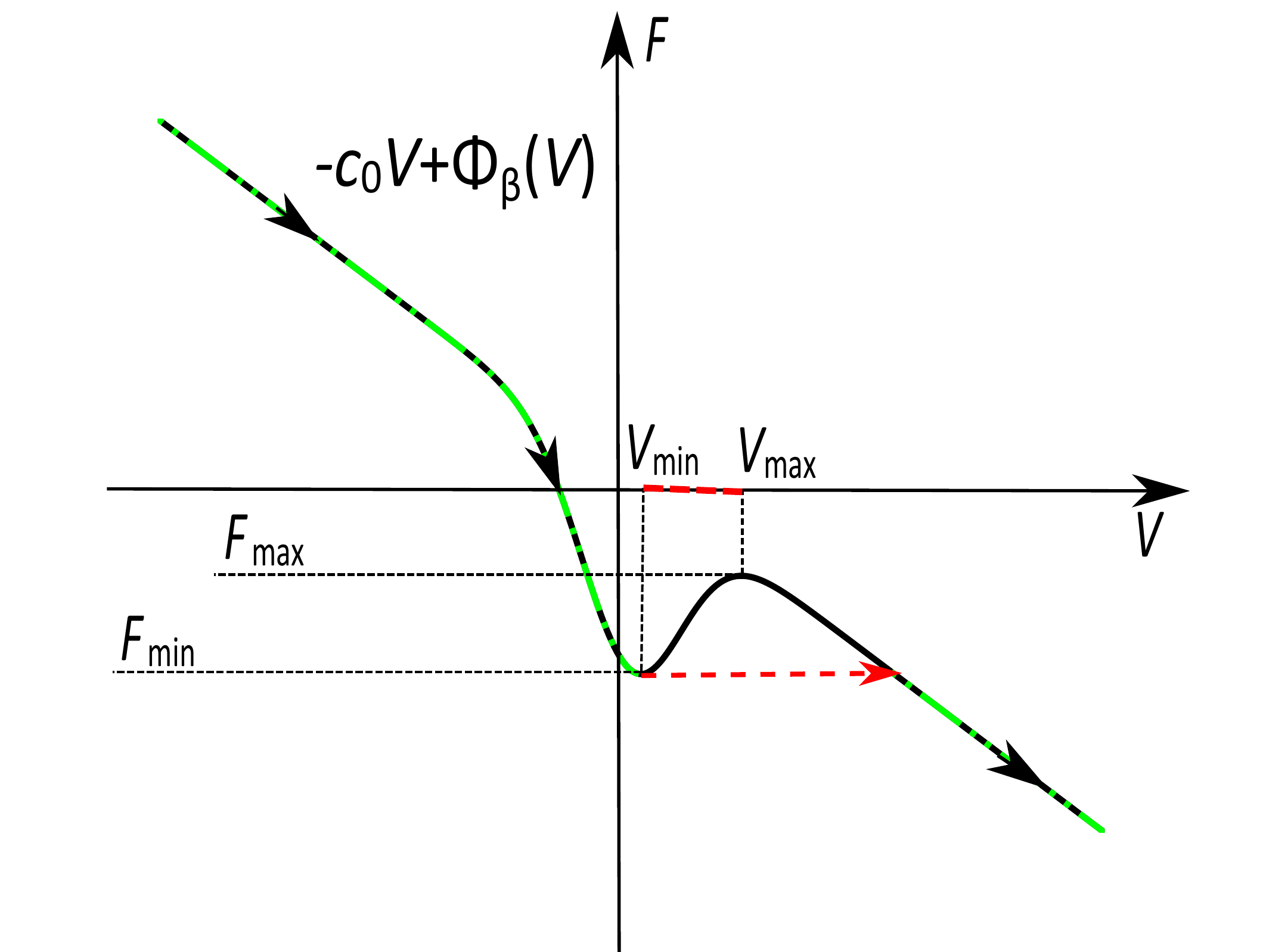}
		\caption{Sketch of the function $F(V)=-c_0V+\Phi_\beta (V)$; $F(V)$ has one local minimum, $F_{\text{min}}=F(V_{\text{min}})$, and one local maximum, $F_{\text{max}}=F(V_{\text{max}})$. {\it Left}: Until $F<F_{\max}$ we stay on the left branch.
			When $F$ exceeds $F_{\max}$ we jump on the right branch; {\it Right}: Until $F>F_{\min}$ we stay on  the right branch;
			When $F$ becomes less than $F_{\min}$ we jump on the left branch. Red arrows on both figures illustrate jumps in velocities.}
		
		\label{fig:two_branches}
	\end{center}
\end{figure}

	\medskip
	
	Now we apply Lemma \ref{exp_stability} to 
	\eqref{semigroup_representation} to obtain the bound 
	\begin{equation}
	\|\eta_\ve(\,\cdot\,,\tau)\|_{L^2}^2 \leq 2M^2 \|\eta_\ve(\,\cdot\,,0)\|^2_{L^2}e^{-\omega \tau}
+\int_0^{\tau}e^{-\omega(\tau-\tau')}\left(
C_\ast\|\eta_\ve(\,\cdot\,,0)\|^4_{L^2}+\delta_\ve\|\eta_\ve(\,\cdot\,,0)\|^2_{L^2}+\delta_\ve\right) d\tau',\label{stabbb_ineq}
	\end{equation}
for $0\leq\tau\leq 1/\ve^r$ ($0<r<1$), where $C_\ast$ depends only on $F(t)$
and $T_{\ast}$ and 
 $\delta_\ve=C(\max_{s\in [t,t+\ve \tau]}\left(|F(t)-F(s)|^2+\tilde o_\ve^2(s)\right) +\ve^2)$ . 
 Let $\alpha(\tau)$ be the right hand side of \eqref{stabbb_ineq}. Consider $s\in [0,\tau]$, by \eqref{stabbb_ineq} we have  
	\begin{equation*}
		\dot{\alpha}(s)= -\omega \alpha(s) +C_{\ast}\|\eta_\ve(\,\cdot\,,0)\|^4_{L^2} +\delta_\ve(s)\|\eta_\ve(\,\cdot\,,0)\|^2_{L^2}+\delta_\ve
		\leq -\omega \alpha(s) +C_{\ast}\alpha^2(s) +\delta_\ve(\tau)\alpha(s)+\delta_\ve(\tau)
	\end{equation*}
and $\alpha(0)=2M^2\|\eta_\ve(\,\cdot\,,0)\|_L^2$. Choosing an arbitrary $q$ from the interval
$$	
q \in(0,\omega/{(2C_\ast)}),
$$	
we see that the function $\overline{\alpha}(s):=q e^{-\omega s/2}+2\delta_\ve(\tau)/\omega$ satisfies for sufficiently small $\ve$ the differential 
inequality
$$
\dot{\overline{\alpha}}(s)+\omega \overline{\alpha}(s) - C_{\ast}\overline{\alpha}^2(s) -\delta_\ve(\tau)\overline{\alpha}(s)-\delta_\ve(\tau)>0.
$$
Therefore, if $\alpha(0)\leq \overline{\alpha}(0)=q+2\delta_\ve(\tau)/\omega$,
then $\alpha(s)\leq \overline{\alpha}(s)$ $\forall$ $0\leq s \leq \tau$. Thus  we have proved that 
$$ 
\|\eta_\ve(\,\cdot\,,\tau)\|_{L^2}^2\leq q e^{-\omega \tau/2}+2\delta_\ve(\tau)/\omega,
$$
provided that $\|\eta_\ve(\,\cdot\,,0)\|_{L^2}\leq \delta$ with $0<\delta<\sqrt{q}/(\sqrt{2}M)$. This concludes the proof of 
Lemma~\ref{l_stab} and Theorem~\ref{theorem:sharp_interface}. 
$\square$

%
%
%

\end{proof}

\medskip

\subsection{Numerical observations. Hysteresis loop.}

In view of the above analysis the equation \eqref{sharp_interface} for large $\beta$ may have many solutions of quite complicated structure (e.g., discontinuous). Therefore, we need to introduce a criterion for selection of the ``correct" solutions that are limiting solutions to the problem with $\ve>0$. This is analogous, {\it e.g.} to viscosity solutions of Allen-Cahn when physical solutions are obtained (by regularization) in the sharp interface limit $\ve\to 0$, \cite{EvaSonSou92}.   

We now introduce such a criterion based on numerical observations and suggested by the stability analysis depicted on Fig. \ref{fig:two_branches}. Define the left velocity interval $\mathcal{B}_{\text{L}}:=(-\infty,V_{\min}]$ and the right velocity interval $\mathcal{B}_{\text{R}}:=[V_{\max},\infty)$ for stable velocities $V$.


\begin{figure}[t]
	\begin{center}
		\includegraphics[width=0.475\textwidth]{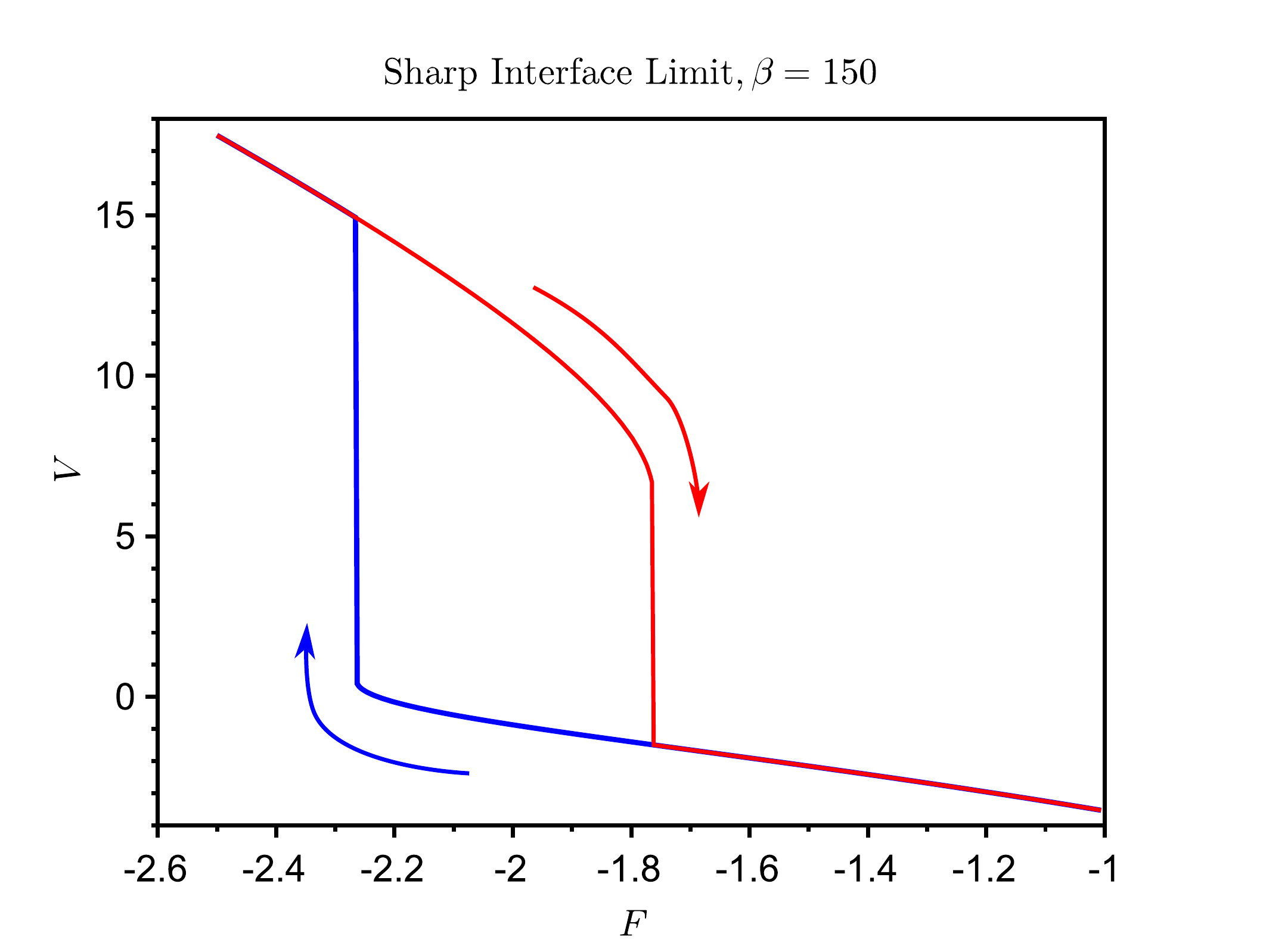}
		\includegraphics[width=0.475\textwidth]{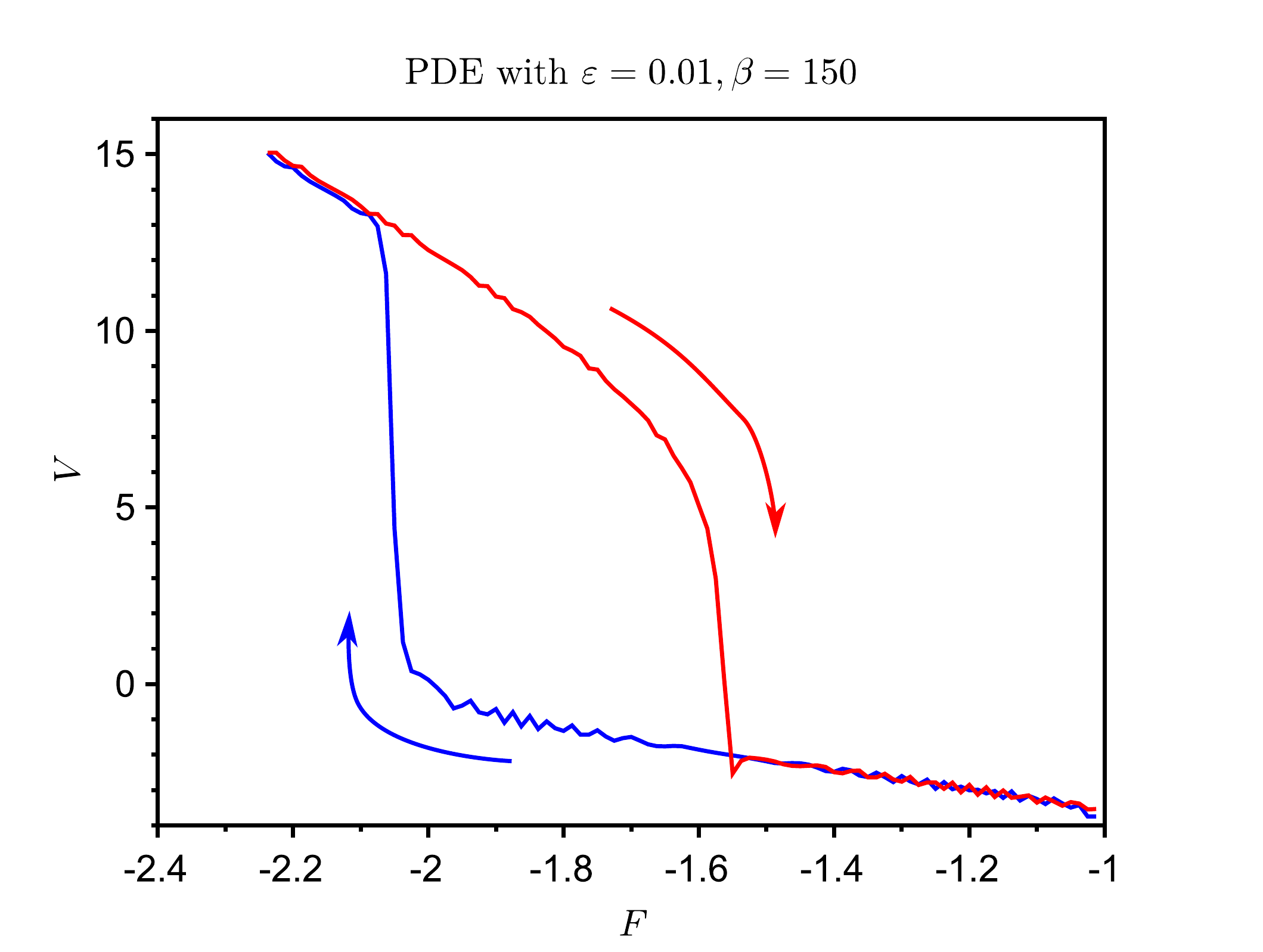}	
		\caption{Hysteresis loop  in the problem of cell motility. Simulations of $V=V(F)$ {\it Left}: \eqref{sharp_interface} 
			Jumping from the left to the right branches and back; 
			{\it Right}: PDE system \eqref{P1}-\eqref{P2}. On both figures arrows show in what direction the system $(V(t),F(t))$ evolves as time $t$ grows; red curve is for $F_{\uparrow}(t)$, blue curve is for $F_{\downarrow}(t)$.}
		\label{fig:hysteresis}
	\end{center}
\end{figure}

Assume for simplicity of presentation that function $F(t)\in C[0,T]$ is strictly increasing. 
Then the solution of \eqref{sharp_interface} is chosen based on the following two criteria

\begin{list}{}{}
	\item {({\bf Cr1})} if $V(0)\in \mathcal{B}_L$, there is a unique $V(t)\in\mathcal{B}_L$ satisfying \eqref{sharp_interface} for all $t\in[0,T]$.
	Note that this $V(t)$ is the only solution which is continuous and never enters the ``forbidden" interval $[V_{\min},V_{\max}]$
	\item {({\bf Cr2})} if $V(0)\in \mathcal{B}_R$, then for any $t\in[0,T]$ the solution $V(t)$ of \eqref{sharp_interface} is chosen in the right velocity interval $\mathcal{B}_R$, unless it is impossible ($F(t)>F_{\min}$, where $F_{\min}$ is defined in Fig. \ref{fig:two_branches}). In the latter case $V(t)$ is chosen from the left velocity interval $\mathcal{B}_L$. 
\end{list}

Intuitively, evolution of the sharp interface velocity can be described as follows. Consider for example the left part of Figure \ref{fig:two_branches} left. As  time evolves, the velocity increases along the right green branch until it reaches $V_{\max}$, then it jumps (along the horizontal red dashed line) to the solution of \eqref{sharp_interface} on the left green branch, and continues increasing along this branch.
%
%
%
%
%
%
%
%
%

Finally, numerical simulations show that the criterion ({\bf Cr2})  predicts {\it hysteresis}  in the system \eqref{sharp_interface}. Consider two forcing terms corresponding to the right and the left parts of Fig. \ref{fig:two_branches}:
\begin{equation*}
F_{\downarrow}(t)=-1.0 +(-2.25+1.0)t,\;\;F_{\uparrow}(t)=-2.25+(-1.0+2.25)t 
\end{equation*} 
and $\beta=150$. For $t\in[0,1]$ both $F_{\downarrow}(t)$ and $F_{\uparrow} (t)$ have the same values but in the opposite order in time $t$. Fig. \ref{fig:hysteresis} (left) depicts the solution of equation \eqref{sharp_interface} according to to the criteria  ({\bf Cr1}) and  ({\bf Cr2}). The red and blue branches coincide when  $F~\notin~[F_{\min},F_{\max}]$. Moreover, a surprising hysteresis loop is observed when $F\in [F_{\min}, F_{\max}]$.


We also performed numerical simulations for the original PDE system \eqref{P1}-\eqref{P2} for $\rho_{\ve}(x,0)=\theta_0(x/\ve)$, $P_{\ve}(x,0)=\theta_0'(x/\ve)$, $\ve=0.01$, and defining $\ve$-interface $x_{\ve}(t)$ as a number such that $\rho_{\ve}(x_{\ve}(t),t)\approx 0.5 (\rho_{\ve}(+\infty,t)+\rho_{\ve}(-\infty,t))$. The branches corresponding to $F_\downarrow$ and $F_{\uparrow}$ 
are depicted on Fig. \ref{fig:hysteresis} (right). The same hysteresis is observed which justifies numerically the above criteria.

\appendix

\section{Auxiliary Inequalities.}
\label{appendix_poincare}

It is well known (see, e.g., \cite{MotSha95}) that under conditions \eqref{condpoten} on the potential $W(\rho)$
the corresponding 
standing wave $\theta_0$ satisfies, for some $\alpha_0>1$,
\begin{equation}\label{exp}
\alpha_0^{-1}e^{-\kappa_- y}<(\theta_0'(y))^2\leq \alpha_0 e^{-\kappa_- y},\;\;\;y\leq 0 \text{ and } \alpha_0^{-1}e^{-\kappa_+ y}<(\theta_0'(y))^2\leq \alpha_0 e^{-\kappa_+ y},\;\;\;y\geq 0,
\end{equation}
where $\kappa_\pm=2\sqrt{W''((1\pm 1)/2)}$. In the case of the symmetric potential  $W(\rho)=\frac14 \rho^2(\rho-1)^2$, $\kappa_-=\kappa_+$ and the standing wave $\theta_0$ is explicitly given
by $\theta_0(y)=\frac{1}{2}(1+\tanh\frac{y}{2\sqrt{2}})$. 

\medskip

\begin{theorem} \label{POINCARE_INEQ}
	({\it Poincar\'{e} inequality}) The following inequality holds
	\begin{equation}\label{poincare}
	\int (\theta_0^\prime)^2 (v-\langle v \rangle)^2dy\leq C_P\int (\theta_0^\prime)^2(v^\prime)^2dy, \;\;\forall\;v\in C^1(\mathbb R),
	\end{equation}
	where
	\begin{equation}
	\langle v \rangle=\frac{1}{\int(\theta_0^\prime)^2dy}\int(\theta_0^\prime)^2vdy.
	\end{equation}
\end{theorem}

\noindent{\it Proof}

{\noindent{STEP 1} ({\it Friedrich's inequality}). Let $u\in C^1(\mathbb R)$ satisfy $u(0)=0$.
	Then we show that  the inequality
	\begin{equation}\label{friedrich}
	\int (\theta_0^\prime)^2u^2dy\leq C_F\int(\theta_0^\prime)^2(u^\prime)^2dy,
	\end{equation}
	holds with  $C_F$ independent of $u$.
	Indeed,
	\begin{equation*}
	\begin{aligned}
	\int_0^{\infty} e^{-\kappa_+ y} u^2 dy &=2 \int_0^{\infty} \left(\int_{y}^{\infty}e^{-\kappa_+ t}dt\right) u^\prime\, u dy
	\leq 2\int_{0}^{\infty}\left(\int_{y}^{\infty}e^{-\kappa_+ t}dt\right) |u^\prime||u|dy\nonumber\\
	&=\frac{2}{\kappa_+}\int_0^{\infty}e^{-\kappa_+ y}|u^\prime||u|dy\leq \frac{2}{\kappa_+}\left(\int_{0}^{\infty} e^{-\kappa_+ y}(u^\prime)^2 dy\right)^{1/2}\left(\int_{0}^{\infty} e^{-\kappa_+ y}u^2 dy\right)^{1/2}.
	\end{aligned}
	\end{equation*}
	Thus,
	\begin{equation}
	\int_0^{\infty} (\theta_0^\prime)^2 u^2 dy\leq \frac{2\alpha_0^2}{\kappa_+}\left(\int_{0}^{\infty} (\theta_0^\prime)^2(u^\prime)^2 dy\right)^{1/2}\left(\int_{0}^{\infty} (\theta_0^\prime)^2u^2 dy\right)^{1/2}.\nonumber
	\end{equation}
	Dividing this inequality by $\left(\int_{0}^{\infty} (\theta_0^\prime)^2u^2 dy\right)^{1/2}$, and than taking square 
	of both sides we get
	\begin{equation}\label{Ap_eq_1}
	\int_{0}^{\infty}(\theta_0^\prime)^2u^2dy \leq \frac{4\alpha_0^4}{\kappa_+^2}\int_0^{\infty}(\theta_0^\prime)^2(u^\prime)^2 dy  
	\end{equation}
	Similarly  we obtain
	\begin{equation}\label{Ap_eq_2}
	\int_{-\infty}^{0}(\theta_0^\prime)^2u^2dy \leq \frac{c_0^4}{\kappa_-^2}\int_{-\infty}^{0}(\theta_0^\prime)^2(u^\prime)^2 dy,
	\end{equation}
	Then adding \eqref{Ap_eq_1} to \eqref{Ap_eq_1} yields  \eqref{friedrich}. \\
	\noindent{STEP 2.} We prove the Poincar\'{e} inequality \eqref{poincare} by contradiction. Namely, assume that there exists a sequence $v_n\in C^1(\mathbb R)\cap L^{\infty}(\mathbb R)$ such that
	\begin{equation}\nonumber
	\int (\theta_0^\prime)^2v^2_ndy =1, \;\;\int(\theta_0^\prime)^2 v_n dy=0\text{ and }\int (\theta_0^\prime)^2 (v_n^\prime)^2dy\rightarrow 0.
	\end{equation}
	Apply Friedrich's inequality \eqref{friedrich} to functions $v_n(y)-v_n(0)$:
	\begin{equation}\nonumber
	\int (\theta_0^\prime)^2(v_n(y)-v_n(0))^2 dy \leq C_F \int (\theta_0^\prime)^2 (v_n^\prime)^2 dy\rightarrow 0.
	\end{equation}
	On the other hand,
	\begin{equation}\nonumber
	\int (\theta_0^\prime)^2(v_n(y)-v_n(0))^2 dy=\int (\theta_0^\prime)^2 v_n^2dy +v_n^2(0)\int(\theta_0^\prime)^2 dy\geq \int (\theta_0^\prime)^2 v_n^2dy.
	\end{equation}
	Hence,
	\begin{equation}\nonumber
	\int (\theta_0^\prime)^2 v_n^2 dy \rightarrow 0
	\end{equation}
	which contradicts the normalization $\int (\theta_0^\prime)^2 v_n^2 dy=1$. The Theorem is proved.
	{$\square$}
	
\begin{corollary} 
		\label{corAp}
		Let $u\in H^1(\mathbb{R})$, then 
		\begin{equation}
		\label{FinalTechIneqA}
		\|u-\langle u \rangle_{\theta_0^\prime}\, \theta_0^\prime\|_{H^1}^2\leq C\int (\theta_0^\prime)^2(v^\prime)^2dy, \ \text{where}\  
		\langle u \rangle_{\theta_0^\prime}=\frac{1}{\int(\theta_0^\prime)^2dy}\int u \theta_0^\prime dy
		\ \text{and}\ v=u/\theta_0^\prime, 
		\end{equation}
with a constant $C$ independent of $u$.
	\end{corollary}
	
\begin{proof}	
	Recall that standing waves $\theta_0$ of the Allen-Cahn equation along with  \eqref{exp} satisfy 
	\begin{equation*}
	\alpha_1^{-1}e^{-\kappa_- y}<(\theta_0^{\prime\prime}(y))^2\leq \alpha_1 e^{-\kappa_- y},\;\;\;y\leq 0 \text{ and } \alpha_1^{-1}e^{-\kappa_+ y}<(\theta_0^{\prime\prime}(y))^2\leq \alpha_1 e^{-\kappa_+ y},\;\;\;y\geq 0,
	\end{equation*}
	for some $\alpha_1>0$. Then applying Theorem \ref{POINCARE_INEQ} to $v=u/\theta_0^\prime$
	and using density of $C^1(\mathbb{R})$ in $H^1(\mathbb{R})$ one derives \eqref{FinalTechIneqA}.
\end{proof} {$\square$}

\section{On spectral properties of operator $\mathcal{T}$ in the case $W(\rho)=\frac{1}{4}\rho^2(\rho-1)^2$}
\label{appendix_integral_equation}

 In this appendix we study the set of stable of velocities $\mathcal{S}$, i.e., the set of such $V\in \mathbb R$ that the point spectrum of the linearized operator $\mathcal{T}(V)$ defined by \eqref{def_of_T} lies in the right half of the complex plane. We restrict ourselves here to the case $W(\rho)=\frac14 \rho^2(\rho-1)^2$.
 
   Theorem \ref{spectral_theorem} implies that if $\text{Re}\lambda\leq 0$, then $\lambda$ solves the equation \eqref{point_spectrum}. Though \eqref{point_spectrum} is a scalar equation with respect to  $\lambda\in \mathbb C$, the evaluation of its left hand side requires solution of the PDE \eqref{psi_in_formal}. By means of  Fourier analysis, we can avoid solving the PDE and rewrite \eqref{point_spectrum} in the form  
\begin{equation}\label{spectrum_main}
\int_{\mathbb R} \frac{-i\beta k \tilde{\theta_0'}\overline{\widetilde{(\theta_0')^2}} }{(k^2-iVk+1)(k^2-iVk+(1-\lambda))}dk=1,
\end{equation}
where $\tilde{\theta_0'}$ and ${\widetilde{(\theta_0')^2}}$ are Fourier transforms of $\theta_0'$ and $(\theta_0')^2$, respectively. In the case $W(\rho)=\frac14 \rho^2(1-\rho)^2$:
\begin{equation}\label{fourier_transform}
\tilde{\theta_0'}(k):=\sqrt{\pi}\text{csch} (\sqrt{2}\pi k),\;\;\;\widetilde{(\theta_0')^2}(k)=\frac{\sqrt{2\pi}}{12}k(2k^2+1)\text{csch}(\sqrt{2}\pi k).
\end{equation}
Introduce $\chi(k):=-\frac{\beta\pi \sqrt{2}}{12} k^2 (2k^2+1)\text{csch}^2 (\sqrt{2}\pi k)$, then equation \eqref{spectrum_main} becomes
\begin{equation}\label{spectrum_main_with_chi}
\int_{\mathbb R} \frac{i \chi(k)}{(k^2-iVk+1)(k^2-iVk+(1-\lambda))}dk=1.
\end{equation}

Next, consider $\lambda=\lambda_r+i\lambda_i$. Denote by $\mathcal{H}_\lambda(k)$ the integrand in \eqref{spectrum_main_with_chi}
and rewrite it in the form   
\begin{eqnarray*}
	\mathcal{H}_{\lambda_r+i\lambda_i}(k)&=&-\chi(k)\frac{\left[Vk(k^2+\mu)+(k^2+1)(Vk+\lambda_i)\right]}{\left((k^2+1)^2+V^2k^2\right)\left((k^2+\mu)^2+(Vk+\lambda_i)^2)\right)}\\
	&&+i\chi(k)\frac{\left[(k^2+1)(k^2+\mu)-Vk(Vk+\lambda_i)\right]}{\left((k^2+1)^2+V^2k^2\right)\left((k^2+\mu)^2+(Vk+\lambda_i)^2)\right)},
\end{eqnarray*}
where $\mu=1-\lambda_r$.

\begin{proposition}\label{prop:stability}
	\noindent(i) Assume $V<\sqrt{2}$. If $\Phi'_\beta(V)<c_0$, then all eigenvalues $\lambda \in \sigma_p(\mathcal{T}(V))$ have positive real part, $\text{Re}\lambda>0$. 
	
	\noindent (ii) There exists $\bar{V}>0$ such that for all $V>\bar{V}$ all eigenvalues of $\mathcal{T}(V)$ have positive real part. 
\end{proposition}

\begin{remark}
	Condition $V<\sqrt{2}$ is a technical assumption in the proof which guarantees that integral \eqref{integral_negative} is negative.  However, numerical simulations show that  integral \eqref{integral_negative} is negative for all $V$.  
\end{remark}

\proof $\;$\\
Part {\it(i)}. First, assume $0<|V|<\sqrt{2}$. We prove that if $\lambda=\lambda_r+i\lambda_i$ with $\lambda_r<1$ ($\mu>0$) is a root of equation $\zeta(\lambda)=1$, then $\lambda_i=0$. In particular, the condition $\lambda_r<1$ guarantees that $\lambda\notin \sigma_{\text{ess}}(\mathcal{A}(V))$.

Rewrite the imaginary part of $\zeta(\lambda_r+i\lambda_i)$:
\begin{eqnarray*}
	\text{Im}\zeta(\lambda)&=&\int\limits_{-\infty}^{\infty}\mathcal{H}_\lambda(k)dk\\
	&=&\lambda_i V \int\limits_{0}^{\infty} \frac{\chi(k)(-2(k^2+1)(k^2+\mu)+V^2k^2-(k^2+\mu)^2-\lambda_i^2)}{((k^2+1)^2+V^2k^2)((k^2+\mu)^2+(Vk+\lambda_i)^2)((k^2+\mu)^2+(Vk-\lambda_i)^2)}dk.
\end{eqnarray*}
Since the numerator is the difference between
$(V^2-2)k^2 $
and a positive expression, 
we obtain $\text{Im}\zeta(\lambda)\neq 0$ for $\lambda_i \neq 0$. 

\medskip 

Take $\lambda_i=0$ and rewrite the real part of $\zeta(1-\mu)$: 
\begin{equation}\label{prop_C2_step2}
\text{Re}\zeta(1-\mu)= - V \int\limits_{-\infty}^{\infty} k\chi(k)\frac{2k^2+1+\mu}{((k^2+1)^2+V^2k^2)((k^2+\mu)^2+V^2k^2)}dk.
\end{equation}
The function $\text{Re}\zeta(1-\mu)$ is obviously monotone for $\mu>0$. Indeed, denote by $\Psi_k(\mu)$ the term of integrand in \eqref{prop_C2_step2} which depends on $\mu$:
\begin{equation*}
\Psi_k(\mu)=\frac{2k^2+1+\mu}{((k^2+\mu)^2+V^2k^2)}.
\end{equation*}
Compute $\Psi'_k(\mu)$:
\begin{equation}
\Psi'_k(\mu)=\frac{ (V^2-2-4\mu)k^2-k^4-2\mu-\mu^2}{((k^2+\mu)^2+V^2k^2)^2}.
\end{equation}
If $|V|<\sqrt{2}$, then $\Psi'_{k}(\mu)<0$, which proves the monotonicity of  $\text{Re}\zeta(1-\mu)$.

Finally, assume by contradiction that $\beta c_0^{-1}\Phi(V)<1$, but there exists an eigenvalue $\lambda_0$ with zero or negative real part, $\text{Re}\lambda_0\leq 0$. Then  $\zeta(\text{Re}\lambda_0)=\zeta(\lambda_0)\leq\zeta(0)<1$ that contradicts $\zeta(\lambda)=1$. 

 Consider $V\leq 0$.  Then $\text{Re}\zeta(\lambda_0)\leq 0$. Indeed, observe that  
\begin{equation}\label{integral_negative}
\text{Re}\zeta(\lambda_0)=\int\limits_{0}^{\infty} \frac{-4Vk\chi(k)\left[(2k^2+1+\mu)((k^2+\mu)^2+V^2k^2)\right]}{((k^2+1)^2+V^2k^2)\left((k^2+1)^2+(Vk-\lambda_i)^2\right)\left((k^2+\mu)^2+(Vk+\lambda_i)^2)\right)}dk.
\end{equation} 
The integral in \eqref{integral_negative} is negative or zero and, thus, cannot be equal to $1$, so equality \eqref{point_spectrum} does not hold and, in particular, there does not exist eigenvalues with negative real part. 
Thus, part (i) is proved.

 Part {\it (ii)} follows immediately from \eqref{integral_negative}.

$\square$

\section*{References}

\bibliography{cell}

\end{document}